\documentclass[11pt, letterpaper]{amsart}
\usepackage{amssymb,euscript,tikz,units}
\usepackage[colorlinks,citecolor=blue,linkcolor=red]{hyperref}
\usepackage{verbatim}
\usepackage[nameinlink]{cleveref}
\usepackage{enumitem}
\usepackage{hyperref}
\usepackage{url}
\usepackage{appendix}

\newcommand{\calT}{\mathcal{T}}
\newcommand{\T}{\calT}
\newcommand{\M}{\mathcal{M}}
\newcommand{\sM}{\mathcal{M}}
\newcommand{\sO}{\mathcal{O}}
\newcommand{\sA}{\mathcal{A}}
\newcommand{\sB}{\mathcal{B}}

\newcommand{\sP}{\mathcal{P}}
\newcommand{\R}{\mathbb{R}}
\newcommand{\Z}{\mathbb{Z}}
\newcommand{\C}{\mathbb{C}}

\newcommand{\eps}{\epsilon}

\newcommand{\vph}{\varphi}
\newcommand{\sbs}{\subset}
\newcommand{\limg}{\lim_{g\rightarrow\infty}}

\newcommand{\E}{\mathbb{E}_{\rm WP}^g}

\newcommand{\eg}{\textit{e.g.\@ }}

\def\sys{\mathop{\rm sys}}
\def\area{\mathop{\rm Area}}
\def\arcsinh{\mathop{\rm arcsinh}}

\def\Vol{\mathop{\rm Vol}}

\def\dist{\mathop{\rm dist}}

\def\Prob{\mathop{\rm Prob}\nolimits_{\rm WP}^g}
\def\Mod{\mathop{\rm Mod}}

\def\Li{\mathop{\rm Li}}

\theoremstyle{plain}
\newtheorem{theorem}{Theorem}
\newtheorem{corollary}[theorem]{Corollary}
\newtheorem{proposition}[theorem]{Proposition}
\newtheorem{lemma}[theorem]{Lemma}

\newtheorem{remark}[theorem]{Remark}

\theoremstyle{definition}

\newtheorem{definition}[theorem]{Definition}

\theoremstyle{remark}

\newtheorem*{rem*}{Remark}
\newtheorem*{def*}{Definition}
\newtheorem*{con*}{Construction}
\newtheorem*{thm*}{\bf Theorem}
\newtheorem*{definition*}{Definition}
\newtheorem*{assum*}{Assumption $(\star)$}
\newtheorem*{obs*}{Observation}
\newtheorem*{conj*}{\bf Conjecture}
\newtheorem*{ques*}{\bf Question}

\newcommand{\be}{\begin{equation}}
\newcommand{\ene}{\end{equation}}
\newcommand{\br}{\begin{remark}}
\newcommand{\er}{\end{remark}}
\newcommand{\bl}{\begin{lemma}}
\newcommand{\el}{\end{lemma}}
\newcommand{\bcor}{\begin{cor}}
\newcommand{\ecor}{\end{cor}}
\newcommand{\bpro}{\begin{pro}}
\newcommand{\epro}{\end{pro}}
\newcommand{\ben}{\begin{enumerate}}
\newcommand{\een}{\end{enumerate}}
\newcommand{\bp}{\begin{proof}}
\newcommand{\ep}{\end{proof}}
\newcommand{\bpo}{\begin{pro}}
\newcommand{\epo}{\end{pro}}
\newcommand{\beq}{\begin{equation*}}
\newcommand{\eeq}{\end{equation*}}
\newcommand{\bear}{\begin{eqnarray}}
\newcommand{\eear}{\end{eqnarray}}
\newcommand{\beqar}{\begin{eqnarray*}}
\newcommand{\eeqar}{\end{eqnarray*}}
\newcommand{\bt}{\begin{theorem}}
\newcommand{\et}{\end{theorem}}
\newcommand{\bex}{\begin{excer}}
\newcommand{\eex}{\end{excer}}

\theoremstyle{definition}

\makeatletter
\@namedef{subjclassname@2020}{\textup{2020} Mathematics Subject Classification}
\makeatother

\subjclass[2020]{}

\begin{document}

\title[PGT for large genus]{Prime geodesic theorem and closed geodesics for large genus}
\author{Yunhui Wu and Yuhao Xue}

\address{Tsinghua University, Beijing, China}
\email[(Y.~W.)]{yunhui\_wu@tsinghua.edu.cn}

\address{Tsinghua University, Beijing, China;   (Current) Institut des Hautes \'Etudes Scientifiques, Bures-sur-Yvette, France}

\email[(Y.~X.)]{xueyh@ihes.fr}

\date{}
\maketitle

\begin{abstract}
Let $\sM_g$ be the moduli space of hyperbolic surfaces of genus $g$ endowed with the Weil-Petersson metric. In this paper, we show that for any $\eps>0$, as $g\to \infty$, for a generic surface in $\sM_g$, the error term in the Prime Geodesic Theorem is bounded from above by $g\cdot  t^{\frac{3}{4}+\eps}$, up to a uniform constant multiplication. The expected value of  the error term in the Prime Geodesic Theorem over $\sM_g$ is also studied. As an application, we show that as $g\to \infty$, on a generic hyperbolic surface in $\sM_g$ most closed geodesics of length significantly less than $\sqrt{g}$ are simple and non-separating, and most closed geodesics of length significantly greater than $\sqrt{g}$ are not simple, which confirms a conjecture of  Lipnowski-Wright. 

A novel effective upper bound for intersection numbers on $\sM_{g,n}$ is also established, when certain indices are large compared to $\sqrt{g+n}$.
\end{abstract}

\maketitle

\section{Introduction}
Let $X_g$ be a closed hyperbolic surface of genus $g$ $(g\geq 2)$.  The spectrum of Laplacian operator on $X_g$ is a discrete closed subset in $\R^{\geq 0}$ and consists of eigenvalues with finite multiplicity. We enumerate them, counted with multiplicity, in the following increasing order
\[0=\lambda_0(X_g)<\lambda_1(X_g)\leq \lambda_2(X_g) \leq \cdots \to \infty.\]

\noindent For any $t>1$, we denote by $\pi_{X_g}(t)$ the number of oriented primitive closed geodesics of length $\leq \ln t$. A remarkable application of Selberg's trace formula \cite{Selb56} is the Prime Geodesic Theorem (PGT) (see \eg \cite{Huber61, Hej76, Rand77}) which states
\begin{equation}\label{e-pgt-0}
\pi_{X_g}(t)=\Li(t)+\sum \limits_{\frac{3}{4}<s_j<1}\Li(t^{s_j})+\mathrm{O}_{X_g}\left(\frac{t^{\frac{3}{4}}}{\ln t}\right) 
\end{equation}
where $\Li(t) = \int_2^t \frac{dx}{\ln x} \sim \frac{t}{\ln t}\ \text{as}\ t\to \infty$, $s_j=\frac{1}{2}+\sqrt{\frac{1}{4}-\lambda_j(X_g)}$ and the implied constant depends on the geometry of $X_g$. As an analog of the Riemann hypothesis for closed hyperbolic surfaces, it is a well-known conjecture that \emph{the $\frac{3}{4}$ in \eqref{e-pgt-0} can be replaced by $\frac{1}{2}+\epsilon$} (see \eg \cite[Page 190]{Ber16}). Rewrite \eqref{e-pgt-0} as 
\begin{equation}\label{e-pgt}
\pi_{X_g}(t)=\Li(t)+\mathrm{Er}_{X_g}(t).
\end{equation}
So the error term $\mathrm{Er}_{X_g}(t)$ depends on the geometry and first non-zero eigenvalue of $X_g$. There are many works on the study of the error term $\mathrm{Er}_{X_g}(t)$, especially for special arithmetic surfaces. See \cite{Sar80, Iwa84, LRS95, Cai02, SY13} and the references therein.

Let $\mathcal{M}_g$ be the moduli space of closed hyperbolic surfaces of genus $g$ endowed with the Weil-Petersson metric. Based on her celebrated thesis works \cite{Mirz07,Mirz07-int}, Mirzakhani in \cite{Mirz13} initiated the study of random hyperbolic surfaces for large genus. In this paper, our aim is to study the average behavior of the error terms in \eqref{e-pgt-0} and \eqref{e-pgt} on $\sM_g$ for large $g$. More precisely, let $\Prob$ and $\E$ be the probability measure and expectation on $\sM_g$ respectively given by the Weil-Petersson metric.  We apply recent results of \cite{Monk20} and \cite{WX22-GAFA, LW21} to show that as $g\to \infty$, for a generic $X_g\in \sM_g$, the error term $\mathrm{Er}_{X_g}(t)$ in \eqref{e-pgt} can be bounded from above by $g \cdot \frac{t^{\frac{3}{4}+\eps}}{\ln t}$ up to a universal constant multiplication. More precisely,
\bt\label{mt-1}
There exists a universal constant $c>0$ such that for any $\epsilon>0$ and all $t=t(g)>2$ which may depend on $g$,
$$\lim \limits_{g\to \infty}\Prob\left(X_g\in \sM_g; \  \left|\pi_{X_g}(t)- \Li(t)\right|\leq c \cdot g \cdot t^{\frac{3}{4}+\eps}\right)=1.$$
\et
\noindent Selberg's theory is important in the proof of Theorem \ref{mt-1}. One may see \cite{GMST21, MS20, Monk20, Thomas22, WX22-GAFA, LW21, Hide21, Ru22} and the references therein for related topics on the same spirit. As an analog of the Riemann hypothesis for Weil-Petersson random hyperbolic surfaces, we conjecture
\begin{conj*}
Theorem \ref{mt-1} still holds with replacing $\frac{3}{4}$ by $\frac{1}{2}$.
\end{conj*}
\

We also study the expected value $\E\left[\left|\pi_{X_g}(t)- \Li(t) \right|\right]$ for large genus. In light of \eqref{e-pgt-0}, we need to study the term $\sum \limits_{\frac{3}{4}<s_j<1}\Li(t^{s_j})+\mathrm{O}_{X_g}\left(\frac{t^{\frac{3}{4}}}{\ln t}\right)$ for large genus. Motivated by recent work \cite{BP22} of Bourque-Petri, we first prove the following result in which the error term $\mathrm{O}_{X_g}\left(\frac{t^{\frac{3}{4}}}{\ln t}\right) $ in \eqref{e-pgt-0} for a single surface $X_g\in \sM_g$ is explicitly given in terms of the geometry of $X_g$. More precisely,
\begin{theorem}\label{thm pgt5/6}
	For any  $X_g\in \sM_g$ and all $t=t(g)>2$ which may depend on $g$, 
\beqar
\pi_{X_g}(t)&=& \Li(t)+\sum_{0<\lambda_j(X_g)\leq \frac{1}{4}}\Li(t^{s_j}) \\
&+& O\left(g\frac{t^\frac{5}{6}}{\ln t}+ g\frac{t^\frac{2}{3}}{\ln t}\max\left\{0,\ln\left(\frac{1}{\sys(X_g)t^{\frac{1}{6}}}\right)\right\} \right)
\eeqar
	where $\sys(X_g)$ is the length of shortest closed geodesic in $X_g$, and the implied constant is universal.
\end{theorem} 

With the help of our previous work \cite{WX22-GAFA}, we will get an effective upper bound of $\E\left[\sum \limits_{0<\lambda_j(X_g)\leq\frac{1}{4}} \Li(t^{s_j})\right]$ for large genus. Then we apply Theorem \ref{thm pgt5/6} to prove the following bound on $\E\left[\pi_{X_g}(t)\right]$ for large $g$ and suitable $t=t(g)$. More precisely,

\begin{theorem}\label{thm pgt E}
Let $t(g)$ be a function such that $t(g)>g^{12}$. Then for any $0<\eps<1$ and large enough $g$, 
$$\E\left[\ \left|1-\frac{\pi_{X_g}(t(g))}{\Li(t(g))}\right|\ \right]=O_\eps(g^{-1+\eps})$$
where the implied constant only depends on $\epsilon$.
\end{theorem} 
\

For any $L>0$, denote by $N(X_g,L)$ the number of unoriented primitive closed geodesics in $X_g$ of length $\leq L$. So $N(X_g,L)=\frac{\pi_{X_g}(e^L)}{2}.$ As introduced above, by the Prime Geodesic Theorem 
\[N(X_g,L)\sim \frac{1}{2}\frac{e^L}{L} \ \textit{as $L \to \infty$}.\]
Denote by $N^s(X_g,L)$, $N^s_{nsep}(X_g,L)$, $N^s_{sep}(X_g,L)$ and $N^{ns}(X_g,L)$ the number of unoriented primitive \emph{simple}, \emph{simple and non-separating}, \emph{simple and separating}, and \emph{non-simple} closed geodesics in $X_g$ of length $\leq L$ respectively. It is clear that $N(X_g,L)=N^s_{nsep}(X_g,L)+N^s_{sep}(X_g,L)+N^{ns}(X_g,L)$ and $N^s(X_g,L)=N^s_{nsep}(X_g,L)+N^s_{sep}(X_g,L)$. A celebrated result of Mirzakhani \cite{Mirz08} states that 
\[N^s(X_g,L)\sim C(X_g)\cdot L^{6g-6} \ \textit{as $L \to \infty$}\]
where $C(X_g)$ is a constant depending on $X_g$, and as a function on $\sM_g$, $C:\sM_g \to \R$ is proper. 

Based on Mirzakhani's Integration Formula \cite{Mirz07}, it is not hard to see that as $g\to \infty$, the expected values of $N^s(X_g,L)$ and $N^s_{nsep}(X_g,L)$ over $\sM_g$ satisfy that
\[\lim \limits_{g\to \infty}\E[N^s(X_g,L)]\sim \lim\limits_{g\to \infty}\E[N^s_{nsep}(X_g,L)]\sim \frac{1}{2}\frac{e^L}{L} \ \textit{as $L\to \infty$}. \] 
It is quite mysterious that both $N(X_g,L)$ and $\lim\limits_{g\to \infty}\E[N^s_{nsep}(X_g,L)]$ have the same asymptotic leading term $\frac{1}{2}\frac{e^L}{L}$ as $L\to \infty$. This suggests that for suitable $L=L(g)$ depending on $g$, as $g\to \infty$, in the set of closed geodesics of length $\leq L$ in a generic surface  in $\sM_g$, simple and non-separating ones are dominant. Meanwhile, as introduced above, $N^s(X_g,L)\sim C(X_g)\cdot L^{6g-6}$. This also seems to suggest that for suitable $L=L(g)$ depending on $g$, as $g\to \infty$, in the set of closed geodesics of length $\geq L$ in a generic surface in $\sM_g$, non-simple ones are major part. Actually it was conjectured in \cite[Conjecture $1.2$]{LW21} that
\begin{conj*}[Lipnowski-Wright]
As $g\to \infty$, on most closed hyperbolic surfaces in $\sM_g$ most geodesics of length significantly less than $\sqrt{g}$ are simple and non-separating,  and most geodesics of length significantly greater than $\sqrt{g}$ are not simple.
\end{conj*}

 It is known by \cite[Theorem 4.4]{Mirz13} that for any $\eps>0$,
\[\lim \limits_{g\to \infty}\Prob\left(X_g\in\M_g;\ N_{sep}^s(X_g,(2-\eps)\ln g)\geq 1\right)=0.\]
And it is also known by \cite[Theorem 4]{NWX20} that for any $\eps>0$,
\[\lim \limits_{g\to \infty}\Prob\left(X_g\in\M_g;\ N^{ns}(X_g,(1-\eps)\ln g)\geq 1\right)=0.\]
The two limits above imply that if $L(g)\leq (1-\eps_0)\ln g$ for some $\eps_0>0$, then
\[\lim \limits_{g\to \infty}\Prob\left(X_g\in\M_g;\ N(X_g,L(g))=N_{nsep}^s(X_g,L(g))\right)=1.\]
This in particular tells that the first part of the conjecture of Lipnowski-Wright above is true when the length is less than $(1-\eps_0)\ln g$. In this paper, we completely solve the conjecture above. More precisely, we prove 
\begin{theorem}\label{mt-geodesic} 
The following two probabilities hold:
\ben
\item
if $L(g)$ satisfies that for some fixed $\eps_0>0$, $$L(g)\geq (1-\eps_0)\ln g \text{ and }\lim \limits_{g\to \infty}\frac{L^2(g)}{g}=0,$$ then there exists a function $\delta(g)>0$ satisfying $\lim \limits_{g\to \infty}\delta(g)=0$ such that
\[\lim \limits_{g\to \infty}\Prob\left(X_g\in\M_g;\ \left|1-\frac{N_{nsep}^s(X_g,L(g))}{N(X_g,L(g))} \right|<\delta(g)\right)=1.\]

\item If $L(g)$ satisfies $$\lim \limits_{g\to \infty}\frac{L^2(g)}{g}=\infty,$$ then
\[\lim \limits_{g\to \infty}\Prob\left(X_g\in\M_g;\ \left|1-\frac{N^{ns}(X_g,L(g))}{N(X_g,L(g))} \right|<\frac{g}{L(g)^2}\right)=1.\]
\een
\end{theorem}

\noindent The proof of Part $(1)$ of Theorem \ref{mt-geodesic} above is divided into two parts: when $L(g)> 12 \ln g$, we use Theorem \ref{mt-1} and \ref{thm pgt E} which are obtained by spectral methods; when $L(g)\leq 12 \ln g$, we apply geometric methods developed in our previous works \cite{NWX20, WX22-GAFA}.

\begin{rem*}
It is known by \cite{BS94} that $\sup_{X_{g_k}\in \sM_{g_k}}\sys(X_{g_k})\geq \frac{4}{3}\ln g_k-C$ for a universal constant $C>0$ and some sequence $\{g_k\}$ with $g_k\to \infty$ as $k\to \infty$. So for Part $(1)$ of Theorem \ref{mt-geodesic} above, $N(X_g, L(g))$ may vanish for certain $X_g\in \sM_g$. However, it follows by \cite[Theorem 5.1]{MP19} that for any $L(g)$ with $\lim \limits_{g\to \infty}L(g)=\infty$,
\[\lim \limits_{g\to \infty}\Prob\left(X_g\in\M_g;\ N(X_g, L(g))\geq 1)\right)=1.\]
Thus, for Part $(1)$ of Theorem \ref{mt-geodesic} above, it suffices to consider $X_g\in \sM_g$ with $N(X_g,(1-\eps_0)\ln g)\geq 1$.
\end{rem*}

Recall that a closed geodesic $\gamma \subset X_g \in \sM_g$ is called \emph{filling} if each component of the complement $X_g\setminus \gamma$ of $\gamma$ in $X_g$ is homeomorphic to a disk. It is known by elementary isoperimetric inequality that the length of each filling closed geodesic in $X_g$ grows at least like $2\pi(g-1)$. We propose the following interesting question:
\begin{ques*}
As $g\to \infty$, on a generic closed hyperbolic surface in $\sM_g$, are most closed geodesics of length significantly greater than $g$ filling?
\end{ques*}

\begin{rem*}
Very recently, the question above is affirmative answered by Dozier-Sapir in \cite{DS23} when \emph{the length grows significantly greater than $g\ln^2g$}.
\end{rem*}

In the proof of Part $(2)$ of Theorem \ref{mt-geodesic}, we establish a new bound on the Weil-Petersson volume $V_{g,n}(x_1,\cdots, x_n)$ of moduli space of bordered hyperbolic surfaces with geodesic boundary components of length $x_1,\cdots, x_n$, which relies on a novel bound on intersection numbers in this paper. Recall that on $\overline{\M}_{g,n}$ the compactified moduli space of surfaces of genus $g$ with $n$ punctures, there are $n$ tautological line bundles $\{\mathcal{L}_i\}_{i=1}^n$. Let $\psi_i =c_1(\mathcal{L}_i) \in H_2(\overline{\M}_{g,n},\mathbb{Q})$ be the first Chern class of $\mathcal{L}_i$ and $\omega$ be the Weil-Petersson symplectic form on $\overline{\M}_{g,n}$. For any $d=(d_1,\cdots,d_n)$ with each $d_i\in\Z_{\geq 0}$ and $|d|=d_1+\cdots+d_n\leq 3g-3+n$, the \emph{intersection number} $\left[\tau_{d_1}\cdots\tau_{d_n}\right]_{g,n}$ is defined as
\begin{equation*}
\left[\tau_{d_1}\cdots\tau_{d_n}\right]_{g,n} = \frac{\prod_{i=1}^n 2^{2d_i}(2d_i+1)!!}{(3g-3+n-|d|)!} \int_{\overline{\M}_{g,n}} \psi_1^{d_1}\cdots \psi_n^{d_n} \omega^{3g-3+n-|d|}.
\end{equation*}
In particular, $\left[\tau_{0}\cdots\tau_{0}\right]_{g,n}=\Vol(\sM_{g,n})$, which is also denoted by $V_{g,n}$. There are deep connections between intersection numbers and Weil-Petersson volumes, especially for their asymptotics. Actually Mirzakhani \cite{Mirz07-int} showed that Weil-Petersson volumes can be expressed in terms of intersection numbers. For asymptotic behaviors, one may see \cite{Mirz13, LX14, MZ15, DGZZ20, Agg21} and the references therein for related results. Based on her recursive formula, Mirzakhani in {\cite[Page 286]{Mirz13}} showed that there exists a constant $c(n)>0$ independent of $g$ and $d_i$'s such that 
$$1-c(n)\frac{|d|^2}{g}\leq \frac{\left[\tau_{d_1}\cdots\tau_{d_n}\right]_{g,n}}{V_{g,n}} \leq 1.$$
In particular, for fixed $n$ and $|d|=o(\sqrt{g})$, 
\[ \frac{\left[\tau_{d_1}\cdots\tau_{d_n}\right]_{g,n}}{V_{g,n}} \sim 1 \text{ \ as $g\to \infty$}.\]

In this paper, we study the case that $\max\limits_{1\leq i\leq n}\{d_i^2\}$ is large compared to $2g+n$, which seems to have not been presented generally before. We prove
\begin{theorem}\label{in-up}
There exists a universal constant $c'>0$ such that
$$\frac{\left[\tau_{d_1}\cdots\tau_{d_n}\right]_{g,n}}{V_{g,n}} \leq c' \frac{2g-2+n}{\max\limits_{1\leq i\leq n}\{d_i^2\}}.$$
\end{theorem}

It is known by \cite[Proposition 3.1]{MP19} and \cite[Lemma 19]{NWX20} that there exists a constant $c(n)>0$ only depending on $n$ such that
$$\left(1-c(n)\frac{\sum_{i=1}^n x_i^2}{g}\right)\prod_{i=1}^n \frac{\sinh(x_i/2)}{x_i/2} \leq \frac{V_{g,n}(x_1,\cdots,x_n)}{V_{g,n}} \leq \prod_{i=1}^n \frac{\sinh(x_i/2)}{x_i/2}.$$
One may also see recent work \cite{AM22} of Anantharaman-Monk for sharper results. In particular, for fixed $n$ and $\sum x_i^2=o(g)$, 
\[ \frac{V_{g,n}(x_1,\cdots,x_n)}{V_{g,n}}  \sim \prod_{i=1}^n \frac{\sinh(x_i/2)}{x_i/2} \text{ \ as $g\to \infty$}.\]

In this paper we apply Theorem \ref{in-up} to study $V_{g,n}(x_1,\cdots, x_n)$ when $\max\limits_{1\leq i\leq n}\{x_i^2\}$ is large compared to $2g+n$. The following bound is essential in the proof of Part $(2)$ of Theorem \ref{mt-geodesic}.
\begin{theorem}\label{thm Vgn(x) big x-2}
For any $k\geq 1$, there exists a constant $c(k)>0$ independent of $g,n$ and $x_i$'s such that 
$$\frac{V_{g,n}(x_1,\cdots,x_n)}{V_{g,n}} \leq \left(\prod_{i=1}^n \min\left\{c(k)\frac{2g-2+n}{x_i^2}, 1\right\}\right)^k \prod_{i=1}^n \frac{\sinh(x_i/2)}{x_i/2}.$$
\end{theorem}

We believe that both Theorem \ref{in-up} and \ref{thm Vgn(x) big x-2} will find further other applications on related problems. 

\subsection*{Notations.} We say two positive functions $$f_1(g)\prec f_2(g) \quad \emph{or} \quad f_2(g)\succ f_1(g)$$ if there exists a universal constant $C>0$, independent of $g$, such that $f_1(g) \leq C \cdot f_2(g)$; and we say $$f_1(g) \asymp f_2(g)$$ if $f_1(g)\prec f_2(g)$ and $f_2(g)\prec f_1(g)$. And $f_1(g)\prec_{*} f_2(g)$ means that the constant $C$ depends on $*$.

We also say two functions  
\[h_1=O\left(h_2\right),\]
where $h_1$ may not be positive, if there exists a universal constant $C>0$ such that $|h_1| \leq C \cdot h_2$. And $h_1=O_*(h_2)$ means that $C$ depends on $*$.

\subsection*{Plan of the paper.} Section \ref{prelina} will provide a review of relevant background materials. The proofs of Theorem \ref{mt-1} and \ref{thm pgt5/6} are independent. We will first prove Theorem \ref{thm pgt5/6} in Section \ref{sec-pgt-gs}, and then apply it to prove Theorem \ref{thm pgt E} in Section \ref{sec-expect}. In Section \ref{sec-pgt-rs} we complete the proof of Theorem \ref{mt-1}. Then in Section \ref{sec-2-bounds} we establish the two new bounds on intersection numbers and Weil-Petersson volumes, i.e., prove Theorem \ref{in-up} and \ref{thm Vgn(x) big x-2}. In the last Section, we apply Theorem \ref{mt-1}, \ref{thm pgt E},  \ref{in-up}, \ref{thm Vgn(x) big x-2} and our previous works \cite{NWX20, WX22-GAFA} to prove Theorem \ref{mt-geodesic}. In Appendix \ref{appendix}, we follow \cite[Section 8]{WX22-GAFA} to prove a new counting result, i.e., Theorem \ref{thm count fill k-tuple}, which will be used in the proof of  Theorem \ref{mt-geodesic} when $L(g)\asymp \ln g$.

\subsection*{Acknowledgement}
We would like to thank all the participants in our seminar on Teichm\"uller theory for helpful discussions and comments on this project. We also would like to thank Ze\'ev Rudnick for his useful comments. The first named author is partially supported by the NSFC grant No. $12171263$ and 12361141813.

\tableofcontents

\section{Preliminaries}\label{prelina}

In this section, we include certain necessary backgrounds and tools used in this paper.

\subsection*{The Weil-Petersson metric}
We denote $S_{g,n}$ to be an oriented topological surface with $g$ genus and $n$ punctures or boundaries. Require $2g+n\geq 3$ so that $S_{g,n}$ admits hyperbolic metrics. Let $\T_{g,n}$ be the Teichm\"uller space of surfaces with $g$ genus and $n$ punctures. The moduli space of Riemann surfaces is $\M_{g,n}=\T_{g,n}/\Mod_{g,n}$ where $\Mod_{g,n}$ is the mapping class group of $S_{g,n}$ which fixes the order of punctures. If $n=0$, write $\T_{g}=\T_{g,0}$ and $\M_{g}=\M_{g,0}$ for simplicity. Given $L=(L_1,\cdots,L_n)\in \R^n_{\geq 0}$, the Teichm\"uller space $\T_{g,n}(L)$ consists of bordered hyperbolic surfaces with geodesic boundary components of length $L_1,\cdots,L_n$. Note that a geodesic boundary component of length $0$ is a puncture, so $\T_{g,n}(0,\cdots,0)=\T_{g,n}$. The weighted moduli space is $\M_{g,n}(L)=\T_{g,n}(L)/\Mod_{g,n}$.

Fix a pants decomposition $\{\alpha_i\}_{i=1}^{3g-3+n}$ of $S_{g,n}$, the Fenchel-Nielsen coordinates $X\mapsto (\ell_{\alpha_i}(X),\tau_{\alpha_i}(X))_{i=1}^{3g-3+n}$ are global coordinates for the Teichm\"uller space $\T_{g,n}(L)$. Where $\ell_{\alpha_i}(X)$ is the length of $\alpha_i$ on $X$ and $\tau_{\alpha_i}(X)$ is the twist along $\alpha_i$ (measured by length). Wolpert \cite{Wolpert82} showed that the Weil-Petersson
symplectic form of $\T_{g,n}(L)$ has a good expression under Fenchel-Nielsen coordinates:
\begin{theorem}[Wolpert]\label{thm Wolpert}
The Weil-Petersson symplectic form $\omega_{\mathrm{WP}}$ on $\T_{g,n}(L)$ is given by 
$$\omega_{\mathrm{WP}} = \sum_{i=1}^{3g-3+n} d\ell_{\alpha_i} \wedge d\tau_{\alpha_i}.$$
\end{theorem}

The Weil-Petersson volume form is 
$$d\Vol\nolimits_{\mathrm{WP}}=\tfrac{1}{(3g-3+n)!}\underbrace{\omega_{\mathrm{WP}}\wedge\cdots\wedge\omega_{\mathrm{WP}}}_{3g-3+n\ \text{copies}}.$$
This is a mapping class group invariant measure on $\T_{g,n}(L)$, and hence a measure on $\M_{g,n}(L)$, which still denoted by $d\Vol_{\mathrm{WP}}$. The total volume of $\M_{g,n}(L)$ under the Weil-Petersson metric is finite and denoted by $V_{g,n}(L)$. This also implies a probability measure $\Prob$ on $\M_g$:
$$\Prob(\mathcal{A}):=\frac{1}{V_{g}}\int_{\M_{g}} \mathbf{1}_{\mathcal{A}} dX$$
where $\mathcal{A}\sbs\M_g$ is a Borel subset, $\mathbf{1}_{\mathcal{A}}:\M_{g}\to\{0,1\}$ is its characteristic function, and $dX$ is short for $d\Vol_{\mathrm{WP}}(X)$. One may see \cite{Mirz13, GPY11, MP19, NWX20, PWX20, SW22, HW22} and the references therein for recent studies of the geometry of Weil-Petersson random hyperbolic surfaces.

\subsection*{Mirzakhani's integration formula}
Now we introduce an integration formula in \cite{Mirz07, Mirz13}, which is an essential tool in the study of random hyperbolic surfaces of Weil-Petersson model.

Given any non-peripheral closed curve $\gamma$ on topological surface $S_{g,n}$ and $X\in\T_{g,n}$, denote $\ell(\gamma)=\ell_\gamma(X)$ to be the hyperbolic length of the unique closed geodesic in
the homotopy class $\gamma$ on $X$. Let $\Gamma=(\gamma_1,\cdots,\gamma_k)$ be an ordered k-tuple where the $\gamma_i$'s are distinct disjoint homotopy classes of nontrivial, non-peripheral, unoriented simple closed curves on $S_{g,n}$. Consider the orbit containing $\Gamma$ under $\Mod_{g,n}$-action
$$\sO_\Gamma = \{(h\cdot\gamma_1,\cdots,h\cdot\gamma_k);\ h\in\Mod\nolimits_{g,n}\}.$$
Given a function $F:\R_{\geq0}^k \to\R$, we define a function on $\M_{g,n}$
\begin{eqnarray*}
F^\Gamma: \M_{g,n} &\to& \R \\
X &\mapsto& \sum_{(\alpha_1,\cdots,\alpha_k)\in\sO_\Gamma} F(\ell_{\alpha_1}(X),\cdots,\ell_{\alpha_k}(X)).
\end{eqnarray*}
Assume $S_{g,n}-\cup_{j=1}^k\gamma_j=\cup_{i=1}^s S_{g_i,n_i}$. For any $x=(x_1,\cdots,x_k)\in\R_{\geq 0}^k$, we consider the moduli space $\M(S_{g,n}(\Gamma);\ell_{\Gamma}=x)$ of the hyperbolic surfaces (possibly disconnected) homeomorphic to $S_{g,n}-\cup_{j=1}^k\gamma_j$ with $\ell_{\gamma_i^1}=\ell_{\gamma_i^2}=x_i$ for every $i=1,\cdots,k$, where $\gamma_i^1$ and $\gamma_i^2$ are two boundary components of $S_{g,n}-\cup_{j=1}^k\gamma_j$ given by cutting along $\gamma_i$. Consider the Weil-Petersson volume
$$V_{g,n}(\Gamma,x)=\Vol\nolimits_{\mathrm{WP}}\left(\M(S_{g,n}(\Gamma);\ell_{\Gamma}=x)\right) = \prod_{i=1}^s V_{g_i,n_i}(x^{(i)})$$
where $x^{(i)}$ is the list of those coordinates $x_j$ of $x$ such that $\gamma_j$ is a boundary component of $S_{g_i,n_i}$. Using Theorem \ref{thm Wolpert}, Mirzakhani proved the following integration
formula. See \cite[Theorem 7.1]{Mirz07}, \cite[Theorem 2.2]{Mirz13}, \cite[Theorem 2.2]{MP19}, \cite[Theorem 4.1]{Wright-tour} for different versions.

\begin{theorem}[Mirzakhani]\label{thm Mirz int formula}
For any $\Gamma=(\gamma_1,\cdots,\gamma_k)$, the integral of $F^\Gamma$ over $M_{g,n}$ with respect to Weil-Petersson metric is given by
$$\int_{\M_{g,n}} F^\Gamma dX = C_\Gamma\int_{\R_{\geq 0}^k}F(x_1,\cdots,x_k)V_{g,n}(\Gamma,x)x_1\cdots x_k dx_1\cdots dx_k$$
where the constant $C_\Gamma\in(0,1]$ only depend on $\Gamma$. Moreover, $C_\Gamma=\frac{1}{2^k}$ if $g>2$ and $S_{g,n}-\cup_{j=1}^k\gamma_j \cong S_{g-k,n+2k}$.
\end{theorem}

\subsection*{The Weil-Petersson volume}
We denote $V_{g,n}(x_1,\cdots,x_n)$ to be the Weil-Petersson volume of $\M_{g,n}(x_1,\cdots,x_n)$ and $V_{g,n}=V_{g,n}(0,\cdots,0)$. Recall that the intersection numbers $\left[\tau_{d_1}\cdots\tau_{d_n}\right]_{g,n}$ are integrals of first Chern classes of tautological line bundles and Weil-Petersson symplectic form over $\overline{\M}_{g,n}$. Mirzakhani \cite{Mirz07-int} showed that this volume can be expressed by intersection numbers:
\begin{theorem}\label{thm Vgn(x)}
For $x_1,\cdots,x_n \geq 0$, 
$$V_{g,n}(2x_1,\cdots,2x_n)= \sum_{|d|\leq 3g-3+n} \left[\tau_{d_1}\cdots\tau_{d_n}\right]_{g,n} \prod_{i=1}^n\frac{x_i^{2d_i}}{(2d_i+1)!}.$$
\end{theorem}

Mirzakhani gave some estimates about $V_{g,n}$:
\begin{theorem}\label{thm Vgn/Vgn+1}
\begin{enumerate}
\item \cite[Lemma 3.2]{Mirz13}
For any $g,n\geq 0$, 
$$b_0<\frac{(2g-2+n)V_{g,n}}{V_{g,n+1}}<b_1$$
where $b_0=\frac{10-\pi^2}{120}\approx 0.001$ and $b_1=\sum_{l=1}^\infty \frac{l \pi^{2l-2}}{(2l+1)!}\approx 0.401$.
\item \cite[Theorem 3.5]{Mirz13}
$$\frac{(2g-2+n)V_{g,n}}{V_{g,n+1}} = \frac{1}{4\pi^2} + O\left(\frac{1}{g}\right),$$
$$\frac{V_{g,n}}{V_{g-1,n+2}} = 1+O\left(\frac{1}{g}\right).$$
Where the implied constants are related to $n$ and independent of $g$.
\end{enumerate}
\end{theorem}

\begin{theorem}[{\cite[Corollary 3.7]{Mirz13}}]\label{thm sum V*V}
\begin{equation*}
	\sum_{\tiny\begin{array}{c}
			g_1+g_2=g  \\
			1\leq g_1\leq g_2
	\end{array}}
	\frac{V_{g_1,1} V_{g_2,1}}{V_g} \asymp \frac{1}{g} 
	\ \ \text{and} \ \ 
	\sum_{\tiny\begin{array}{c}
			g_1+g_2=g-1  \\
			1\leq g_1\leq g_2
	\end{array}}
	\frac{V_{g_1,2} V_{g_2,2}}{V_g} \asymp \frac{1}{g^2}.
\end{equation*}
The implied constants are independent of $g$.
\end{theorem}

The asymptotic behavior of $V_{g,n}(x_1,\cdots,x_n)$ was first studied in \cite{MP19}, also see \cite{NWX20} and refined results in  \cite{AM22}.
\begin{theorem}\label{thm Vgn(x) small x}
There exists a constant $c(n)>0$ independent of $g$ and $x_i$'s such that
$$\left(1-c(n)\frac{\sum_{i=1}^n x_i^2}{g}\right)\prod_{i=1}^n \frac{\sinh(x_i/2)}{x_i/2} \leq \frac{V_{g,n}(x_1,\cdots,x_n)}{V_{g,n}} \leq \prod_{i=1}^n \frac{\sinh(x_i/2)}{x_i/2}.$$
\end{theorem}

\subsection*{Counting closed geodesics}
Here we introduce some counting results. We want some upper bounds for the number of closed geodesics of length $\leq L$. On a hyperbolic surface, a closed geodesic is called \emph{primitive} if it is not an iterate of any other closed geodesic at least twice. First, the collar lemma (see e.g. \cite[Theorem 4.1.6]{Buser10}) states that:
\begin{theorem}\label{thm collar}
In a hyperbolic surface $X_g\in\M_{g}$, there are at most $3g-3$ unoriented primitive closed geodesics of length $\leq 2\arcsinh 1 \approx 1.7627$. Moreover, they are simple and disjoint with each other.
\end{theorem}

Buser showed that:
\begin{theorem}[{\cite[Theorem 6.6.4]{Buser10}}]\label{thm count Buser}
Let $X_g\in\M_g$ and $L>0$. In $X_g$ there are at most $(g-1)e^{L+6}$ oriented closed geodesics of length $\leq L$ which are not iterates of closed geodesics of length $\leq 2\arcsinh 1$.
\end{theorem}
For a hyperbolic surface $X_g\in\M_g$, denote $\mathcal{P}(X_g)$ to be the set of all oriented primitive closed geodesics in $X_g$. Then combine Theorem \ref{thm collar} and \ref{thm count Buser} we have: 
\begin{corollary}\label{thm count ge^L upp}
For any $X_g\in\M_g$ and $L>0$,
$$\#\left\{\gamma\in\mathcal{P}(X_g);\ \ell(\gamma)\leq L\right\} \leq (g-1)e^{L+7}.$$
\end{corollary}

For compact hyperbolic surfaces with geodesic boundaries, the following counting result in our previous work \cite{WX22-GAFA} will be applied later.
\begin{theorem}[{\cite[Theorem 4]{WX22-GAFA}}]\label{sec-count}
For any $\eps_1>0$ and $m=2g-2+n\geq 1$, there exists a constant $c(\eps_1,m)>0$ only depending on $m$ and $\eps_1$ such that for all $L>0$ and any compact hyperbolic surface $X$ of genus $g$ with $n$ boundary simple closed geodesics, we have
\begin{equation*}
\#_f(X,L)\leq c(\eps_1,m) \cdot e^{L-\frac{1-\eps_1}{2}\ell(\partial X)}.
\end{equation*}
Where $\#_f(X,L)$ is the number of filling closed geodesics in $X$ of length $\leq L$ and $\ell(\partial X)$ is the total length of the boundary closed geodesics of $X$.
\end{theorem}

Actually Theorem \ref{sec-count} also holds for filling multi-curves. 
\begin{definition}\label{def fill k-tuple}
	For $2g+n\geq 3$, let $\Gamma=(\gamma_1,\cdots,\gamma_k)$ be an ordered $k$-tuple where $\gamma_i$'s are homotopy classes of
	nontrivial non-peripheral unoriented primitive closed curves on the topological surface $S_{g,n}$. We say $\Gamma$ is \emph{filling} in $S_{g,n}$ if each component of the complement $S_{g,n}\setminus \cup_{i=1}^k \gamma_i$ is homeomorphic to a disk or a punctured disk which is homotopic to a cusp of $S_{g,n}$.
	
	In particular, a filling $1$-tuple is a filling closed curve in $S_{g,n}$.
	
	On a hyperbolic surface $X\in\T_{g,n}(L_1,\cdots,L_n)$, we define the length of a $k$-tuple $\Gamma=(\gamma_1,\cdots,\gamma_k)$ to be the total length of $\gamma_i$'s, that is, 
	$$\ell_{\Gamma}(X):=\sum_{i=1}^k \ell_{\gamma_i}(X).$$
	Define the counting function $N_k^{\text{fill}}(X,L)$ for $L\geq 0$ and $X\in\T_{g,n}(L_1,\cdots,L_n)$ to be 
	$$N_k^{\text{fill}}(X,L):= \#\left\{\Gamma=(\gamma_1,\cdots,\gamma_k);\ 
	\begin{aligned}
		&\Gamma\ \text{is a filling}\ k\text{-tuple in}\ X \\
		&\text{and}\ \ell_{\Gamma}(X)\leq L
	\end{aligned} \right\}.$$
\end{definition}

\begin{theorem}\label{thm count fill k-tuple}
	For any $k\in\Z_{\geq 1}$, $0<\eps<\frac{1}{2}$ and $m=2g-2+n\geq 1$, there exists
	a constant $c(k,\eps,m)>0$ only depending on $k,\eps$ and $m$ such that for all $L>0$ and any compact hyperbolic surface $X$ of genus $g$ with $n$ boundary simple closed geodesics, the following holds:
		$$N_k^{\text{fill}}(X,L)\leq c(k,\eps,m)\cdot (1+L)^{k-1} e^{L-\frac{1-\eps}{2}\ell(\partial X)}.$$
	Where $\ell(\partial X)$ is the total length of the boundary closed geodesics of $X$.
\end{theorem}
The proof of Theorem \ref{thm count fill k-tuple} is actually the same as the proof of Theorem \ref{sec-count} given in {\cite[Section 8]{WX22-GAFA}}. We will give a sketch in Appendix \ref{appendix}.

\subsection*{Small eigenvalues}
For $X_g\in\M_g$, the spectrum of the Laplacian operator consists of discrete eigenvalues 
$$0=\lambda_0(X_g) < \lambda_1(X_g)\leq \lambda_2(X_g)\leq \cdots \to \infty.$$
Spectrum of hyperbolic surfaces are widely studied in the past century, see \cite{Chavel, Buser10, Ber16} for example. Eigenvalues that are $\leq \frac{1}{4}$ are called small eigenvalues. Buser \cite{Bus77} showed that $\lambda_{4g-2}(X_g)>\frac{1}{4}$ and constructed $\mathcal{X}_g\in\M_g$ which admits arbitrary small $\lambda_{2g-3}(\mathcal{X}_g)$. Later, Otal and Rosas \cite{OR09} proved that $\lambda_{2g-2}(X_g)>\frac{1}{4}$. Here $2g-2$ is sharp. One may also see Ballmann-Matthiesen-Mondal \cite{BMM16,BMM17} for more general statements on $\lambda_{2g-2}(X_g)$.
\begin{theorem}[\cite{OR09}]\label{thm OR09}
	There are at most $2g-2$ eigenvalues in $[0,\frac{1}{4}]$ for any $X_g\in\M_g$.
\end{theorem}

\subsection*{Selberg's trace formula}
We introduce the Selberg's trace formula \cite{Selb56} of closed hyperbolic surfaces here, which is a useful tool when considering the spectrum of Laplacian and length spectrum. One may see \cite[Theorem 9.5.3]{Buser10} or \cite[Theorem 5.6]{Ber16} for the version we describe here.

Given a smooth function with compact support $\vph\in C_c^\infty (\R)$, its Fourier transform is defined as
$$\widehat{\vph}(z) = \int_\R \vph(t) e^{-\mathbf{i} tz}dt$$
for any $z\in\C$.

For $k\in\Z_{\geq 0}$, denote
\be
r_k(X_g) =
\begin{cases}
\sqrt{\lambda_k(X_g)-\frac{1}{4}}\ &\ \text{if}\ \lambda_k(X_g)> \frac{1}{4} \\
\mathbf{i} \cdot\sqrt{\frac{1}{4}-\lambda_k(X_g)}\ &\ \text{if}\ \lambda_k(X_g)\leq \frac{1}{4}
\end{cases}
\ene
where $\lambda_k(X_g)$ is the $k$-th eigenvalue of $X_g$. Denote $\mathcal{P}(X_g)$ to be the set of all oriented primitive closed geodesics on $X_g$. Then the remarkable trace formula of Selberg states:
\begin{theorem}[Selberg's trace formula]\label{thm trace formula}
Let $X_g\in\M_g$ be a closed hyperbolic surface of genus $g$. For any even function $\vph\in C_c^\infty(\R)$, we have
\begin{eqnarray*}
\sum_{k=0}^\infty \widehat{\vph}(r_k(X_g))
&=& (g-1)\int_\R r \tanh(\pi r)\widehat{\vph}(r)dr \\
&& + \sum_{k=1}^\infty \sum_{\gamma\in\mathcal{P}(X_g)} \frac{\ell(\gamma)}{2\sinh\left(\frac{k\ell(\gamma)}{2}\right)} \vph(k\ell(\gamma)).
\end{eqnarray*}
And both sides of the formula are absolutely convergent.
\end{theorem}

\section{PGT for general hyperbolic surfaces}\label{sec-pgt-gs}

In this section we finish the proof of Theorem \ref{thm pgt5/6}. Actually we show 
\begin{theorem}\label{thm pgt5/6-r}
	For any $X_g\in \M_g$ and $t>2$, 
	$$\pi_{X_g}(t)= \Li(t)+\sum_{0<\lambda_j\leq \frac{1}{4}}\Li(t^{s_j}) + O\left(g\frac{t^\frac{5}{6}}{\ln t}+ \frac{t^\frac{2}{3}}{\ln t}\sum_{\tiny\begin{array}{c}\gamma\in\sP(X_g), \\ \ell(\gamma)<t^{-1/6}\end{array}} \ln\left(\frac{1}{\ell(\gamma)t^{\frac{1}{6}}}\right) \right)$$
	where $s_j=s_j(X_g)=\frac{1}{2}+\sqrt{\frac{1}{4}-\lambda_j(X_g)}$, $\sP(X_g)$ is the set of all oriented primitive closed geodesics in $X_g$, and the implied constant is universal.
\end{theorem}

Theorem \ref{thm pgt5/6} is a direct consequence of Theorem \ref{thm pgt5/6-r} above.
\bp[Proof of Theorem \ref{thm pgt5/6}]
By Theorem \ref{thm collar} we know that there exist at most $(3g-3)$ mutually disjoint primitive closed geodesics in $X_g$ of length $\leq 2^{-1/6}$. So we have  
$$\sum_{\tiny\begin{array}{c}\gamma\in\sP(X_g), \\ \ell(\gamma)<t^{-\frac{1}{6}}\end{array}} \ln\left(\frac{1}{\ell(\gamma)t^{\frac{1}{6}}}\right) \prec g\cdot\max\left\{0,\ln\left(\frac{1}{\sys(X_g)t^{\frac{1}{6}}}\right)\right\}$$
which together with Theorem \ref{thm pgt5/6-r} implies the conclusion.
\ep

The error term in Theorem \ref{thm pgt5/6-r} is explicitly given in terms of the geometry of $X_g$. In order to achieve it, we will choose suitable test functions for Selberg's trace formula such that we are able to distinguish all the geometric terms as shown in Theorem \ref{thm pgt5/6-r}. Our method is motivated by the proof of \cite[Corollary 1.4]{BP22} by Bourque and Petri.

Before showing Theorem \ref{thm pgt5/6-r}, we first make the following application.
\begin{corollary}\label{thm count ge^L/L upp}
For any $X_g\in\M_g$, if $L$ satisfies $e^{e^{\frac{1}{3}L}}> \frac{1}{\sys(X_g)}$, then
$$\#\left\{\gamma\in\mathcal{P}(X_g);\ \ell(\gamma)\leq L\right\} \prec g \frac{e^L}{L}.$$
\end{corollary}
\begin{proof}
By Theorem \ref{thm pgt5/6-r} and \ref{thm OR09}, 
$$\#\left\{\gamma\in\mathcal{P}(X_g);\ \ell(\gamma)\leq L\right\} \prec g\frac{e^L}{L} + \frac{e^{\frac{2}{3}L}}{L} \sum_{\tiny\begin{array}{c}\gamma\in\sP(X_g), \\ \ell(\gamma)<e^{-L/6}\end{array}} \ln\left(\frac{1}{\ell(\gamma)e^{\frac{1}{6}L}}\right).$$
When $e^{e^{\frac{1}{3}L}}> \frac{1}{\sys(X_g)}$, we have
$$\sum_{\tiny\begin{array}{c}\gamma\in\sP(X_g), \\ \ell(\gamma)<e^{-L/6}\end{array}} \ln\left(\frac{1}{\ell(\gamma)e^{\frac{1}{6}L}}\right) \prec \max\left\{0,g\cdot\ln\left(\frac{1}{\sys(X_g)}\right)\right\} \leq g e^{\frac{1}{3}L}.$$

Then the conclusion follows by the two inequalities above.
\end{proof}

Now we return to prove Theorem \ref{thm pgt5/6-r}.

For bounded $t$, the theorem holds naturally. So we only consider sufficiently large $t$. Let $\eta(x)\geq 0$ be a non-negative smooth even function with $supp(\eta)\sbs (-\frac{1}{2},\frac{1}{2})$, $\int_\R\eta(x)dx=1$ and $\eta(0)=\max\limits_{x\in\R}\eta(x)=2$. Let $\eta_\eps(x)=\frac{1}{\eps}\eta(\frac{x}{\eps})$. For any $0<\eps<0.01$ and $L>1$, let
\begin{equation}
	f_\eps(x)= (\eta_\eps * \eta_\eps)(x),
\end{equation}
\begin{equation}
	\vph_{L,\eps}^\pm(x)= \frac{1}{2}(f_\eps(x-L)+f_\eps(x+L)) \pm f_\eps(x).
\end{equation}
We will use $\vph_{L,\eps}^+$ and $\vph_{L,\eps}^-$ as test functions to show the lower bound and upper bound of $\pi_{X_g}(t)$ respectively. By direct calculation, we have 
$$0\leq f_\eps(x)=\int_\R \eta_\eps(y)\eta_\eps(x-y)dy \leq \frac{2}{\eps}\int_\R \eta_\eps(y)dy=\frac{2}{\eps}.$$
So
\begin{equation}\label{equ pgt5/6 vph bound}
0\leq \vph_{L,\eps}^+(x)\leq \frac{4}{\eps}\ \ \ \text{and} \ \  -\frac{2}{\eps}\leq \vph_{L,\eps}^-(x)\leq \frac{2}{\eps}.
\end{equation}

It is clear that both $f_\eps$ and $\vph_{L,\eps}^\pm$ are smooth even functions with compact supports. And their Fourier transforms satisfy
\begin{equation}
	\widehat{f_\eps}(r)= (\widehat\eta(\eps r))^2\geq 0,
\end{equation}
\begin{equation}
	\widehat{\vph_{L,\eps}^+}(r)= (\cos(Lr) + 1)(\widehat\eta(\eps r))^2 \geq 0
\end{equation}
and
\begin{equation}
	\widehat{\vph_{L,\eps}^-}(r)= (\cos(Lr) - 1)(\widehat\eta(\eps r))^2 \leq 0.
\end{equation}

\begin{lemma}\label{thm pgt5/6 count lemma}
For any $L>1$ and $\eps<0.01$ we have
\begin{equation*}
\#\left\{\begin{array}{c}
	\gamma\in\sP(X_g),\\
	1<\ell(\gamma)\leq L
\end{array}\right\}
= 2\int_{1-\eps}^{L+\eps} \sum_{\tiny\begin{array}{c}\gamma\in\sP(X_g), \\ 1<\ell(\gamma)\leq L\end{array}} \vph_{\tau,\eps}^\pm(\ell(\gamma)) d\tau.
\end{equation*}
\end{lemma}
\begin{proof}
For any $x>1$ and $0<\eps<0.01$, since $supp(f_\eps)\sbs (-\eps,\eps)$ we have
\begin{eqnarray*}
	\int_0^\infty \vph_{\tau,\eps}^\pm(x) \mathbf{1}_{[-\eps,\eps]}(x-\tau) d\tau
	&= \int_{x-\eps}^{x+\eps} \frac{1}{2} f_\eps(x-\tau)d\tau +  \int_{x-\eps}^{x+\eps} \frac{1}{2} f_\eps(x+\tau)d\tau \\
	&\pm \int_{x-\eps}^{x+\eps} f_\eps(x)d\tau=  \int_{-\eps}^{\eps} \frac{1}{2} f_\eps(\tau)d\tau= \frac{1}{2}.
\end{eqnarray*}
So
\begin{eqnarray}
	\#\left\{\begin{array}{c}
		\gamma\in\sP(X_g),\\
		1<\ell(\gamma)\leq L
	\end{array}\right\} 
	&=& \sum_{\tiny\begin{array}{c}
			\gamma\in\sP(X_g), \\
			1<\ell(\gamma)\leq L
	\end{array}} 1 \nonumber\\
	&=& 2 \int_0^\infty \sum_{\tiny\begin{array}{c}
			\gamma\in\sP(X_g), \\
			1<\ell(\gamma)\leq L
	\end{array}} \vph_{\tau,\eps}^\pm(\ell(\gamma)) \mathbf{1}_{[-\eps,\eps]}(\ell(\gamma)-\tau) d\tau . \nonumber
\end{eqnarray}
Since $\vph_{\tau,\eps}^\pm(\ell(\gamma))=\frac{1}{2}f_\eps(\ell(\gamma)-\tau)$ when $1<\ell(\gamma)\leq L$ and $\tau\geq 0$, we have
\begin{eqnarray}
	\#\left\{\begin{array}{c}
		\gamma\in\sP(X_g),\\
		1<\ell(\gamma)\leq L
	\end{array}\right\} 
	&=& 2\int_{0}^{\infty} \sum_{\tiny\begin{array}{c}
			\gamma\in\sP(X_g), \\
			1<\ell(\gamma)\leq L
	\end{array}} \frac{1}{2}f_\eps(\ell(\gamma)-\tau) \mathbf{1}_{[-\eps,\eps]}(\ell(\gamma)-\tau) d\tau \nonumber\\
	&=& 2\int_{0}^{\infty} \sum_{\tiny\begin{array}{c}
			\gamma\in\sP(X_g), \\
			1<\ell(\gamma)\leq L
	\end{array}} \frac{1}{2}f_\eps(\ell(\gamma)-\tau) d\tau \nonumber\\
	&=& 2\int_{1-\eps}^{L+\eps} \sum_{\tiny\begin{array}{c}
			\gamma\in\sP(X_g), \\
			1<\ell(\gamma)\leq L
	\end{array}} \frac{1}{2}f_\eps(\ell(\gamma)-\tau) d\tau \nonumber\\
	&=& 2\int_{1-\eps}^{L+\eps} \sum_{\tiny\begin{array}{c}
			\gamma\in\sP(X_g), \\
			1<\ell(\gamma)\leq L
	\end{array}} \vph_{\tau,\eps}^\pm(\ell(\gamma)) d\tau. \nonumber
\end{eqnarray}
This finishes the proof.
\end{proof}

By Selberg's trace formula, i.e., Theorem \ref{thm trace formula}, 
\begin{eqnarray}\label{equ pgt5/6 Selberg}
	\sum_{k=1}^\infty\sum_{\gamma\in\sP(X_g)} \frac{\ell(\gamma)}{2\sinh\frac{k\ell(\gamma)}{2}} \vph_{\tau,\eps}^\pm(k\ell(\gamma)) &=& \sum_{0\leq\lambda_j\leq\frac{1}{4}} \widehat{\vph_{\tau,\eps}^\pm}(r_j) + \sum_{\lambda_j>\frac{1}{4}} \widehat{\vph_{\tau,\eps}^\pm}(r_j) \nonumber\\
	&& -(g-1)\int_\R r\tanh(\pi r) \widehat{\vph_{\tau,\eps}^\pm}(r)dr. 
\end{eqnarray}
We will estimate each term and give an upper bound for $\sum \vph_{\tau,\eps}^-(\ell(\gamma))$ and a lower bound for $\sum \vph_{\tau,\eps}^+(\ell(\gamma))$ which appear in Lemma \ref{thm pgt5/6 count lemma}.

\begin{lemma}\label{thm pgt5/6 vph-<}
	For $\tau\in[1-\eps, L+\eps]$ where $L>1$ and $0<\eps<0.01$, we have
	\begin{eqnarray*}
		&&\sum_{\tiny\begin{array}{c}
				\gamma\in\sP(X_g), \\
				1<\ell(\gamma)\leq L
		\end{array}} \vph_{\tau,\eps}^-(\ell(\gamma)) \\
		&&\leq \frac{1}{2}\sum_{0\leq \lambda_j\leq\frac{1}{4}} \frac{e^{s_j\tau}}{\tau} + O_\eta\left(g\eps\frac{e^{\tau}}{\tau}+\frac{g}{\eps^2}\frac{e^{\frac{1}{2}\tau}}{\tau}+\frac{1}{\eps}\frac{e^{\frac{1}{2}\tau}}{\tau}\sum_{\tiny\begin{array}{c}
				\gamma\in\sP(X_g), \\
				\ell(\gamma)<\eps
		\end{array}} \ln\frac{\eps}{\ell(\gamma)}\right).
	\end{eqnarray*}
\end{lemma}
\begin{proof}
	Since $\vph_{\tau,\eps}^-(x)\geq 0$ when $x>\eps$ and $\vph_{\tau,\eps}^-(x)=0$ when $x\geq \tau+\eps$, by applying $\vph_{\tau,\eps}^-(x)\geq -\frac{2}{\eps}$ in \eqref{equ pgt5/6 vph bound} and the fact that the function $\frac{x}{2\sinh\frac{x}{2}}$ is decreasing on $(0,\infty)$, we have
	\begin{eqnarray}\label{equ 5/6 upper Selberg LHS}
		&&\sum_{k=1}^\infty\sum_{\gamma\in\sP(X_g)} \frac{\ell(\gamma)}{2\sinh\frac{k\ell(\gamma)}{2}} \vph_{\tau,\eps}^-(k\ell(\gamma)) \\
		&\geq& \frac{\tau+\eps}{2\sinh\frac{\tau+\eps}{2}} \sum_{\tiny\begin{array}{c} \gamma\in\sP(X_g), \\
		\ell(\gamma)>\eps \end{array}} \vph_{\tau,\eps}^-(\ell(\gamma)) - \frac{2}{\eps}\sum_{k=1}^\infty \sum_{\tiny\begin{array}{c} \gamma\in\sP(X_g), \\
		k\ell(\gamma)\leq\eps \end{array}} \frac{\ell(\gamma)}{2\sinh\frac{k\ell(\gamma)}{2}} \nonumber\\
		&\geq& \frac{\tau+\eps}{2\sinh\frac{\tau+\eps}{2}} \sum_{\tiny\begin{array}{c} \gamma\in\sP(X_g), \\
		1<\ell(\gamma)\leq L \end{array}} \vph_{\tau,\eps}^-(\ell(\gamma)) - \frac{2}{\eps}\sum_{k=1}^\infty \sum_{\tiny\begin{array}{c} \gamma\in\sP(X_g), \\
		k\ell(\gamma)\leq\eps \end{array}} \frac{\ell(\gamma)}{2\sinh\frac{k\ell(\gamma)}{2}}. \nonumber
	\end{eqnarray}
	
	Denote $\{\gamma_1,\cdots,\gamma_m\}$ to be the set of all oriented primitive closed geodesics with length $\leq \eps$. By Theorem \ref{thm collar}, $m\leq 6g-6$. Then
	\begin{equation}\label{equ 5/6 upper Selberg LHS'}
\begin{aligned}
		&\frac{2}{\eps}\sum_{k=1}^\infty \ \sum_{\tiny\begin{array}{c} \gamma\in\sP(X_g), \\
		k\ell(\gamma)\leq\eps \end{array}} \frac{\ell(\gamma)}{2\sinh\frac{k\ell(\gamma)}{2}} = \frac{2}{\eps}\sum_{j=1}^m\sum_{k=1}^{[\eps/\ell(\gamma_j)]} \frac{\ell(\gamma_j)}{2\sinh\frac{k\ell(\gamma_j)}{2}} \\
		&\leq \frac{2}{\eps}\sum_{j=1}^m\sum_{k=1}^{[\eps/\ell(\gamma_j)]} \frac{1}{k}\prec \frac{1}{\eps}\sum_{j=1}^m \left(1+\ln\frac{\eps}{\ell(\gamma_j)}\right) \\
		&\prec \frac{g}{\eps} + \frac{1}{\eps} \sum_{\tiny\begin{array}{c} \gamma\in\sP(X_g), \\
		\ell(\gamma)<\eps \end{array}}  \ln\frac{\eps}{\ell(\gamma)}.
\end{aligned}
	\end{equation}
	
	Since $\widehat{\vph_{\tau,\eps}^-}(r)\leq 0$ for $r\in\R$, we have
	\begin{equation}\label{equ 5/6 upper Selberg RHS1}
		\sum_{\lambda_j>\frac{1}{4}} \widehat{\vph_{\tau,\eps}^-}(r_j)\leq 0.
	\end{equation}

	For each small eigenvalue $\lambda_j\in [0,\frac{1}{4}]$,
	\begin{eqnarray}
		\widehat{\vph_{\tau,\eps}^-}(r_j) &=& \left(\cosh(\tau\sqrt{\frac{1}{4}-\lambda_j})-1\right)\left(\widehat{\eta}({\bf i} \eps\sqrt{\frac{1}{4}-\lambda_j})\right)^2 \nonumber\\
		&\leq& \frac{1}{2}e^{(s_j-\frac{1}{2})\tau} \left(1+O_\eta(\eps^2)\right) \nonumber
	\end{eqnarray}
	where $\widehat{\eta}({\bf i}\eps\sqrt{\frac{1}{4}-\lambda_j})=1+O_\eta(\eps^2)$ since $\widehat{\eta}(0)=1$ and $\widehat{\eta}$ is even. And hence combined with Theorem \ref{thm OR09} we have 
	\begin{equation}\label{equ 5/6 upper Selberg RHS2}
		\sum_{0\leq\lambda_j\leq\frac{1}{4}} \widehat{\vph_{\tau,\eps}^-}(r_j)\leq \frac{1}{2}\sum_{0\leq \lambda_j\leq\frac{1}{4}} e^{(s_j-\frac{1}{2})\tau} + O_\eta\left(g\eps^2 e^{\frac{1}{2}\tau}\right).
	\end{equation}
	
	For the integral term,
	\begin{equation}\label{equ 5/6 upper Selberg RHS3}
\begin{aligned}
		&\left| (g-1)\int_\R r\tanh(\pi r)\widehat{\vph_{\tau,\eps}^-}(r) dr \right| \\
		&= \left| 2(g-1)\int_0^\infty r\tanh(\pi r)(\cos(r\tau)-1)(\widehat{\eta}(\eps r))^2 dr \right| \\
		&\leq 4(g-1)\int_0^\infty r\cdot(\widehat{\eta}(\eps r))^2 dr\\
& = 4(g-1)\frac{1}{\eps^2} \int_0^\infty r\cdot(\widehat{\eta}(r))^2 dr\\
& \prec_\eta \frac{g}{\eps^2}.
\end{aligned}
	\end{equation}
	
	Then plug \eqref{equ 5/6 upper Selberg LHS}, \eqref{equ 5/6 upper Selberg LHS'}, \eqref{equ 5/6 upper Selberg RHS1}, \eqref{equ 5/6 upper Selberg RHS2} and \eqref{equ 5/6 upper Selberg RHS3} into \eqref{equ pgt5/6 Selberg} we have
	\begin{eqnarray}
		&&\frac{\tau+\eps}{2\sinh\frac{\tau+\eps}{2}} \sum_{\tiny\begin{array}{c} \gamma\in\sP(X_g), \\
		1<\ell(\gamma)\leq L \end{array}} \vph_{\tau,\eps}^-(\ell(\gamma)) \nonumber\\
		&\leq& \frac{1}{2}\sum_{0\leq \lambda_j\leq\frac{1}{4}} e^{(s_j-\frac{1}{2})\tau} + O_\eta\left(g\eps^2e^{\frac{1}{2}\tau}+\frac{g}{\eps^2}+\frac{1}{\eps} \sum_{\tiny\begin{array}{c} \gamma\in\sP(X_g), \\
		\ell(\gamma)<\eps \end{array}} \ln\frac{\eps}{\ell(\gamma)}\right). \nonumber
	\end{eqnarray}
	Since $\frac{2\sinh\frac{\tau+\eps}{2}}{\tau+\eps} \leq \frac{e^{\tau/2}}{\tau} e^{\eps/2}=\frac{e^{\tau/2}}{\tau}(1+O(\eps))$, combining with Theorem \ref{thm OR09} we get
	\begin{eqnarray}
		&& \sum_{\tiny\begin{array}{c} \gamma\in\sP(X_g), \\
		1<\ell(\gamma)\leq L \end{array}} \vph_{\tau,\eps}^-(\ell(\gamma)) \nonumber\\
		&\leq& \frac{1}{2}\sum_{0\leq \lambda_j\leq\frac{1}{4}} \frac{e^{s_j\tau}}{\tau} + O_\eta\left(g\eps\frac{e^{\tau}}{\tau}+\frac{g}{\eps^2}\frac{e^{\frac{1}{2}\tau}}{\tau}+\frac{1}{\eps}\frac{e^{\frac{1}{2}\tau}}{\tau} \sum_{\tiny\begin{array}{c} \gamma\in\sP(X_g), \\
		\ell(\gamma)<\eps \end{array}} \ln\frac{\eps}{\ell(\gamma)}\right) \nonumber
	\end{eqnarray}
	as desired.
\end{proof}

\begin{lemma}\label{thm pgt5/6 vph+>}
For $\tau\in[1+\eps, L-\eps]$ where $L>1$ and $0<\eps<0.01$, we have
\begin{eqnarray*}
&&\sum_{\tiny\begin{array}{c} \gamma\in\sP(X_g), \\
1<\ell(\gamma)\leq L \end{array}} \vph_{\tau,\eps}^+(\ell(\gamma)) \\
&&\geq \frac{1}{2}\sum_{0\leq \lambda_j\leq\frac{1}{4}} \frac{e^{s_j\tau}}{\tau} - O_\eta\left(g\eps\frac{e^{\tau}}{\tau}+g\frac{(1+\eps\tau)}{\eps^2}\frac{e^{\frac{1}{2}\tau}}{\tau}+\frac{1}{\eps}\frac{e^{\frac{1}{2}\tau}}{\tau} \sum_{\tiny\begin{array}{c} \gamma\in\sP(X_g), \\
\ell(\gamma)<\eps \end{array}} \ln\frac{\eps}{\ell(\gamma)}\right).
\end{eqnarray*}
\end{lemma}
\begin{proof}
Since $\vph_{\tau,\eps}^+(x)= 0$ when $x\in[\eps,\tau-\eps]\cup[\tau+\eps,+\infty]$, by applying $\vph_{\tau,\eps}^+(x)\leq \frac{4}{\eps}$ in \eqref{equ pgt5/6 vph bound} and the fact that the function $\frac{x}{2\sinh\frac{x}{2}}$ is decreasing on $(0,\infty)$, we have
\begin{eqnarray}\label{equ 5/6 lower Selberg LHS}
	&&\sum_{k=1}^\infty\sum_{\gamma\in\sP(X_g)} \frac{\ell(\gamma)}{2\sinh\frac{k\ell(\gamma)}{2}} \vph_{\tau,\eps}^+(k\ell(\gamma)) \\
	&\leq& \frac{\tau-\eps}{2\sinh\frac{\tau-\eps}{2}} \sum_{\tiny\begin{array}{c} \gamma\in\sP(X_g), \\
	\ell(\gamma)\in(\tau-\eps,\tau+\eps) \end{array}} \vph_{\tau,\eps}^+(\ell(\gamma))  \nonumber\\
	&&+\frac{4}{\eps}\sum_{k=2}^\infty 
	\sum_{\tiny\begin{array}{c} \gamma\in\sP(X_g), \\
	k\ell(\gamma)\in(\tau-\eps,\tau+\eps) \end{array}} \frac{\ell(\gamma)}{2\sinh\frac{k\ell(\gamma)}{2}} + \frac{4}{\eps}\sum_{k=1}^\infty\ \sum_{\tiny\begin{array}{c} \gamma\in\sP(X_g), \\
	k\ell(\gamma)\leq\eps \end{array}} \frac{\ell(\gamma)}{2\sinh\frac{k\ell(\gamma)}{2}} \nonumber\\
	&\leq& \frac{\tau-\eps}{2\sinh\frac{\tau-\eps}{2}} \sum_{\tiny\begin{array}{c} \gamma\in\sP(X_g), \\
	1<\ell(\gamma)\leq L \end{array}} \vph_{\tau,\eps}^+(\ell(\gamma))  \nonumber\\
	&&+\frac{4}{\eps}\sum_{k=2}^\infty 
	\sum_{\tiny\begin{array}{c} \gamma\in\sP(X_g), \\
	k\ell(\gamma)\in(\tau-\eps,\tau+\eps) \end{array}} \frac{\ell(\gamma)}{2\sinh\frac{k\ell(\gamma)}{2}} + \frac{4}{\eps}\sum_{k=1}^\infty\ \sum_{\tiny\begin{array}{c} \gamma\in\sP(X_g), \\
	k\ell(\gamma)\leq\eps \end{array}} \frac{\ell(\gamma)}{2\sinh\frac{k\ell(\gamma)}{2}}. \nonumber
\end{eqnarray}
With the same argument as \eqref{equ 5/6 upper Selberg LHS'}, 
\begin{equation}\label{equ 5/6 lower Selberg LHS'}
	\frac{4}{\eps}\sum_{k=1}^\infty \  \sum_{\tiny\begin{array}{c} \gamma\in\sP(X_g), \\
	k\ell(\gamma)\leq\eps \end{array}} \frac{\ell(\gamma)}{2\sinh\frac{k\ell(\gamma)}{2}}
	\prec \frac{g}{\eps} + \frac{1}{\eps} \sum_{\tiny\begin{array}{c} \gamma\in\sP(X_g), \\
	\ell(\gamma)<\eps \end{array}} \ln\frac{\eps}{\ell(\gamma)}.
\end{equation}

By Corollary \ref{thm count ge^L upp}, 
\begin{equation}\label{equ 5/6 lower Selberg LHS''}
\begin{aligned}
	&\frac{4}{\eps}\sum_{k=2}^\infty \ \sum_{\tiny\begin{array}{c} \gamma\in\sP(X_g), \\
	k\ell(\gamma)\in(\tau-\eps,\tau+\eps) \end{array}} \frac{\ell(\gamma)}{2\sinh\frac{k\ell(\gamma)}{2}} \leq \frac{4}{\eps} \frac{1}{2\sinh\frac{\tau-\eps}{2}}
	\sum_{\tiny\begin{array}{c} \gamma\in\sP(X_g), \\
	\ell(\gamma)<\frac{\tau+\eps}{2} \end{array}}  \sum_{k=\lceil\frac{\tau-\eps}{\ell(\gamma)}\rceil}^{[\frac{\tau+\eps}{\ell(\gamma)}]} \ell(\gamma)\\ &\prec\frac{1}{\eps} e^{-\frac{\tau}{2}} \sum_{\tiny\begin{array}{c} \gamma\in\sP(X_g), 
	\ell(\gamma)<\frac{\tau+\eps}{2} \end{array}}   \left(1+\frac{2\eps}{\ell(\gamma)}\right)\ell(\gamma) \prec  \frac{1}{\eps} e^{-\frac{\tau}{2}}  \left((g-1)e^{\frac{\tau+\eps}{2}+7}\right) \tau \\ 
& \prec g\frac{\tau}{\eps}.
\end{aligned}
\end{equation}

Since $\widehat{\vph_{\tau,\eps}^+}(r)\geq 0$ for $r\in\R$, we have
\begin{equation}\label{equ 5/6 lower Selberg RHS1}
	\sum_{\lambda_j>\frac{1}{4}} \widehat{\vph_{\tau,\eps}^+}(r_j)\geq 0.
\end{equation}

For each small eigenvalue $\lambda_j\in [0,\frac{1}{4}]$,
\begin{eqnarray}
	\widehat{\vph_{\tau,\eps}^+}(r_j) &=& \left(\cosh(\tau\sqrt{\frac{1}{4}-\lambda_j})+1\right)\left(\widehat{\eta}({\bf i} \eps\sqrt{\frac{1}{4}-\lambda_j})\right)^2 \nonumber\\
	&\geq& \frac{1}{2}e^{(s_j-\frac{1}{2})\tau} \left(1+O_\eta(\eps^2)\right) \nonumber
\end{eqnarray}
where $\widehat{\eta}({\bf i}\eps\sqrt{\frac{1}{4}-\lambda_j})=1+O_\eta(\eps^2)$ since $\widehat{\eta}(0)=1$ and $\widehat{\eta}$ is even. And hence together with Theorem \ref{thm OR09}, 
\begin{equation}\label{equ 5/6 lower Selberg RHS2}
\sum_{0\leq\lambda_j\leq\frac{1}{4}} \widehat{\vph_{\tau,\eps}^+}(r_j) \geq \frac{1}{2}\sum_{0\leq\lambda_j\leq\frac{1}{4}} e^{(s_j-\frac{1}{2})\tau} - O_\eta\left(g\eps^2 e^{\frac{1}{2}\tau}\right).
\end{equation}

For the integral term, similar as \eqref{equ 5/6 upper Selberg RHS3} we have
\begin{equation}\label{equ 5/6 lower Selberg RHS3}
\begin{aligned}
	&\left|(g-1)\int_\R r\tanh(\pi r)\widehat{\vph_{\tau,\eps}^+}(r) dr\right| \\
	&= \left|2(g-1)\int_0^\infty r\tanh(\pi r)(\cos(r\tau)+1)(\widehat{\eta}(\eps r))^2 dr\right| \\
	&\prec_\eta \frac{g}{\eps^2}.
\end{aligned}
\end{equation}

Then plug \eqref{equ 5/6 lower Selberg LHS}, \eqref{equ 5/6 lower Selberg LHS'}, \eqref{equ 5/6 lower Selberg LHS''}, \eqref{equ 5/6 lower Selberg RHS1}, \eqref{equ 5/6 lower Selberg RHS2} and \eqref{equ 5/6 lower Selberg RHS3} into the trace formula \eqref{equ pgt5/6 Selberg}, we have
\begin{eqnarray}
	&&\frac{\tau-\eps}{2\sinh\frac{\tau-\eps}{2}} \sum_{\tiny\begin{array}{c} \gamma\in\sP(X_g), \\
	1<\ell(\gamma)\leq L \end{array}} \vph_{\tau,\eps}^+(\ell(\gamma)) \nonumber\\
	&\geq& \frac{1}{2}\sum_{0\leq \lambda_j\leq\frac{1}{4}} e^{(s_j-\frac{1}{2})\tau} - O_\eta\left(g\eps^2e^{\frac{1}{2}\tau} +\frac{g}{\eps^2} +g\frac{\tau}{\eps} + \frac{1}{\eps} \sum_{\tiny\begin{array}{c} \gamma\in\sP(X_g), \\
	\ell(\gamma)<\eps \end{array}} \ln\frac{\eps}{\ell(\gamma)}\right). \nonumber
\end{eqnarray}
Since $\frac{2\sinh\frac{\tau-\eps}{2}}{\tau-\eps} \geq \frac{e^{\tau/2}}{\tau}(1-O(\eps+e^{-\tau}))$, combining with Theorem \ref{thm OR09} we get
\begin{eqnarray}
	&& \sum_{\tiny\begin{array}{c} \gamma\in\sP(X_g), \\
	1<\ell(\gamma)\leq L \end{array}} \vph_{\tau,\eps}^+(\ell(\gamma)) \geq \frac{1}{2}\sum_{0\leq \lambda_j\leq\frac{1}{4}} \frac{e^{s_j\tau}}{\tau} \nonumber \\
&&- O_\eta\left(g\eps\frac{e^{\tau}}{\tau}+g\frac{(1+\eps\tau)}{\eps^2}\frac{e^{\frac{1}{2}\tau}}{\tau}+\frac{1}{\eps}\frac{e^{\frac{1}{2}\tau}}{\tau} \sum_{\tiny\begin{array}{c} \gamma\in\sP(X_g), \\
	\ell(\gamma)<\eps \end{array}} \ln\frac{\eps}{\ell(\gamma)}\right) \nonumber
\end{eqnarray}
as desired.
\end{proof}

Now we are ready to prove Theorem \ref{thm pgt5/6-r}.

\begin{proof}[Proof of Theorem \ref{thm pgt5/6-r}]
Lemma \ref{thm pgt5/6 count lemma} and \ref{thm pgt5/6 vph-<} give the upper bound:
\begin{eqnarray}
&&\#\left\{\begin{array}{c}
		\gamma\in\sP(X_g), \\
		1<\ell(\gamma)\leq L
	\end{array}\right\} 
= 2\int_{1-\eps}^{L+\eps} 
\sum_{\tiny\begin{array}{c} \gamma\in\sP(X_g), \\
1<\ell(\gamma)\leq L \end{array}} \vph_{\tau,\eps}^-(\ell(\gamma)) d\tau \nonumber\\
&\leq& \sum_{0\leq \lambda_j\leq\frac{1}{4}}\int_{1-\eps}^{L+\eps} \frac{e^{s_j \tau}}{\tau}d\tau +O_\eta\left(g\eps e^{\frac{L}{2}} +\frac{g}{\eps^2} +\frac{1}{\eps} \sum_{\tiny\begin{array}{c} \gamma\in\sP(X_g), \\
\ell(\gamma)<\eps \end{array}} \ln\frac{\eps}{\ell(\gamma)}\right)\cdot\frac{e^{\frac{L}{2}}}{L} \nonumber\\
&=& \sum_{0\leq \lambda_j\leq\frac{1}{4}} \Li(e^{s_j (L+\eps)})+ O_\eta\left(g\eps e^{\frac{L}{2}}+\frac{g}{\eps^2}+\frac{1}{\eps} \sum_{\tiny\begin{array}{c} \gamma\in\sP(X_g), \\
\ell(\gamma)<\eps \end{array}} \ln\frac{\eps}{\ell(\gamma)}\right)\cdot\frac{e^{\frac{L}{2}}}{L} \nonumber\\
&=& \sum_{0\leq \lambda_j\leq\frac{1}{4}} \Li(e^{s_j L})+ O_\eta\left(g\eps e^{\frac{L}{2}}+\frac{g}{\eps^2}+\frac{1}{\eps} \sum_{\tiny\begin{array}{c} \gamma\in\sP(X_g), \\
\ell(\gamma)<\eps \end{array}} \ln\frac{\eps}{\ell(\gamma)}\right)\cdot\frac{e^{\frac{L}{2}}}{L} \nonumber
\end{eqnarray}
where in the last equality we apply $ \Li(e^{s_j (L+\eps)})=\Li(e^{s_j L})+O\left(\eps\frac{e^L}{L} \right)$. 

For the other direction, Lemma \ref{thm pgt5/6 count lemma} and \ref{thm pgt5/6 vph+>} give the lower bound:
\begin{eqnarray}
&&\#\left\{\begin{array}{c}
		\gamma\in\sP(X_g),\\
		1<\ell(\gamma)\leq L
	\end{array}\right\} \nonumber
= 2\int_{1-\eps}^{L+\eps} \sum_{\tiny\begin{array}{c} \gamma\in\sP(X_g), \\
1<\ell(\gamma)\leq L \end{array}} \vph_{\tau,\eps}^+(\ell(\gamma)) d\tau \nonumber\\
&&\geq 2\int_{1+\eps}^{L-\eps} \sum_{\tiny\begin{array}{c} \gamma\in\sP(X_g), \\
1<\ell(\gamma)\leq L \end{array}} \vph_{\tau,\eps}^+(\ell(\gamma)) d\tau \nonumber \geq \sum_{0\leq \lambda_j\leq\frac{1}{4}}\int_{1+\eps}^{L-\eps} \frac{e^{s_j \tau}}{\tau}d\tau \nonumber \\
&& \quad -O_\eta\left(g\eps e^{\frac{L}{2}}+g\frac{1+\eps L}{\eps^2}+\frac{1}{\eps} \sum_{\tiny\begin{array}{c} \gamma\in\sP(X_g), \\
\ell(\gamma)<\eps \end{array}} \ln\frac{\eps}{\ell(\gamma)}\right)\cdot\frac{e^{\frac{L}{2}}}{L} \nonumber\\
&&= \sum_{0\leq \lambda_j\leq\frac{1}{4}} \Li(e^{s_j(L-\eps)}) \nonumber\\
&&  \quad - O_\eta\left(g\eps e^{\frac{L}{2}}+g\frac{1+\eps L}{\eps^2}+\frac{1}{\eps} \sum_{\tiny\begin{array}{c} \gamma\in\sP(X_g), \\
\ell(\gamma)<\eps \end{array}} \ln\frac{\eps}{\ell(\gamma)}\right)\cdot\frac{e^{\frac{L}{2}}}{L} \nonumber\\
&&= \sum_{0\leq \lambda_j\leq\frac{1}{4}} \Li(e^{s_j L}) - O_\eta\left(g\eps e^{\frac{L}{2}}+g\frac{1+\eps L}{\eps^2}+\frac{1}{\eps} \sum_{\tiny\begin{array}{c} \gamma\in\sP(X_g), \\
\ell(\gamma)<\eps \end{array}} \ln\frac{\eps}{\ell(\gamma)}\right)\cdot\frac{e^{\frac{L}{2}}}{L} \nonumber
\end{eqnarray}
where in the last equality we apply $ \Li(e^{s_j (L-\eps)})=\Li(e^{s_j L})+O\left(\eps\frac{e^L}{L} \right)$. For sufficiently large $L$, take $\eps=e^{-\frac{1}{6}L}<0.01$ and fix a function $\eta$, we have
\begin{eqnarray*}
&&\#\left\{\begin{array}{c}
		 \gamma\in\sP(X_g),\\
		1<\ell(\gamma)\leq L
\end{array}\right\}
= \sum_{0\leq \lambda_j\leq\frac{1}{4}} \Li(e^{s_j L}) \nonumber\\
	&&+O\left(g\frac{e^{\frac{5}{6}L}}{L}+ \frac{e^{\frac{2}{3}L}}{L} \sum_{\tiny\begin{array}{c} \gamma\in\sP(X_g), \\
	\ell(\gamma)<e^{-L/6} \end{array}} \ln\frac{1}{\ell(\gamma)e^{\frac{1}{6}L}}\right).
\end{eqnarray*}

Then the proof is completed by taking $t=e^L$.
\end{proof}


\section{Expectation of the error term in PGT}\label{sec-expect}
In this section we study the expected value of $\pi_{X_{g}}(t)$ over $\sM_g$ for large genus. We will apply Theorem \ref{thm pgt5/6-r} and our previous work \cite{WX22-GAFA} to complete the proof of Theorem \ref{thm pgt E}.

We first see the following consequence of Theorem \ref{thm pgt5/6-r}.
\begin{corollary}\label{thm pgt E cor}
For any $t>2$, 
$$\E\left[\ \left|\pi_{X_g}(t) - \Li(t)-\sum_{0<\lambda_j\leq\frac{1}{4}}\Li(t^{s_j})\right|\ \right]\prec g\frac{t^\frac{5}{6}}{\ln t}.$$
\end{corollary}

\bp
By Theorem \ref{thm collar} we know that there are at most $6g-6$ oriented primitive closed geodesics with length $\leq 1$. So for $t>2$, we have
$$\sum_{\tiny\begin{array}{c} \gamma\in\sP(X_g), \\
\ell(\gamma)<t^{-1/6} \end{array}} \ln\left(\frac{1}{\ell(\gamma)t^{\frac{1}{6}}}\right) \prec \max\left\{0,g\cdot\ln\left(\frac{1}{\sys(X_g)t^{\frac{1}{6}}}\right)\right\} \leq \frac{g}{\sys(X_g)t^{\frac{1}{6}}}.$$ It was shown by Mirzakhani in \cite[Corollary 4.3]{Mirz13} that 
\begin{equation*}
\E\left[\frac{1}{\sys(X_g)}\right] \asymp 1.
\end{equation*}

Then the conclusion clearly follows by Theorem \ref{thm pgt5/6-r} and the two estimates above.
\ep

To prove Theorem \ref{thm pgt E}, in light of Corollary \ref{thm pgt E cor} we need to estimate the term $$\E\left[\sum_{0<\lambda_j\leq\frac{1}{4}}\Li(t^{s_j})\right].$$
We will apply our previous work \cite{WX22-GAFA} to control the term above. In \cite{WX22-GAFA} we proved 
\begin{theorem}\label{thm lambda1 3/16}
For any $\eps>0$, 
$$\limg\Prob\left(X_g\in\M_g;\ \lambda_1(X_g)>\frac{3}{16}-\eps\right)=1.$$
\end{theorem}
\noindent One can also see an independent proof by Lipnowski-Wright in \cite{LW21}. The first positive lower bound $0.0024$ was established by Mirzakhani in \cite{Mirz13}. Very recently, Anantharaman-Monk in \cite{AM23} can improve the lower bound to $\frac{2}{9}-\eps$.  One may also see related results for random covers of a fixed closed hyperbolic surface in \cite{MNP20} by Magee-Naud-Puder; and see a recent breakthrough by Hide-Magee \cite{HM21} for optimal relative spectral gap of random covers of a fixed cusped hyperbolic surface. 

Let $\phi_T$ be the test function on \cite[Page 355]{WX22-GAFA}. Its Fourier transform satisfies
\be \label{above-0}
\widehat{\phi_{T}}(z)\geq 0 \text{ when $z\in \R\cup {\bf i}\cdot \R$.}
\ene
The proof of \cite[Lemma 22]{WX22-GAFA} also gives that for any $j$ with $\lambda_j(X_g)<\frac{1}{4}$ and any $0<\epsilon<1$ there exists a constant $C_\epsilon>0$ only depending on $\epsilon$ such that 
\be\label{lb-tf-j}
\widehat{\phi_{4\ln g}}(r_j(X_g))\geq  C_\epsilon g^{4(1-\epsilon)\sqrt{\frac{1}{4}-\lambda_j(X_g)}}\ln g.
\ene
For the proof of Theorem \ref{thm lambda1 3/16}, the essential part is the following estimation.
\bt \label{key-ine}
Let $\phi_{T}$ be the test function on \cite[Page 355]{WX22-GAFA}. Then for any $\eps_1>0$, there exists a constant $C(\eps_1)>0$ only depending on $\eps_1$ such that for all $g>C(\eps_1)$,
\begin{equation*}
	\E\left[\sum_{0<\lambda_j\leq\frac{1}{4}} \widehat{\phi_{4\ln g}}(r_j(X_g))\right]\leq g^{1+4\eps_1}(\ln g)^{68}.
\end{equation*}
\et

\bp
This is exactly \cite[Equation (56)]{WX22-GAFA}.
\ep
 
Now we apply Theorem \ref{key-ine} to prove the following result which is slightly stronger than Theorem \ref{thm lambda1 3/16} and essential in the proof of Theorem \ref{thm pgt E}.
\begin{theorem}\label{thm lambda_j >}
For any $\eps>0$, $0\leq r<\frac{1}{2}$ and $1\leq j\leq 2g-3$, there exists a constant $c(\eps)>0$ independent of $g$, $r$ and $j$ such that for $g>0$ large enough,
\begin{equation*}
	\Prob\left(X_g\in\M_g;\ \lambda_j(X_g)\leq \frac{1}{4}-r^2\right) \leq c(\epsilon)\frac{1}{j} g^{1-4r+\epsilon}.
\end{equation*}
\end{theorem}

\begin{proof}
By \eqref{above-0} and \eqref{lb-tf-j} we know that for any $0<\epsilon<1$ and any $0\leq r <\frac{1}{2}$ there exists a constant $C_\epsilon>0$ only depending on $\epsilon$ such that
\begin{eqnarray*}
	&&\E\left[\sum_{0<\lambda_j\leq\frac{1}{4}} \widehat{\phi_{4\ln g}}(r_j(X_g))\right]\\
	&&\geq \Prob\left(X_g\in\M_g;\ \lambda_j(X_g)\leq \frac{1}{4}-r^2\right)\cdot j\cdot C_\epsilon g^{4(1-\epsilon)r}\ln g.
\end{eqnarray*}
Which together with Theorem \ref{key-ine} implies that for any $\eps_1>0$ and $g>0$ large enough we have
\[\Prob\left(X_g\in\M_g;\ \lambda_j(X_g)\leq \frac{1}{4}-r^2\right)\le \frac{(\ln g)^{67}}{j\cdot C_\epsilon}\cdot g^{1-4r+4(\eps_1+r\cdot \eps)}.\] 

Then the conclusion follows by choosing suitable $\eps, \eps_1$ and $c(\eps)$.
\end{proof}

\begin{rem*}
If we choose $j=1$, $r=\frac{1}{4}+\eps$, by letting $g\to \infty$ Theorem \ref{thm lambda_j >} implies Theorem \ref{thm lambda1 3/16}.
\end{rem*}

Now we apply Theorem \ref{thm lambda_j >} to estimate $\E\left[\sum_{0<\lambda_j\leq\frac{1}{4}}\Li(t^{s_j})\right]$ for large $t=t(g)$. More precisely,
\bt\label{bound-mid}
For any $\eps>0$ and $t>2g^4$, we have
\[\E\left[\sum_{0<\lambda_j\leq\frac{1}{4}} \Li(t^{s_j})\right]\prec_\eps g^{-1+3\eps} \Li(t).\]
\et

\bp
For a non-negative random variable $\xi(X)$, the expectation is 
\begin{eqnarray*}
\E[\xi]&=&\frac{1}{V_g} \int_{\M_g} \xi(X)dX \\
&=& \frac{1}{V_g} \int_{\M_g} \int_0^\infty \mathbf{1}_{\{\xi(X)\geq t\}}dt\ dX \\
&=& \int_0^\infty \Prob(\xi(X)\geq t) dt.
\end{eqnarray*}
So 
\begin{eqnarray*}
\E\left[\Li(t^{s_j}) \mathbf{1}_{\{\frac{1}{2}\leq s_j<1\}}\right]
&=& \int_0^\infty \Prob\left(\Li(t^{s_j}) \mathbf{1}_{\{\frac{1}{2}\leq s_j<1\}} \geq \rho\right) d\rho \\
&\leq& \Li(t^{\frac{1}{2}}) + \int_{\Li(t^{\frac{1}{2}})}^{\Li(t)} \Prob\left(\Li(t^{s_j}) \mathbf{1}_{\{\frac{1}{2}\leq s_j<1\}} \geq \rho\right) d\rho \\
&=& \Li(t^{\frac{1}{2}}) + \int_{\frac{1}{2}}^1 \Prob\left(s\leq s_j<1 \right) \left(\frac{d}{ds}\Li(t^s)\right) ds \\
&=& \Li(t^{\frac{1}{2}}) + \int_{\frac{1}{2}}^1 \frac{t^s}{s} \Prob\left(s_j\geq s \right) ds
\end{eqnarray*}
where in the last equation we apply $\frac{d}{ds}\Li(t^s)=\frac{d}{ds}\left(\int_2^{t^s}\frac{1}{\ln x}dx  \right)=\frac{t^s}{s}$. By Theorem \ref{thm lambda_j >}, for any $\eps>0$, $\frac{1}{2}\leq s<1$ and $1\leq j\leq 2g-3$, we have 
\begin{equation*}
	\Prob\left(X_g\in\M_g;\ s_j(X_g)\geq s\right) \leq c(\epsilon)\frac{1}{j} g^{3-4s+\epsilon}.
\end{equation*}
Then it follows by Theorem \ref{thm OR09} that for $t>2g^{4}$, 
\begin{eqnarray*}
\E\left[\sum_{0<\lambda_j\leq\frac{1}{4}} \Li(t^{s_j})\right]
&=& \sum_{j=1}^{2g-3} \E\left[\Li(t^{s_j}) \mathbf{1}_{\{\frac{1}{2}\leq s_j<1\}}\right]\\
&\leq& (2g-3)\Li(t^{\frac{1}{2}}) + \sum_{j=1}^{2g-3} \int_{\frac{1}{2}}^1 \frac{t^s}{s} c(\eps)\frac{1}{j}g^{3-4s+\eps} ds\\
&\asymp& g\Li(t^{\frac{1}{2}})+c(\eps)\cdot g^{3-\eps} \cdot \left(\sum_{j=1}^{2g-3}\frac{1}{j}\right)\cdot \int_{\frac{1}{2}}^1 \left(\frac{t}{g^4}\right)^sds\\
&\prec& g\Li(t^{\frac{1}{2}}) + c(\eps) g^{-1+\eps} \ln(g) \frac{t}{\ln t -4\ln g}\\
&\prec_\eps& g\Li(t^{\frac{1}{2}})+ g^{-1+2\eps} \Li(t)\ln(g)\\
&\prec_\eps& g^{-1+3\eps} \Li(t).
\end{eqnarray*}

The proof is complete.
\ep

Now we complete the proof of Theorem \ref{thm pgt E}.
\begin{proof}[Proof of Theorem \ref{thm pgt E}]
By triangle inequality and Corollary \ref{thm pgt E cor} we know that for $t>g^{12}$,
\[\E\left[\left|\pi_{X_g}(t) - \Li(t)\right|\right]\prec \left(\E\left[\sum_{0<\lambda_j\leq\frac{1}{4}} \Li(t^{s_j})\right]+ \frac{g}{t^{\frac{1}{6}}}\cdot \Li(t)\right)\]
which together with Theorem \ref{bound-mid} implies that for any $\eps>0$,
\[\E\left[\ \left|1-\frac{\pi_{X_g}(t)}{\Li(t)}\right|\ \right]\prec_\eps \left(\frac{1}{g^{1-3\eps}}+ \frac{g}{t^{\frac{1}{6}}}\right).\]

Then the conclusion clearly follows because $t>g^{12}$.
\end{proof}


\section{PGT for random hyperbolic surfaces}\label{sec-pgt-rs}

In this section, our main goal is to show that
\begin{theorem}\label{thm pgt 3/4 rdm}
	There exists a subset $\mathcal{A}_g\sbs \M_g$ with $\lim\limits_{g\to\infty}\Prob(\mathcal{A}_g)=1$ such that for any $X_g\in \mathcal{A}_g$ and $t=t(g)>2$ which may depend on $g$, 
	$$\pi_{X_g}(t)= \Li(t)+\sum_{0<\lambda_j\leq \frac{1}{4}}\Li(t^{s_j}) + O\left(g\frac{t^{\frac{3}{4}}}{\ln t} \right)$$
	where the implied constant is independent of $t$, $g$ and $X_g$.
\end{theorem}

Theorem \ref{mt-1} is as a direct consequence.
\bp[Proof of Theorem \ref{mt-1}]
Since $\ln t>\ln 2$, it clearly follows by Theorem \ref{thm OR09}, \ref{thm lambda1 3/16} and \ref{thm pgt 3/4 rdm}.
\ep

For the proof of Theorem \ref{thm pgt 3/4 rdm}, we split into two parts. In the first subsection, with an assumption on Weyl's law which is based on the work of Monk \cite{Monk20}, we will show a Prime Geodesic Theorem with a complicated remainder. Then in the second subsection we will simplify the remainder to the form as shown in Theorem \ref{thm pgt 3/4 rdm} for random hyperbolic surfaces.

\subsection{First part of Theorem \ref{thm pgt 3/4 rdm}}

We follow the strategy by Randol in \cite[Chapter XI]{Chavel}.

For any $0\leq a\leq b$, denote $N_{X_g}^\Delta(a,b)$ to be the number of eigenvalues of $X_g$ in the interval $[a,b]$. The standard Weyl's law says that 
$$\lim_{b\to\infty}\frac{N_{X_g}^\Delta(0,b)}{(g-1)b} =1$$
for any fixed hyperbolic surface $X_g\in\M_g$. Monk \cite{Monk20} gave a bound of the remainder of Weyl's law for random hyperbolic surfaces.
\begin{theorem}[{\cite[Theorem 5]{Monk20}}]\label{thm Monk}
There exists a universal constant $C>0$ and a subset $\mathcal{B}_g\sbs \M_g$ such that $1-\Prob(\mathcal{B}_g) = O\left(g^{-\frac{1}{12}}(\ln g)^{\frac{9}{8}}\right)$. And for any large enough $g$, any $0\leq a\leq b$ and any surface $X_g\in \mathcal{B}_g$, one can write 
$$\frac{N_{X_g}^\Delta(a,b)}{4\pi(g-1)} = \frac{1}{4\pi}\int_{\frac{1}{4}}^{\infty} \mathbf{1}_{[a,b]}(\lambda) \tanh\left(\pi\sqrt{\lambda-\frac{1}{4}}\right)d\lambda + R(X_g,a,b)$$
where
$$-C\sqrt{\frac{b+1}{\ln g}} \leq R(X_g,a,b) \leq C\sqrt{\frac{b+1}{\ln g}} \left(\ln\left(2+(b-a)\sqrt{\frac{\ln g}{b+1}} \right)\right)^{\frac{1}{2}}.$$
\end{theorem}

Now we make the following definition.

\begin{definition*}
For a constant $A\geq 1$, we say that a hyperbolic surface $X_g\in \M_g$ satisfies the \emph{condition $\sP_A$} if for any $T>0$, 
\begin{equation}\label{equ cond P_A}
\#\left\{k;\ \frac{1}{4}<\lambda_k(X_g)\leq \frac{1}{4}+T\right\} \leq A g (1+T).
\end{equation}
\end{definition*}

\noindent If we choose $a=\frac{1}{4}$ and $b=\frac{1}{4}+T$ in Theorem \ref{thm Monk}, since
\begin{equation*}
	(g-1)\int_{\frac{1}{4}}^{\frac{1}{4}+T} \tanh\left(\pi\sqrt{\lambda-\tfrac{1}{4}}\right)d\lambda \leq (g-1)T
\end{equation*}
and
\begin{eqnarray*}
 	4\pi(g-1)R(X_g,\tfrac{1}{4},\tfrac{1}{4}+T) 
 	&\leq& 4C\pi (g-1)\sqrt{T+\frac{5}{4}} \sqrt{\frac{\ln(2+T\sqrt{\ln g})}{\ln g}} \\
 	&\prec& (g-1)\sqrt{(1+T)\ln(2+T)}, 
\end{eqnarray*}
we know that there exists $A\geq 1$ such that as $g\to \infty$,   the condition $\sP_A$ holds for a generic closed hyperbolic surface $X_g\in\M_g$.

In this subsection we aim to show that
\begin{theorem}\label{thm pgt 3/4}
If $X_g\in \M_g$ satisfies the condition $\sP_A$ for some $A\geq 1$, then for any $t>2$,
\begin{equation*}
\pi_{X_g}(t) = \Li(t)+\sum_{0<\lambda_j\leq \frac{1}{4}}\Li(t^{s_j}) + O\left( g\sqrt{A}\frac{t^{\frac{3}{4}}}{\ln t}+ \sum_{\tiny\begin{array}{c} \gamma\in\sP(X_g), \\
\ell(\gamma)<1 \end{array}} \ln\left(\frac{1}{\ell(\gamma)}\right)\right)
\end{equation*}
where the implied constant is independent of $t$, $A$ and $X_g$.
\end{theorem}

For $X_g\in\M_g$ and $T>0$, denote 
\begin{equation}
\begin{aligned}
H(T)=H(X_g,T) &:= \sum_{\ell(\gamma)\leq T} \frac{\Lambda(\gamma)}{2\sinh(\frac{1}{2}\ell(\gamma))} 2\cosh(\frac{1}{2}\ell(\gamma)) \\
&= \sum_{\ell(\gamma)\leq T} \Lambda(\gamma)\frac{1+e^{-\ell(\gamma)}}{1-e^{-\ell(\gamma)}} 
\end{aligned}
\end{equation}
where the sum is taken over all oriented closed geodesics in $X_g$ with length $\leq T$, the number $\Lambda(\gamma)=\ell(\gamma_0)$ where $\gamma_0$ is the unique oriented primitive closed geodesic such that $\gamma$ is an iterate of $\gamma_0$. Let 
\begin{equation}
\vph_T(x) := 2\cosh(\frac{x}{2}) \cdot \mathbf{1}_{[-T,T]}(x)
\end{equation}
where $\mathbf{1}_{[-T,T]}(x)$ is the characteristic function on $[-T,T]$. Since Selberg's trace formula only admits smooth functions, we apply a convolution to $\vph_T(x)$. Let $\eta(x)\geq 0$ be a non-negative smooth even function satisfying that $supp(\eta) \sbs (-1,1)$ and $\int_\R \eta(x)dx =1$. For any $0<\eps<1$, let
\begin{equation}
\eta_\eps (x) := \frac{1}{\eps} \eta(\frac{x}{\eps})
\end{equation}
and
\begin{equation}
\vph_T^\eps(x) := \vph_T * \eta_\eps (x) = \int_\R 2\cosh(\frac{x-y}{2}) \mathbf{1}_{[-T,T]}(x-y) \eta_\eps(y) dy.
\end{equation}
Then $\vph_T^\eps(x)$ is a smooth even function with $supp(\vph_T^\eps) \sbs (-T-\eps, T+\eps)$. Their Fourier transforms satisfy
\begin{equation}
\begin{aligned}
\widehat{\vph_T}(r) 
&= \int_{-T}^T e^{-{\bf i}rx} 2\cosh(\frac{x}{2}) dx \\
&= \frac{2\sinh(T(\frac{1}{2}+{\bf i}r))}{\frac{1}{2}+{\bf i}r} + \frac{2\sinh(T(\frac{1}{2}-{\bf i}r))}{\frac{1}{2}-{\bf i}r} \\
&= \frac{8}{1+4r^2} \left( 2r\sin(rT)\cosh(\frac{T}{2}) +\cos(rT)\sinh(\frac{T}{2}) \right)
\end{aligned}
\end{equation}
and
\begin{equation}
\widehat{\vph_T^\eps}(r) 
= \widehat{\vph_T}(r) \widehat{\eta_\eps}(r) 
= \widehat{\vph_T}(r) \widehat{\eta}(\eps r).
\end{equation}
Denote 
\begin{equation}
H_\eps(T) = H_\eps(X_g,T) := \sum_{\gamma} \frac{\Lambda(\gamma)}{2\sinh(\frac{1}{2}\ell(\gamma))} \vph_T^\eps(\ell(\gamma))
\end{equation}
where the sum is taken over all oriented closed geodesics in $X_g$.

\begin{proposition}\label{thm H_eps}
If $X_g\in\M_g$ satisfies the condition $\sP_A$ defined in \eqref{equ cond P_A} for some $A\geq 1$, then
$$H_\eps(T)= e^T + \sum_{0<\lambda_j\leq\frac{1}{4}} \frac{e^{s_j T}}{s_j}+ O_\eta\left(Ag \frac{1}{\eps} e^{\frac{1}{2}T} +\eps^2 g e^{T}\right).$$
\end{proposition}

\begin{proof}
Plug $\vph_T^\eps$ into Selberg's trace formula, i.e., Theorem \ref{thm trace formula}, we have 
\begin{equation}\label{equ H_eps}
H_\eps(T) = \sum_{0\leq\lambda_j \leq \frac{1}{4}} \widehat{\vph_T^\eps}(r_j) + \sum_{\lambda_j > \frac{1}{4}} \widehat{\vph_T^\eps}(r_j) - (g-1)\int_\R r\widehat{\vph_T^\eps}(r)\tanh(\pi r) dr.
\end{equation}

Now we estimate these three terms separately.

First we estimate the term $ \sum_{0\leq\lambda_j \leq \frac{1}{4}} \widehat{\vph_T^\eps}(r_j)$. Note that $\widehat{\eta}(r) = 1+ O_\eta(r^2)$ since $\int_\R \eta(x)dx=1$ and $\widehat{\eta}$ is even. For $0\leq \lambda_j\leq \frac{1}{4}$ we have 

\begin{equation}
\begin{aligned}
&\widehat{\vph_T^\eps}(r_j) 
= \widehat{\vph_T}(r_j) \widehat{\eta}(\eps r_j)   \\
&= \widehat{\eta}(\eps r_j) \left( \frac{2\sinh(T(\frac{1}{2}+\sqrt{\frac{1}{4}-\lambda_j}))}{\frac{1}{2}+\sqrt{\frac{1}{4}-\lambda_j}} + \frac{2\sinh(T(\frac{1}{2}-\sqrt{\frac{1}{4}-\lambda_j}))}{\frac{1}{2}-\sqrt{\frac{1}{4}-\lambda_j}} \right)\\
&= \left(1+O_\eta(\eps^2)\right) \left( \frac{e^{s_j T}}{s_j} + O(e^{\frac{1}{2}T}) \right)  \\
&= \frac{e^{s_j T}}{s_j} + O_\eta\left(e^{\frac{1}{2}T} +\eps^2 e^{s_j T} \right).
\end{aligned}
\end{equation}
By Theorem \ref{thm OR09}, there are at most $2g-2$ eigenvalues in $[0,\frac{1}{4}]$. So 
\begin{equation}\label{equ H_eps term1}
\sum_{0\leq\lambda_j \leq \frac{1}{4}} \widehat{\vph_T^\eps}(r_j) = e^T + \sum_{0<\lambda_j\leq\frac{1}{4}} \frac{e^{s_j T}}{s_j} + O_\eta\left(g e^{\frac{1}{2}T} +\eps^2 g e^{T}\right).
\end{equation}
\

Next we estimate the integral term $(g-1)\int_\R r\widehat{\vph_T^\eps}(r)\tanh(\pi r) dr$. For real $r$, since $|\widehat{\vph_T}(r)| = O(\frac{1}{1+|r|}e^{\frac{1}{2}T})$ and $|\widehat{\eta}(\eps r)|=O_\eta(\frac{1}{1+|\eps r|^2})$, we have
\begin{equation}\label{equ H_eps term3}
\begin{aligned}
\left|(g-1)\int_\R r\widehat{\vph_T^\eps}(r)\tanh(\pi r) dr\right| 
&\prec_\eta g e^{\frac{1}{2}T} \int_\R \frac{1}{1+|\eps r|^2} dr \\
&\prec g\frac{1}{\eps}e^{\frac{1}{2}T}.
\end{aligned}
\end{equation}
\

For the term $\sum_{\lambda_j > \frac{1}{4}} \widehat{\vph_T^\eps}(r_j)$ related to large eigenvalues, we first denote $N(r)=\#\left\{j;\ \frac{1}{4}<\lambda_j(X_g)\leq \frac{1}{4}+r^2\right\}$ and $dN(r)= \sum_{\lambda_j>\frac{1}{4}} \delta_{r_j}(r)dr$ to be a discrete measure where $r_j=\sqrt{\lambda_j -\frac{1}{4}}$ and $\delta(\cdot)$ is the Dirac function. Then 
\begin{equation*}
\sum_{\lambda_j > \frac{1}{4}} \widehat{\vph_T^\eps}(r_j) = \int_0^\infty \widehat{\vph_T^\eps}(r) dN(r).
\end{equation*}
Since $X_g$ satisfies the condition $\sP_A$,
\[N(r)\leq A\cdot g \cdot (1+r^2).\]
So we have $N(0)=0$ and $\lim \limits_{r\to \infty}\frac{N(r)}{(1+r)\cdot (1+\eps^2r^2)}=0$. Then using Integration by Parts we have
\begin{eqnarray}
\left|\sum_{\lambda_j > \frac{1}{4}} \widehat{\vph_T^\eps}(r_j)\right| 
&\leq& \int_0^\infty |\widehat{\vph_T^\eps}(r)| dN(r) \nonumber\\
&\prec_\eta& \int_0^\infty \frac{e^{T/2}}{1+r} \frac{1}{1+|\eps r|^2} dN(r) \nonumber\\
&=& - e^{T/2} \int_0^\infty N(r) \frac{d}{dr}\left(\frac{1}{(1+r)(1+\eps^2 r^2)}\right) dr \nonumber\\
&=& e^{T/2} \int_0^\infty N(r) \left( \frac{1}{(1+r)^2(1+\eps^2 r^2)} + \frac{2\eps^2 r}{(1+r)(1+\eps^2 r^2)^2} \right) dr \nonumber\\
&\prec& e^{T/2} \int_0^\infty \frac{N(r)}{1+r^2} \frac{1}{1+\eps^2 r^2} dr. \nonumber
\end{eqnarray}
So if $X_g\in\M_g$ satisfies the condition $\sP_A$ defined in \eqref{equ cond P_A}, then
\begin{equation}\label{equ H_eps term2}
\left|\sum_{\lambda_j > \frac{1}{4}} \widehat{\vph_T^\eps}(r_j)\right| \prec_\eta Ag\frac{1}{\eps}e^{\frac{1}{2}T}.
\end{equation}

Then the conclusion follows by \eqref{equ H_eps}, \eqref{equ H_eps term1}, \eqref{equ H_eps term3}and \eqref{equ H_eps term2}.
\end{proof}

\begin{lemma}\label{thm vph_T}
For any $0<\eps<1$, $T>\eps$, let $\vph_T$ and $\vph_T^\eps$ be functions as above, then we have that for all $x>0$,
$$\frac{1}{\cosh \frac{\eps}{2}} \vph_{T-\eps}^\eps(x)\leq \vph_T(x) \leq \vph_{T+\eps}^{\eps}(x).$$
\end{lemma}
\begin{proof}
If $x>T$, the conclusion clearly follows because
\[\vph_{T-\eps}^\eps(x)= \vph_T(x) =0 \text{ and } \vph_{T+\eps}^{\eps}(x)\geq 0.\]

\noindent Now we assume $x\in[0,T]$. By definition 
$$\vph_T(x)=2\cosh\frac{x}{2}.$$
Since $\eta$ is even,
\begin{equation} \nonumber
\begin{aligned}
&\int_{-\eps}^\eps 2\cosh\frac{x-y}{2}\eta_\eps(y) dy \\
&=\int_0^\eps 2(\cosh\frac{x-y}{2}+\cosh\frac{x+y}{2})\eta_\eps(y) dy \\
&= 4\int_0^\eps \cosh\frac{x}{2}\cosh\frac{y}{2} \eta_\eps(y) dy \\
&= 2\cosh\frac{x}{2} \int_{-\eps}^\eps \cosh\frac{y}{2} \eta_\eps(y) dy.
\end{aligned}
\end{equation}
Then we have
\begin{eqnarray*}
\vph_{T+\eps}^{\eps}(x) &=& \int_{-\eps}^\eps 2\cosh\frac{x-y}{2}\eta_\eps(y) dy\\
& > &2\cosh\frac{x}{2}
\end{eqnarray*}
and
\begin{eqnarray*}
\vph_{T-\eps}^{\eps}(x) 
&\leq& \int_{-\eps}^\eps 2\cosh\frac{x-y}{2}\eta_\eps(y) dy\\
& <&2\cosh\frac{x}{2} \cosh\frac{\eps}{2}.
\end{eqnarray*}

The proof is complete.
\end{proof}

\begin{proposition}\label{thm H}
If $X_g\in\M_g$ satisfies the condition $\sP_A$ for some $A\geq 1$, then for $T>0$ large enough,
$$H(T)= e^T + \sum_{0<\lambda_j\leq\frac{1}{4}} \frac{e^{s_j T}}{s_j}+ O\left(g\sqrt{A} e^{\frac{3}{4}T}\right).$$
\end{proposition}
\begin{proof}
By the above Lemma \ref{thm vph_T}, 
$$\frac{1}{\cosh\frac{\eps}{2}} H_\eps(T-\eps) \leq H(T) \leq H_\eps(T+\eps).$$
Since $\frac{1}{\cosh(\eps/2)}=1+O(\eps^2)$ and $H_\eps(T)$ is increasing with respect to $T$, 
\begin{eqnarray}
\left|H(T)-H_\eps(T)\right| 
&\leq& \left|H_\eps(T+\eps)-H_\eps(T)\right| + \left|H_\eps(T+\eps)-\frac{1}{\cosh\frac{\eps}{2}} H_\eps(T-\eps)\right| \nonumber\\
&\leq& 2\left|H_\eps(T+\eps)- H_\eps(T-\eps)\right| + O(\eps^2)H_\eps(T-\eps). \nonumber
\end{eqnarray}
Recall that $s_j\in [\frac{1}{2},1]$. Since $\frac{e^{s_j(T\pm\eps)}}{s_j} = \frac{e^{s_j T}}{s_j}(1+O(\eps))$, it follows by Proposition \ref{thm H_eps} and Theorem \ref{thm OR09} that for $X_g\in\M_g$ satisfying the condition $\sP_A$ defined in \eqref{equ cond P_A} we have
\begin{equation*}
\eps^2 H_\eps(T-\eps) \prec Ag\eps e^{\frac{1}{2}T} + g\eps^2 e^T
\end{equation*}
and
\begin{eqnarray*}
&&\left|H_\eps(T+\eps)- H_\eps(T-\eps)\right| \\
&=&  \sum_{0\leq\lambda_j\leq\frac{1}{4}} \frac{e^{s_j (T+\eps)}}{s_j} - \sum_{0\leq\lambda_j\leq\frac{1}{4}} \frac{e^{s_j (T-\eps)}}{s_j} +O_\eta\left(Ag\frac{1}{\eps}e^{\frac{1}{2}T}+g\eps^2 e^T\right)\\
&=& O(\eps)\sum_{0\leq\lambda_j\leq\frac{1}{4}} \frac{e^{s_j T}}{s_j} +O_\eta\left(Ag\frac{1}{\eps}e^{\frac{1}{2}T}+g\eps^2 e^T\right) \\
&=& O_\eta\left(Ag\frac{1}{\eps}e^{\frac{1}{2}T}+g\eps e^T\right).
\end{eqnarray*}
So 
\begin{equation*}
\left|H(T)-H_\eps(T)\right| \prec_\eta \left(Ag\frac{1}{\eps}e^{\frac{1}{2}T}+g\eps e^T\right).
\end{equation*}
And hence
$$H(T)= e^T + \sum_{0<\lambda_j\leq\frac{1}{4}} \frac{e^{s_j T}}{s_j} +O_\eta\left(Ag\frac{1}{\eps}e^{\frac{1}{2}T}+ g\eps e^T\right).$$

Then the conclusion follows by taking $\eps=\sqrt{A}e^{-\frac{1}{4}T}$ and fixing $\eta$.
\end{proof}

Now we aim to bound $\pi_{X_g}(t)$ in Theorem \ref{thm pgt 3/4}. First we prove two lemmas. Denote 
\begin{equation}
\psi(T)=\psi(X_g,T):=\sum_{\ell(\gamma)\leq T}\Lambda(\gamma)
\end{equation}
where the sum is taken over all oriented closed geodesics in $X_g$ with length $\leq T$, and the number $\Lambda(\gamma)=\ell(\gamma_0)$ where $\gamma_0$ is the unique oriented primitive closed geodesic such that $\gamma$ is an iterate of $\gamma_0$. Then we have 
\begin{lemma}\label{thm psi}
$$0\leq H(T)-\psi(T)\prec g\cdot(1+T^2)+ \sum_{\tiny\begin{array}{c} \gamma\in\sP(X_g), \\
\ell(\gamma)<1 \end{array}}
\ln\left(\frac{1}{\ell(\gamma)}\right)$$
where $\sP(X_g)$ is the set of all oriented primitive closed geodesics in $X_g$.
\end{lemma}
\begin{proof}
Rewrite
\begin{eqnarray*}
H(T)&=& \sum_{\ell(\gamma)\leq T} \Lambda(\gamma)\frac{1+e^{-\ell(\gamma)}}{1-e^{-\ell(\gamma)}} \\
&=&\psi(T) + \sum_{\ell(\gamma)\leq T} \Lambda(\gamma)\frac{2}{e^{\ell(\gamma)}-1}.
\end{eqnarray*}
So
\begin{equation*}
0\leq H(T)-\psi(T) = 
\sum_{\tiny\begin{array}{c}
		\ell(\gamma)\leq T \\
		\Lambda(\gamma)> 2\arcsinh 1
	\end{array}} 
\frac{2\Lambda(\gamma)}{e^{\ell(\gamma)}-1} +\sum_{\tiny\begin{array}{c}
		\ell(\gamma)\leq T \\
		\Lambda(\gamma)\leq 2\arcsinh 1
\end{array}} 
\frac{2\Lambda(\gamma)}{e^{\ell(\gamma)}-1}.
\end{equation*}

For the first term, by Theorem \ref{thm count Buser},
\begin{eqnarray*}
\sum_{\tiny\begin{array}{c}
		\ell(\gamma)\leq T \\
		\Lambda(\gamma)> 2\arcsinh 1
\end{array}} 
\frac{2\Lambda(\gamma)}{e^{\ell(\gamma)}-1} 
&\prec& \sum_{k=1}^{[T]} \sum_{\tiny\begin{array}{c}
		k<\ell(\gamma)\leq k+1 \\
		\Lambda(\gamma)> 2\arcsinh 1
\end{array}} 
\frac{\ell(\gamma)}{e^{\ell(\gamma)}} \\
&\prec& \sum_{k=1}^{[T]} g e^k \frac{k}{e^k} \\
&\prec& gT^2.
\end{eqnarray*}

For the second term, let $\{\gamma_1,\cdots,\gamma_m\}$ be the set of all oriented primitive closed geodesics in $X_g$ with length $\leq 2\arcsinh 1$. By Theorem \ref{thm collar}, $m\leq 6g-6$. Then we have
\begin{eqnarray*}
\sum_{\tiny\begin{array}{c}
		\ell(\gamma)\leq T \\
		\Lambda(\gamma)\leq 2\arcsinh 1
\end{array}} 
\frac{2\Lambda(\gamma)}{e^{\ell(\gamma)}-1} 
&=& \sum_{j=1}^m \sum_{k=1}^{[T/\ell(\gamma_j)]} \frac{2\ell(\gamma_j)}{e^{k\ell(\gamma_j)}-1} \\
&\leq& \sum_{j=1}^m \ell(\gamma_j) \left(\frac{2}{e^{\ell(\gamma_j)}-1} +\int_1^\infty \frac{2}{e^{x\ell(\gamma_j)}-1} dx\right) \\
&\leq& \sum_{j=1}^m \ell(\gamma_j) \left(\frac{2}{\ell(\gamma_j)}+ \frac{2}{\ell(\gamma_j)} \ln(\frac{1}{1-e^{-\ell(\gamma_j)}})\right) \\
&\prec& \sum_{j=1}^m \left(1+\max\left\{1,\ln\frac{1}{\ell(\gamma_j)}\right\}\right) \\
&\prec& g+ \sum_{\tiny\begin{array}{c}
		\gamma\in\sP(X_g), \\
		\ell(\gamma)<1 
\end{array}} \ln\left(\frac{1}{\ell(\gamma)}\right).
\end{eqnarray*}

Then combining the two parts above, we finish the proof.
\end{proof}

Now denote 
\begin{equation}
\nu(T)=\nu(X_g,T):=\sum_{\tiny\begin{array}{c} \gamma\in\sP(X_g), \\
\ell(\gamma)\leq T \end{array}}\ell(\gamma).
\end{equation}
Then we have 
\begin{lemma}\label{thm nu}
	$$0\leq \psi(T)-\nu(T)\prec g T e^{\frac{1}{2}T}.$$
\end{lemma}
\begin{proof}
By definitions, 
\begin{eqnarray*}
0\leq \psi(T)-\nu(T)
&=& \sum_{\tiny\begin{array}{c} \gamma\in\sP(X_g), \\
\ell(\gamma)\leq T/2 \end{array}}  \sum_{k=2}^{[T/\ell(\gamma)]}\ell(\gamma) \\
&\leq& T\cdot \#\left\{\gamma\in\sP(X_g);\ \ell(\gamma)\leq\frac{T}{2}\right\} .
\end{eqnarray*}
By Corollary \ref{thm count ge^L upp}, 
$$\#\left\{\gamma\in\sP(X_g);\ \ell(\gamma)\leq\frac{T}{2}\right\} \prec g e^{\frac{1}{2}T}.$$

Then the conclusion follows.
\end{proof}

Now we are ready to prove Theorem \ref{thm pgt 3/4}.
\begin{proof}[Proof of Theorem \ref{thm pgt 3/4}]
First combining Proposition \ref{thm H}, Lemma \ref{thm psi} and \ref{thm nu} we have
\begin{equation}\label{eq-nu}
\nu(T)= e^T + \sum_{0<\lambda_j\leq\frac{1}{4}} \frac{e^{s_j T}}{s_j}+ O\left(g\sqrt{A} e^{\frac{3}{4}T} + \sum_{\tiny\begin{array}{c} \gamma\in\sP(X_g), \\
\ell(\gamma)<1 \end{array}} \ln\left(\frac{1}{\ell(\gamma)}\right) \right).
\end{equation}
We know that
\begin{equation}\label{eq-nu-1}
\begin{aligned}
\pi_{X_g}(t)
&= \sum_{\tiny\begin{array}{c} \gamma\in\sP(X_g), \\
\ell(\gamma)\leq \ln t \end{array}}  \frac{1}{\ell(\gamma)}\ell(\gamma) \\
&= \#\{\gamma\in\sP(X_g);\ \ell(\gamma)\leq \ln 2\} +\int_{T\in(\ln 2,\ln t]} \frac{1}{T} d\nu(T) 
\end{aligned}
\end{equation}
where $d\nu(T)=\sum \ell(\gamma) \delta_{\ell(\gamma)}(T)dT$ is a discrete measure, and $\delta(\cdot)$ is the Dirac function. By using \eqref{eq-nu}, Integration by Parts and Theorem \ref{thm OR09}, the second term satisfies
\begin{equation}\label{eq-nu-2}
\begin{aligned}
&\int_{(\ln 2,\ln t]} \frac{d\nu(T)}{T} 
= \frac{\nu(\ln t)}{\ln t} -\frac{\nu(\ln 2)}{\ln 2} +\int_{\ln 2}^{\ln t} \nu(T)\frac{1}{T^2} dT \\
&= \frac{t}{\ln t} + \sum_{0<\lambda_j\leq\frac{1}{4}} \frac{t^{s_j}}{s_j\ln t} - \frac{2}{\ln 2} - \sum_{0<\lambda_j\leq\frac{1}{4}} \frac{2^{s_j}}{s_j\ln 2} \\
& +\int_{\ln 2}^{\ln t} \left(e^T + \sum_{0<\lambda_j\leq\frac{1}{4}} \frac{e^{s_j T}}{s_j}\right) \frac{1}{T^2}dT \\
& +O\left(g\sqrt{A} \frac{t^{\frac{3}{4}}}{\ln t} + \sum_{\tiny\begin{array}{c} \gamma\in\sP(X_g), \\
\ell(\gamma)<1 \end{array}} \ln\frac{1}{\ell(\gamma)} \right) \\
&= \int_{\ln 2}^{\ln t} \frac{e^T}{T}dT + \sum_{0<\lambda_j\leq\frac{1}{4}}\int_{\ln 2}^{\ln t} \frac{e^{s_j T}}{T}dT \\
& +O\left(g\sqrt{A} \frac{t^{\frac{3}{4}}}{\ln t} + \sum_{\tiny\begin{array}{c} \gamma\in\sP(X_g), \\
\ell(\gamma)<1 \end{array}} \ln\frac{1}{\ell(\gamma)} \right) \\
&= \Li(t) + \sum_{0<\lambda_j\leq\frac{1}{4}} \Li(t^{s_j}) +O\left(g\sqrt{A} \frac{t^{\frac{3}{4}}}{\ln t} + \sum_{\tiny\begin{array}{c} \gamma\in\sP(X_g), \\
\ell(\gamma)<1 \end{array}} \ln\frac{1}{\ell(\gamma)} \right). 
\end{aligned}
\end{equation}
For the first term, it follows by Theorem \ref{thm collar} that for $t>2$, 
\begin{equation}\label{eq-nu-3}\#\{\gamma\in\sP(X_g);\ \ell(\gamma)\leq \ln 2\} \leq 6g-6 \prec g\sqrt{A} \frac{t^{\frac{3}{4}}}{\ln t}.
\end{equation}

Then the conclusion follows by \eqref{eq-nu}, \eqref{eq-nu-1}, \eqref{eq-nu-2} and \eqref{eq-nu-3}.
\end{proof}

\subsection{Second part of Theorem \ref{thm pgt 3/4 rdm}}
In this subsection we aim to finish the proof of Theorem \ref{thm pgt 3/4 rdm}. For $X_g\in\mathcal{B}_g\sbs\M_g$ given in Theorem \ref{thm Monk}, there exists a universal constant $A>1$ such that the condition $\sP_A$ \eqref{equ cond P_A} holds. So using Theorem \ref{thm pgt 3/4}, it remain to estimate $\sum_{\ell(\gamma)<1,\gamma\in\sP(X_g)}\ln\frac{1}{\ell(\gamma)}$ for random hyperbolic surfaces. 

\begin{lemma}\label{thm sum ln(1/gamma) small}
\begin{equation*}
\Prob\left(X_g\in\M_g;\ \sum_{\tiny\begin{array}{c} \gamma\in\sP(X_g), \\
\ell(\gamma)<1 \end{array}} \ln\frac{1}{\ell(\gamma)}<g \right) = 1-O\left(\frac{1}{g^{0.99}}\right).
\end{equation*}
\end{lemma}
\begin{proof}
For a hyperbolic surface $X_g\in\M_g$, denote $\{\gamma_1,\cdots,\gamma_m\}$ to be the set of all unoriented primitive closed geodesics with length $<1$. By Theorem \ref{thm collar}, they are all simple and $m\leq 3g-3$.
By \cite[Theorem 4.4]{Mirz13}, 
\begin{equation*}
\Prob\left(X_g\in\M_g;\ \ell_{sys}^s(X_g)<0.01\ln g\right) = O\left(\frac{(\ln g)^3}{g^{0.995}}\right) =O\left(\frac{1}{g^{0.99}}\right)
\end{equation*}
where $\ell_{sys}^s(X_g)$ is the length of the shortest simple separating geodesic in $X_g$. So for sufficiently large $g$, 
\begin{equation}\label{equ no short sep geod}
\Prob\left(X_g\in\M_g;\ 
\begin{array}{c}
\gamma_1,\cdots,\gamma_m\ \text{are all} \\
\text{simple non-separating}
\end{array}
\right) = 1-O\left(\frac{1}{g^{0.99}}\right).
\end{equation}

\noindent Then we consider unoriented primitive simple non-separating closed geodesics, which are all in the same orbit of $\Mod_g$. By using Mirzakhani's Integration Formula, i.e., Theorem \ref{thm Mirz int formula}, and Theorem \ref{thm Vgn(x) small x} and \ref{thm Vgn/Vgn+1}, we have 
\begin{eqnarray*}
\E\left[\sideset{}{^{nsep}}\sum_{\ell(\gamma)<1}\ln\frac{1}{\ell(\gamma)}\right]
&=& \frac{1}{V_g}\frac{1}{2} \int_\R \left(\ln\frac{1}{x}\right) \mathbf{1}_{(0,1)}(x) x V_{g-1,2}(x,x) dx \\
&\leq& \frac{V_{g-1,2}}{V_g} \int_0^1 \left(\ln\frac{1}{x}\right) \frac{2}{x} (\sinh(x/2))^2 dx \\
&\prec& 1
\end{eqnarray*}
where the sum $\sum^{nsep}$ is taken over all unoriented primitive simple non-separating closed geodesics. Then by Markov's inequality we have
\begin{equation}
\Prob\left(X_g\in\M_g;\ \sideset{}{^{nsep}}\sum_{\ell(\gamma)<1}\ln\frac{1}{\ell(\gamma)}\geq \frac{1}{2}g \right) = O\left(\frac{1}{g}\right).
\end{equation}
Together with \eqref{equ no short sep geod} we finish the proof.
\end{proof}

Now we are ready to prove Theorem \ref{thm pgt 3/4 rdm}.

\begin{proof}[Proof of Theorem \ref{thm pgt 3/4 rdm}]
Take 
$$\sA_g=\sB_g \cap \left\{X_g\in\M_g;\ \sum_{\tiny\begin{array}{c} \gamma\in\sP(X_g), \\
\ell(\gamma)<1 \end{array}} \ln\frac{1}{\ell(\gamma)}<g\right\}$$
where $\sB_g$ is the subset given in Theorem \ref{thm Monk}. By Theorem \ref{thm Monk} and Lemma \ref{thm sum ln(1/gamma) small},
$$\Prob(\sA_g)=1-O\left(g^{-\frac{1}{12}} (\ln g)^{\frac{9}{8}}\right).$$
And by Theorem \ref{thm pgt 3/4} it is clear that for any $X_g\in\sA_g$ and $t>2$, 
$$\pi_{X_g}(t)= \Li(t)+\sum_{0<\lambda_j\leq \frac{1}{4}}\Li(t^{s_j}) + O\left(g\frac{t^{\frac{3}{4}}}{\ln t} \right).$$

The proof is complete.
\end{proof}



\section{Intersection numbers and Weil-Petersson volumes} \label{sec-2-bounds}
In this section, based on Mirzakhani's recursive formula \cite{Mirz07} we prove Theorem \ref{in-up} and \ref{thm Vgn(x) big x-2}, which are essential in the proof of Theorem \ref{mt-geodesic}.

\subsection{Intersection numbers}
For any $d=(d_1,\cdots,d_n)$ with each $d_i\in\Z_{\geq 0}$ and $|d|=d_1+\cdots+d_n\leq 3g-3+n$, Mirzakhani \cite{Mirz07} gave a recursive formula as follow. Let
\begin{equation*}
	a_0=\frac{1}{2},\ a_n=\zeta(2n)(1-2^{1-2n}) \ \text{for}\ n\geq 1.
\end{equation*}
where $\zeta(\cdot)$ is the Riemann zeta function. Then 
\begin{equation}\label{equ intsc nmb rec}
	\left[\tau_{d_1}\cdots\tau_{d_n}\right]_{g,n} = \sum_{j=2}^n \left[\tau_{d_1}\cdots\tau_{d_n}\right]_{g,n}^{A,j} + \left[\tau_{d_1}\cdots\tau_{d_n}\right]_{g,n}^{B} + \left[\tau_{d_1}\cdots\tau_{d_n}\right]_{g,n}^{C}.
\end{equation}
where
\begin{eqnarray}\label{equ intsc Aj}
	\left[\tau_{d_1}\cdots\tau_{d_n}\right]_{g,n}^{A,j} 
	&=& 8\sum_{L=0}^{3g-3+n-|d|} (2d_j+1) a_L \left[\tau_{d_1+d_j+L-1}\prod_{i\neq 1,j}\tau_{d_i}\right]_{g,n-1}\nonumber\\
	&=& 8\sum_{L=d_1+d_j-1}^{3g-4+n-|d|+d_1+d_j} (2d_j+1) a_{L-d_1-d_2+1} \left[\tau_{L}\prod_{i\neq 1,j}\tau_{d_i}\right]_{g,n-1}, 
\end{eqnarray}
\begin{eqnarray}\label{equ intsc B}
	\left[\tau_{d_1}\cdots\tau_{d_n}\right]_{g,n}^{B} 
	&=& 16\sum_{L=0}^{3g-3+n-|d|} \sum_{k_1+k_2=L+d_1-2} a_L \left[\tau_{k_1}\tau_{k_2}\prod_{i\neq 1}\tau_{d_i}\right]_{g-1,n+1}\nonumber \\
	&=& 16\sum_{L=d_1-2}^{3g-5+n-|d|+d_1} \sum_{k_1+k_2=L} a_{L-d_1+2} \left[\tau_{k_1}\tau_{k_2}\prod_{i\neq 1}\tau_{d_i}\right]_{g-1,n+1}, 
\end{eqnarray}
\begin{eqnarray}\label{equ intsc C}
	\left[\tau_{d_1}\cdots\tau_{d_n}\right]_{g,n}^{C} 
	&=& 16\sum_{I\amalg J=\{2,\cdots,n\}} \sum_{g_1+g_2=g} \sum_{L=0}^{3g-3+n-|d|} \sum_{k_1+k_2=L+d_1-2} \nonumber\\
	&& a_L \left[\tau_{k_1}\prod_{i\in I}\tau_{d_i}\right]_{g_1,|I|+1} \times \left[\tau_{k_2}\prod_{i\in J}\tau_{d_i}\right]_{g_2,|J|+1}\nonumber\\
	&=& 16\sum_{I\amalg J=\{2,\cdots,n\}} \sum_{g_1+g_2=g} \sum_{L=d_1-2}^{3g-5+n-|d|+d_1} \sum_{k_1+k_2=L} \nonumber\\
	&& a_{L-d_1+2} \left[\tau_{k_1}\prod_{i\in I}\tau_{d_i}\right]_{g_1,|I|+1} \times \left[\tau_{k_2}\prod_{i\in J}\tau_{d_i}\right]_{g_2,|J|+1}.
\end{eqnarray}
For convenience, here we denote $a_i=0$ if $i<0$ and $[\tau_{d_1}\cdots\tau_{d_{n}}]_{g,n}=0$ if some index $d_i<0$.

Mirzakhani showed that
\begin{lemma}[{\cite[Lemma 3.1]{Mirz13}}]\label{thm intsc nmb a}
	The sequence $\{a_n\}_{n=0}^\infty$ is increasing. Moreover, $$\lim_{n\to\infty} a_n=1\ \ \text{and} \ \  a_{n+1}-a_{n}\asymp\frac{1}{2^{2n}}.$$
\end{lemma}
So clearly we have
\begin{equation}
\left[\tau_{d_1+1}\cdots\tau_{d_n}\right]_{g,n}^{A,j} < \left[\tau_{d_1}\cdots\tau_{d_n}\right]_{g,n}^{A,j},\nonumber
\end{equation}
\begin{equation}
\left[\tau_{d_1+1}\cdots\tau_{d_n}\right]_{g,n}^{B} < \left[\tau_{d_1}\cdots\tau_{d_n}\right]_{g,n}^{B},\nonumber
\end{equation}
\begin{equation}
\left[\tau_{d_1+1}\cdots\tau_{d_n}\right]_{g,n}^{C} < \left[\tau_{d_1}\cdots\tau_{d_n}\right]_{g,n}^{C}.\nonumber
\end{equation}
These directly yield to 
\begin{equation}
\left[\tau_{d_1+1}\cdots\tau_{d_n}\right]_{g,n} < \left[\tau_{d_1}\cdots\tau_{d_n}\right]_{g,n}\leq V_{g,n} \nonumber
\end{equation}
which implies 
\begin{equation}\label{equ tau_d<Vgn}
\left[\tau_{d_1}\cdots\tau_{d_n}\right]_{g,n} \leq \left[\tau_0\tau_{d_2}\cdots\tau_{d_n}\right]_{g,n}\leq V_{g,n}.
\end{equation}

Our main result on intersection numbers in this section is as follows. Theorem \ref{in-up} is as a direct consequence.
\begin{theorem}\label{thm intsc nmb big d}
There exists a universal constant $c>0$ such that for any $d=(d_1,\cdots,d_n)$, assume that $I \neq \emptyset \subset\{1,\cdots,n\}$ is the subset consisting of all $i$'s with $h_i=c\cdot \frac{d_i^2}{2g-2+n}>1$, then we have
$$\frac{\left[\tau_{d_1}\cdots\tau_{d_n}\right]_{g,n}}{V_{g,n}} \leq \prod_{i\in I} \left(\frac{1}{h_i}\right)^{\frac{1}{2}\log_4(h_i)}.$$
\end{theorem}

First assuming Theorem \ref{thm intsc nmb big d}, we prove Theorem \ref{in-up}.
\begin{proof}[Proof of Theorem \ref{in-up}]
Let $I\sbs\{1,\cdots,n\}$ be the subset in Theorem \ref{thm intsc nmb big d}. If $I=\emptyset$, we have $\frac{2g-2+n}{\max\limits_{1\leq i\leq n}\{d_i^2\}}\geq c$. Then the conclusion clearly holds by choosing $c'=\frac{1}{c}$ because $\left[\tau_{d_1}\cdots\tau_{d_n}\right]_{g,n}\leq V_{g,n}$. If $I\neq \emptyset$, let $i_0 \in[1,n]$ with $d_{i_0}^2=\max\limits_{1\leq i\leq n}\{d_i^2\}$. It is clear that the function $f(x)=\left(\frac{1}{x}\right)^{\frac{1}{2}\log_4 x-1}$ is uniformly bounded on $[1,\infty)$. Then it follows by Theorem \ref{thm intsc nmb big d} that
\begin{eqnarray}
\frac{\left[\tau_{d_1}\cdots\tau_{d_n}\right]_{g,n}}{V_{g,n}} \leq  \left(\frac{1}{h_{i_0}}\right)^{\frac{1}{2}\log_4(h_{i_0})}\leq \frac{c''}{h_{i_0}}
\end{eqnarray}
for some universal constant $c''>0$. Then the conclusion follows by choosing $c'=\frac{c''}{c}$.
\end{proof}

Now we return to prove Theorem \ref{thm intsc nmb big d}.
\begin{proof}[Proof of Theorem \ref{thm intsc nmb big d}]
For convenience, we denote $a_i=0$ if $i<0$ and $[\tau_{d_1}\cdots\tau_{d_{n}}]_{g,n}=0$ if some index $d_i<0$. For any $d'=(d'_1,\cdots,d'_{n+1})$, firstly by \eqref{equ intsc Aj} we have 
\begin{eqnarray}
&&\left[\tau_{d'_1}\tau_{d'_2}\cdots\tau_{d'_{n+1}}\right]_{g,n+1}^{A,j} - \left[\tau_{d'_1+1}\tau_{d'_2}\cdots\tau_{d'_{n+1}}\right]_{g,n+1}^{A,j} \nonumber\\
&=& 8\sum_{L=d'_1+d'_j-1}^{3g-3+n-\sum_{i\neq1,j}d'_i} (2d'_j+1) (a_{L-d'_1-d'_2+1}-a_{L-d'_1-d'_2}) \left[\tau_{L}\prod_{i\neq 1,j}\tau_{d'_i}\right]_{g,n} \nonumber\\
&\geq& 8(2d'_j+1) a_0 \left[\tau_{d'_1+d'_j-1}\prod_{i\neq 1,j}\tau_{d'_i}\right]_{g,n}. \nonumber
\end{eqnarray}
So 
\begin{equation}
\begin{aligned}
&\left[\tau_{0}\tau_{d'_2}\cdots\tau_{d'_{n+1}}\right]_{g,n+1}^{A,j} 
= \left[\tau_{d'_1}\cdots\tau_{d'_{n+1}}\right]_{g,n+1}^{A,j} \\
&+ \sum_{i=0}^{d'_1-1} \left(\left[\tau_{i}\tau_{d'_2}\cdots\tau_{d'_{n+1}}\right]_{g,n+1}^{A,j} - \left[\tau_{i+1}\tau_{d'_2}\cdots\tau_{d'_{n+1}}\right]_{g,n+1}^{A,j}\right) \\
&\geq \sum_{i=0}^{d'_1-1} 8(2d'_j+1) a_0 \left[\tau_{i+d'_j-1}\prod_{i\neq 1,j}\tau_{d'_i}\right]_{g,n} \\
&\geq 8d'_1(2d'_j+1) a_0 \left[\tau_{d'_1+d'_j-2}\prod_{i\neq 1,j}\tau_{d'_i}\right]_{g,n}.
\end{aligned}
\end{equation}

\noindent And then for any $d'_1\geq 0$ and $2\leq j\leq n+1$, 
\begin{equation}
\left[\tau_{0}\tau_{d'_2}\cdots\tau_{d'_{n+1}}\right]_{g,n+1} \geq 8d'_1(2d'_j+1) a_0 \left[\tau_{d'_1+d'_j-2}\prod_{i\neq 1,j}\tau_{d'_i}\right]_{g,n}.
\end{equation}
For $d=(d_1,\cdots,d_n)$, taking $d'_1=\lfloor\frac{d_1+2}{2}\rfloor$, $d'_j=\lceil\frac{d_1+2}{2}\rceil$ and $j=2$ we have
\begin{equation}
\begin{aligned}
\left[\tau_{d_1}\cdots\tau_{d_n}\right]_{g,n} 
&\leq& \frac{\left[\tau_{0}\tau_{\lceil\frac{d_1+2}{2}\rceil}\tau_{d_2}\cdots\tau_{d_n}\right]_{g,n+1}}{8a_0\lfloor\frac{d_1+2}{2}\rfloor(2\lceil\frac{d_1+2}{2}\rceil+1)} \\
&\leq& \frac{1}{2d^2_1} \left[\tau_{0}\tau_{\lceil\frac{d_1+2}{2}\rceil}\tau_{d_2}\cdots\tau_{d_n}\right]_{g,n+1}
\end{aligned}
\end{equation}
where the fact $a_0=\frac{1}{2}$ is used in the last inequality. Now we apply this argument for $p_1\geq 0$ times and each time $d_1$ is reduced by no more than $\frac{1}{2}$. So combine with \eqref{equ tau_d<Vgn}, we get
\begin{equation*}
\left[\tau_{d_1}\cdots\tau_{d_n}\right]_{g,n} \leq \prod_{k=1}^{p_1}\frac{1}{2(d_1/2^{k-1})^2} \left[\tau_0\cdots\tau_0\tau_{d_2}\cdots\tau_{d_n}\right]_{g,n+p_1}.
\end{equation*}
By induction, for any $p=(p_1,\cdots,p_n)$, 
\begin{equation}
\left[\tau_{d_1}\cdots\tau_{d_n}\right]_{g,n} \leq \prod_{i=1}^n \left(\prod_{k_i=1}^{p_i}\frac{1}{2(d_i/2^{k_i-1})^2}\right)  V_{g,n+|p|}.
\end{equation}
By Theorem \ref{thm Vgn/Vgn+1} we know that
\begin{equation*}
V_{g,n+|p|} \leq (b(2g-2+n+|p|))^{|p|} V_{g,n}
\end{equation*}
for a universal constant $b=\frac{120}{10-\pi^2}\approx 920.276$. So 
\begin{equation}\label{equ pf of intsc nmb}
\left[\tau_{d_1}\cdots\tau_{d_n}\right]_{g,n} \leq \prod_{i=1}^n \left(\prod_{k_i=1}^{p_i}\frac{b(2g-2+n+|p|)}{2(d_i/2^{k_i-1})^2}\right)  V_{g,n}.
\end{equation}
Let $h_i = \frac{2}{3b}\frac{d_i^2}{2g-2+n}$. Assume $I\sbs\{1,\cdots,n\}$ is the subset consisting of all $i$'s with $h_i>1$. For any $i\in[1,n]$, take
\begin{equation*}
p_i =
\begin{cases}
0 & \text{if}\ h_i\leq 1\\
1+[\log_4(h_i)]  & \text{if}\ h_i>1
\end{cases}.
\end{equation*}
Recall that the function $\log_4 (x)$ is  concave on $(0,\infty)$. In particular, $$\frac{1}{|I|}\log_4 \left(\prod_{i\in I} d_i\right)\leq \log_4\left(\frac{\sum \limits_{i\in I}d_i}{|I|} \right).$$
Then we have 
\begin{equation}\label{bound-p}
\begin{aligned}
|p|&=\sum\limits_{i\in I}p_i\leq |I|+ \log_4\left(\left(\frac{2}{3b(2g-2+n)}\right)^{|I|}\prod_{i\in I} d_i^2\right) \\
&\leq |I|+ |I|\log_4\left(\frac{2}{3b(2g-2+n)}\left(\frac{\sum \limits_{i\in I}d_i}{|I|}\right)^2\right) \\
&\leq |I|+ |I|\log_4\left(\frac{2(3g-3+n)^2}{3b(2g-2+n)}\frac{1}{|I|^2}\right) \\
&\leq |I|+|I|\frac{3g-3+n}{b}\frac{1}{|I|^2}\\
&\leq 2(2g-2+n).
\end{aligned}
\end{equation}
And so for $i\in I$, 
\begin{equation}\label{bound-p-2}
\begin{aligned}
\prod_{k_i=1}^{p_i}&\frac{b(2g-2+n+|p|)}{2(d_i/2^{k_i-1})^2}
\leq \prod_{k=1}^{p_i}\frac{3b(2g-2+n)}{2(d_i/2^{k-1})^2} \\
&= \prod_{k=1}^{p_i} \frac{4^{k-1}}{h_i}= \frac{4^{\frac{1}{2}(1+[\log_4(h_i)])[\log_4(h_i)]}}{h_i^{1+[\log_4(h_i)]}} \\
&\leq (h_i)^{-\frac{1}{2}\log_4(h_i)}.
\end{aligned}
\end{equation}

Then the conclusion follows by \eqref{equ pf of intsc nmb}, \eqref{bound-p} and \eqref{bound-p-2}.
\end{proof}

\subsection{Weil-Petersson volumes}\label{sec wp vol}
When $\sum x_i^2 = o(2g-2+n)$, for the asymptotic behavior of $V_{g,n}(x_1,\cdots,x_n)$, the bounds in Theorem \ref{thm Vgn(x) small x} are effective. In this subsection we apply Theorem \ref{thm intsc nmb big d} to prove Theorem \ref{thm Vgn(x) big x-2}, which is a more effective upper bound for $V_{g,n}(x_1,\cdots,x_n)$ when some of $x_i$'s are larger than $\sqrt{2g-2+n}$. We first prove the following result.
\begin{theorem}\label{thm Vgn(x) big x}
For any $k\geq 1$, there exists a constant $c(k)>0$ independent of $g,n$ and $x_i$'s such that 
$$\frac{V_{g,n}(x_1,\cdots,x_n)}{V_{g,n}} \leq \left(\prod_{i\in I} c(k)\frac{2g-2+n}{x_i^2}\right)^k \prod_{i=1}^n \frac{\sinh(x_i/2)}{x_i/2}$$
where we assume that $I\neq \emptyset \sbs\{1,\cdots,n\}$ is the subset consisting of all $i$'s satisfying that $x_i^2 >c(k)(2g-2+n)$.
\end{theorem}

\begin{proof}
As in Theorem \ref{thm intsc nmb big d} we denote $h_i=c\frac{d_i^2}{2g-2+n}$. Then similar as in the proof of Theorem \ref{in-up} we have that for any fixed $k\geq 1$, $\left(\frac{1}{h_i}\right)^{\frac{1}{2}\log_4(h_i)} \prec_k \left(\frac{2g-2+n}{(2d_i+2k+1)^2}\right)^{k}$ whenever $h_i>1$, and $1\prec_k \left(\frac{2g-2+n}{(2d_i+2k+1)^2}\right)^{k}$ whenever $h_i\leq 1$. So combine with Theorem \ref{thm intsc nmb big d} we have that for any subset $I'\sbs\{1,\cdots,n\}$, 
\begin{equation}
\frac{\left[\tau_{d_1}\cdots\tau_{d_n}\right]_{g,n}}{V_{g,n}} \leq \prod_{i\in I'} \left(c'(k)\frac{2g-2+n}{(2d_i+2k+1)^2}\right)^{k}
\end{equation}
for some $c'(k)>0$. Then it follows by Theorem \ref{thm Vgn(x)} that 
\begin{equation}\label{equ Vgn(x)<}
\begin{aligned}
&\frac{V_{g,n}(2x_1,\cdots,2x_n)}{V_{g,n}} 
\leq \sum_{|d|\leq 3g-3+n}  \prod_{i\in I'} \left(c'(k)\frac{2g-2+n}{(2d_i+2k+1)^2}\right)^{k} \prod_{i=1}^n \frac{x_i^{2d_i}}{(2d_i+1)!} \\
&\leq \left(c'(k)(2g-2+n)\right)^{k|I'|}\cdot \left(\prod_{i\in I'} \sum_{d_i=0}^\infty \frac{x_i^{2d_i}}{(2d_i+2k+1)!}\right)\cdot\left(\prod_{i\notin I'} \sum_{d_i=0}^\infty \frac{x_i^{2d_i}}{(2d_i+1)!}\right).
\end{aligned}
\end{equation}

\noindent By Taylor's expansion, 
\begin{equation*}
\sum_{d=0}^\infty \frac{x^{2d}}{(2d+1)!} =\frac{\sinh(x)}{x}
\end{equation*}
and hence
\begin{equation*}
\sum_{d=0}^\infty \frac{x^{2d}}{(2d+2k+1)!} \leq \frac{1}{x^{2k}}\sum_{d=0}^\infty \frac{x^{2d}}{(2d+1)!}  =\frac{\sinh(x)}{x^{2k+1}}.
\end{equation*}
Then again by \eqref{equ Vgn(x)<}, 
\begin{equation}
\frac{V_{g,n}(2x_1,\cdots,2x_n)}{V_{g,n}} \leq \left(\prod_{i\in I'} c'(k)\frac{2g-2+n}{x_i^{2}}\right)^k \prod_{i=1}^n \frac{\sinh(x_i)}{x_i}
\end{equation}
for any subset $I'\sbs\{1,\cdots,n\}$. Now let $I'=I$, then the conclusion follows by choosing $c(k)=4c'(k)$. The proof is complete.
\end{proof}

Now we complete the proof of Theorem \ref{thm Vgn(x) big x-2}.
\bp[Proof of Theorem \ref{thm Vgn(x) big x-2}]
Recall that the upper bound in Theorem \ref{thm Vgn(x) small x} says that
\begin{equation}\label{Vgn-up}
\frac{V_{g,n}(x_1,\cdots,x_n)}{V_{g,n}} \leq \prod_{i=1}^n \frac{\sinh(x_i/2)}{x_i/2}.
\end{equation}

Then the conclusion directly follows by \eqref{Vgn-up} and Theorem \ref{thm Vgn(x) big x}.
\ep


\section{Proof of Theorem \ref{mt-geodesic}}
In this section we prove Theorem \ref{mt-geodesic} confirming \cite[Conjecture $1.2$]{LW21}.

Similar as \cite{NWX20, WX22-GAFA}, we first make the following notations. Let $X_g\in \mathcal{M}_g$ be a closed hyperbolic surface of genus $g$ and $\mathcal{P}'(X_g)$ be the set of all unoriented primitive closed geodesics on $X_g$. For any $L>0$, we define
\begin{enumerate}
\item $N(X_g,L):=\#\{\gamma \in \mathcal{P}'(X_g);\ \ell_\gamma(X_g)\leq L\}$,
\item $N^s(X_g,L):=\#\{\gamma \in \mathcal{P}'(X_g) \text{ is simple};\ \ell_\gamma(X_g)\leq L\}$,
\item $N^{ns}(X_g,L):=\#\{\gamma \in \mathcal{P}'(X_g) \text{ is non-simple};\ \ell_\gamma(X_g)\leq L\}$,
\item $N_{nsep}^s(X_g,L):=\#\{\gamma \in \mathcal{P}'(X_g) \text{ is simple and non-separating};\ \ell_\gamma(X_g)\leq L\}$,
\item $N_{sep,k}^s(X_g,L):=\#\{\gamma \in \mathcal{P}'(X_g)\text{ and $X_g\setminus \gamma\cong S_{k,1}\cup S_{g-k,1}$};\ \ell_\gamma(X_g)\leq L\}$ where $1\leq k \leq \left[\frac{g}{2}\right]$.
\end{enumerate}

\begin{figure}[h]
    \centering
    \includegraphics[width=3.2 in]{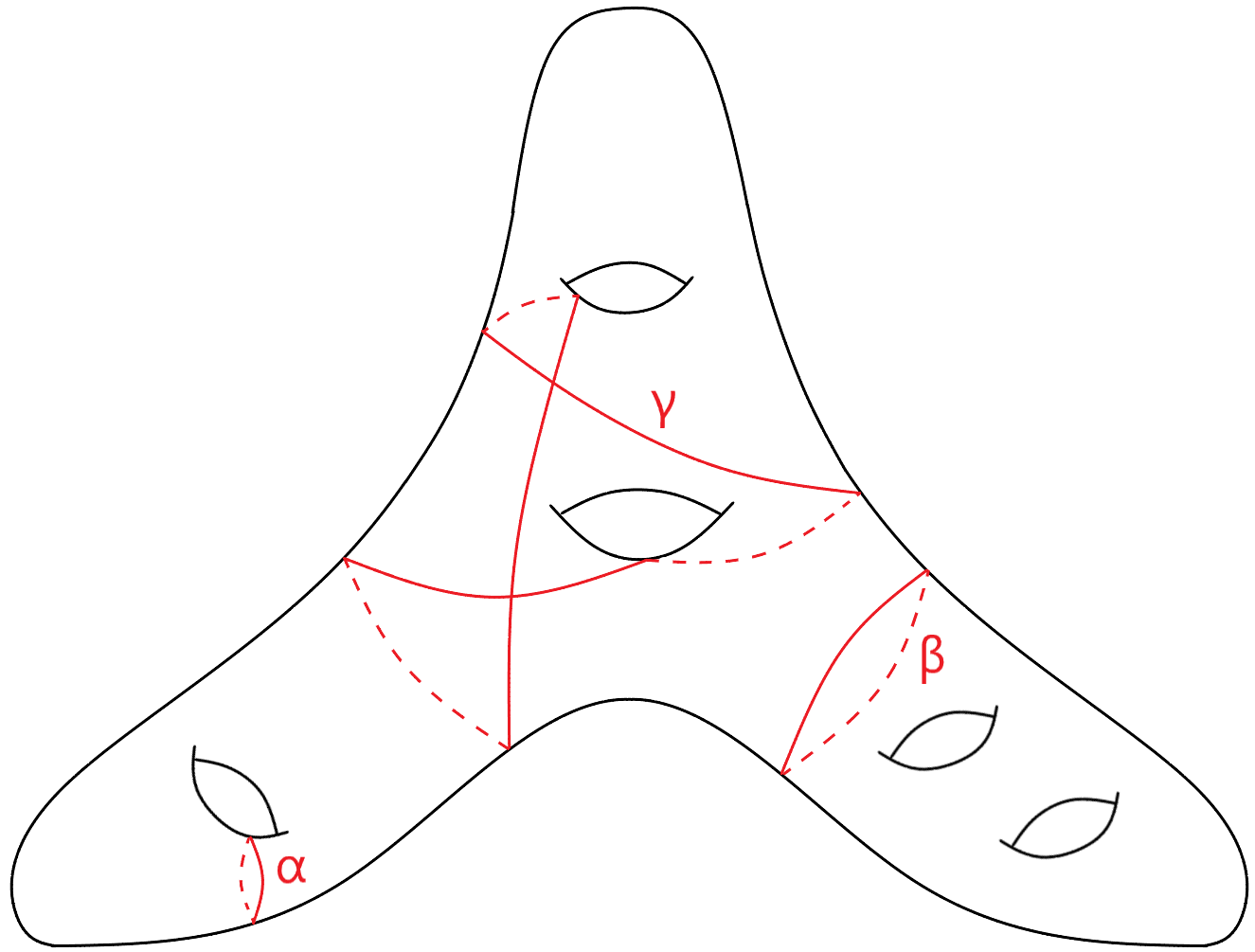}
    \caption{Examples: $\alpha$ is simple and non-separating; $\beta$ is simple and separating; $\gamma$ is non-simple.}
    \label{fig:lemma33case2}
\end{figure}

It is clear that
\[N(X_g,L)=\frac{\pi_{X_g}(e^L)}{2}\]
and
\begin{equation}\label{equ N=s+ns}
\begin{aligned}
&N(X_g,L)=N^s(X_g,L)+N^{ns}(X_g,L)\\
&=N_{nsep}^s(X_g,L)+\sum_{k=1}^{\left[\frac{g}{2}\right]}N_{sep,k}^s(X_g,L)+N^{ns}(X_g,L).
\end{aligned}
\end{equation}

We split the proof of Theorem \ref{mt-geodesic} into three parts. 
\subsection{Significantly greater than $\sqrt{g}$}
In this subsection we confirm the second part of \cite[Conjecture $1.2$]{LW21}. That is, we show that as $g\to \infty$, on most surfaces in $\sM_g$ most geodesics of length significantly greater than $\sqrt{g}$ are non-simple. More precisely,
\begin{theorem}\label{conj-nsimple}
Assume that $L=L(g)$ satisfies $$\lim \limits_{g\to \infty}\frac{L^2(g)}{g}=\infty.$$ Then for any $k>0\in \mathbb{Z}$, we have
\[\lim \limits_{g\to \infty}\Prob\left(X_g\in\M_g;\ \left|1-\frac{N^{ns}(X_g,L)}{N(X_g,L)} \right|<\left( \frac{g}{L^2}\right)^k\right)=1.\]
\end{theorem}

We first begin with the following easy consequence of Theorem \ref{thm Vgn(x) big x-2}.
\begin{lemma}\label{vol-up-1}
For any $x>0$ and $k>0\in \mathbb{Z}$, we have that for all $g\geq 2$ and $1\leq i \leq \left[\frac{g}{2}\right]$,
\[\frac{V_{g-1,2}(x,x)}{V_{g-1,2}}\prec_k   \left(\frac{g}{1+x^2}\right)^{k+1}\cdot \left(\frac{\sinh\left(\frac{x}{2}\right)}{\frac{x}{2}} \right)^2\]
and
\[\frac{V_{g-i,1}(x)\cdot V_{i,1}(x)}{V_{g-i,1}\cdot V_{i,1}}\prec_k  \left(\frac{g}{1+x^2}\right)^{k+1}\cdot \left(\frac{\sinh\left(\frac{x}{2}\right)}{\frac{x}{2}} \right)^2.\]
\end{lemma}
\begin{proof}
By Theorem \ref{thm Vgn(x) big x-2} we have
\begin{equation}\nonumber
\begin{aligned}
\frac{V_{g-1,2}(x,x)}{V_{g-1,2}}&\leq \left(\min\left\{c(k+1)\frac{2g-2}{x^2}, 1\right\}\right)^{2k+2}  \cdot \left(\frac{\sinh\left(\frac{x}{2}\right)}{\frac{x}{2}} \right)^2\\
&\leq \left(\min\left\{c(k+1)\frac{2g-2}{x^2}, 1\right\}\right)^{k+1}  \cdot \left(\frac{\sinh\left(\frac{x}{2}\right)}{\frac{x}{2}} \right)^2\\
&\prec_k   \left(\frac{g}{1+x^2}\right)^{k+1}\cdot \left(\frac{\sinh\left(\frac{x}{2}\right)}{\frac{x}{2}} \right)^2.
\end{aligned}
\end{equation}
Now we prove the second inequality. Recall that the upper bound in Theorem \ref{thm Vgn(x) small x} implies that $\frac{V_{i,1}(x)}{V_{i,1}}\leq \frac{\sinh\left(\frac{x}{2}\right)}{\frac{x}{2}}.$
Similarly, then it follows by Theorem \ref{thm Vgn(x) big x-2} that for all $1\leq i \leq \left[\frac{g}{2}\right]$,
\begin{equation}\nonumber
\begin{aligned}
\frac{V_{g-i,1}(x)\cdot V_{i,1}(x)}{V_{g-i,1}\cdot V_{i,1}}&\leq \frac{V_{g-i,1}(x)}{V_{g-i,1}}
 \cdot \left(\frac{\sinh\left(\frac{x}{2}\right)}{\frac{x}{2}}\right) \\
&\leq \left(\min\left\{c(k+1)\frac{2g-2i-1}{x^2}, 1\right\}\right)^{k+1}  \cdot \left(\frac{\sinh\left(\frac{x}{2}\right)}{\frac{x}{2}} \right)^2\\
&\prec_k   \left(\frac{g}{1+x^2}\right)^{k+1}\cdot \left(\frac{\sinh\left(\frac{x}{2}\right)}{\frac{x}{2}} \right)^2.
\end{aligned}
\end{equation}

The proof is complete.
\end{proof}

\begin{lemma}\label{exp-2}
Assume that $L=L(g)$ satisfies $$\lim \limits_{g\to \infty}\frac{L^2(g)}{g}=\infty.$$ Then for any $k>0\in \mathbb{Z}$,
\[\E \left[N_{nsep}^s(X_g,L) \right]\prec_k \Li(e^L) \cdot \left( \frac{g}{L^2}\right)^{k+1}.\]
\end{lemma}
\begin{proof}
It follows by Mirzakhani's integration formula, i.e., Theorem \ref{thm Mirz int formula}, and Lemma \ref{vol-up-1} that
\begin{equation}\nonumber
\begin{aligned}
\E \left[N_{nsep}^s(X_g,L) \right]&=\frac{1}{V_g}\frac{1}{2}\int_0^L V_{g-1,2}(x,x)xdx\\
&\prec_k \frac{V_{g-1,2}}{V_g}\cdot \int_0^L \left( \sinh\left(\frac{x}{2}\right)\right)^2 \cdot \frac{1}{x}\cdot \left(\frac{g}{1+x^2}\right)^{k+1} dx.
\end{aligned}
\end{equation}
By Theorem \ref{thm Vgn/Vgn+1} we also know that
\[\frac{V_g}{V_{g-1,2}}=1+O\left(\frac{1}{g}\right).\] Thus, 
\begin{equation}
\nonumber
\begin{aligned}
\E \left[N_{nsep}^s(X_g,L) \right]&\prec_k \left(1+O\left(\frac{1}{g}\right)\right)\cdot  \int_0^L \left( \sinh\left(\frac{x}{2}\right)\right)^2  \frac{1}{x} \left(\frac{g}{1+x^2}\right)^{k+1} dx\\
&\asymp \Li(e^L) \cdot \left( \frac{g}{L^2}\right)^{k+1}
\end{aligned}
\end{equation}
where we apply $\lim \limits_{g\to \infty}\frac{L^2(g)}{g}=\infty$ in the last equation. The proof is complete. 
\end{proof}

\begin{lemma}\label{exp-3}
Assume that $L=L(g)$ satisfies $$\lim \limits_{g\to \infty}\frac{L^2(g)}{g}=\infty.$$ Then for any $k>0\in \mathbb{Z}$,
\[\sum_{i=1}^{\left[\frac{g}{2}\right]}\E \left[N_{sep,i}^s(X_g,L) \right]\prec_k \frac{\Li(e^L)}{g} \cdot \left( \frac{g}{L^2}\right)^{k+1}.\]
\end{lemma}
\begin{proof}
It follows by Theorem \ref{thm Mirz int formula} and Lemma \ref{vol-up-1} that
\begin{equation}\nonumber
\begin{aligned}
&\sum_{i=1}^{\left[\frac{g}{2}\right]}\E \left[N_{sep,i}^s(X_g,L) \right]\leq\frac{1}{V_g}\sum_{i=1}^{\left[\frac{g}{2}\right]}\int_0^L V_{g-i,1}(x)V_{i,1}(x)xdx\\
&\prec_k \left(\sum_{i=1}^{\left[\frac{g}{2}\right]}\frac{V_{g-i,1}V_{i,1}}{V_g}\right)\cdot \int_0^L \left( \sinh\left(\frac{x}{2}\right)\right)^2 \cdot \frac{1}{x}\cdot \left(\frac{g}{1+x^2}\right)^{k+1} dx.
\end{aligned}
\end{equation}
By Theorem \ref{thm sum V*V} we also know that
\[\sum_{i=1}^{\left[\frac{g}{2}\right]}\frac{V_{g-i,1}V_{i,1}}{V_g}\asymp \frac{1}{g}.\] Thus, 
\begin{equation}
\nonumber
\begin{aligned}
\sum_{i=1}^{\left[\frac{g}{2}\right]}\E \left[N_{sep,i}^s(X_g,L) \right]&\prec_k \frac{1}{g}\cdot  \int_0^L \left( \sinh\left(\frac{x}{2}\right)\right)^2  \frac{1}{x} \left(\frac{g}{1+x^2}\right)^{k+1} dx\\
&\asymp \frac{\Li(e^L)}{g} \cdot \left( \frac{g}{L^2}\right)^{k+1}
\end{aligned}
\end{equation}
where we apply $\lim \limits_{g\to \infty}\frac{L^2(g)}{g}=\infty$ in the last equation. The proof is complete. 
\end{proof}

Now we are ready to prove Theorem \ref{conj-nsimple}.
\begin{proof}[Proof of Theorem \ref{conj-nsimple}]
First by Lemma \ref{exp-2} and Lemma \ref{exp-3} we know that for any $k>0\in \mathbb{Z}$,
\begin{equation}
\nonumber
\begin{aligned}
\E \left[N^s(X_g,L) \right]&=\E \left[N_{nsep}^s(X_g,L) \right]+\sum_{i=1}^{\left[\frac{g}{2}\right]}\E \left[N_{sep,i}^s(X_g,L) \right]\\
&\prec_k \Li(e^L) \cdot \left( \frac{g}{L^2}\right)^{k+1}.
\end{aligned}
\end{equation}
Then by Markov's inequality we have
\begin{equation}\label{pr-in-1-1}
\begin{aligned}
\Prob&\left(X_g \in \sM_g; \  N^s(X_g,L)\geq \frac{1}{4} \left(\frac{g}{L^2} \right)^k \cdot\Li(e^L)\right)\\
& \prec_k \frac{g}{L^2}\to 0 \text{ as } g\to \infty.
\end{aligned}
\end{equation}

\noindent Recall that $N(X_g,L)=\frac{\pi_{X_g}(e^L)}{2}$. Let $\mathcal{A}_g$ be as in Theorem \ref{thm pgt 3/4 rdm} with 
\begin{equation}\label{pr-in-1-2}
\lim \limits_{g\to \infty}\Prob\left(X_g \in \mathcal{A}_g \right)=1.
\end{equation}
By Theorem \ref{thm pgt 3/4 rdm}, for any $X_g\in \mathcal{A}_g$ and large enough $g>0$,
\begin{equation*}
	N(X_g,L) = \frac{1}{2}\Li(e^L)+\sum_{0<\lambda_j\leq \frac{1}{4}}\frac{1}{2}\Li(e^{s_j L}) + O\left(g\frac{e^{\frac{3}{4}L}}{L} \right).
\end{equation*}
For each eigenvalue $0<\lambda_j\leq \frac{1}{4}$ (if there exists) and the corresponding $s_j=\frac{1}{2}+\sqrt{\frac{1}{4}-\lambda_j}\in[\frac{1}{2},1)$, the middle term is non-negative: 
\begin{equation*}
	\frac{1}{2}\Li(e^{s_j L}) = \frac{1}{2} \int_2^{e^{s_j L}} \frac{dx}{\ln x} \geq 0
\end{equation*}
where we apply that $L=L(g)$ is large enough such that $e^{s_j L}\geq e^{\frac{1}{2} L}>2$. Since $\lim \limits_{g\to \infty}\frac{L^2(g)}{g}=\infty$ and $\Li(e^L)\sim \frac{e^L}{L}$ as $L=L(g)\to\infty$, the remainder should have the lower bound
\begin{equation*}
	O\left(g\frac{e^{\frac{3}{4}L}}{L} \right) = \Li(e^L)\cdot O\left(g e^{-\frac{1}{4}L} \right) \geq -\frac{1}{4}\Li(e^L)
\end{equation*}
for large enough $g$. Hence for any $X_g\in \mathcal{A}_g$ and large enough $g>0$, 
\begin{equation}\label{pr-in-1-3}
N(X_g,L) \geq \frac{1}{2}\Li(e^L) - \frac{1}{4}\Li(e^L) =  \frac{1}{4}\Li(e^L).
\end{equation}
Since $N^s(X_g,L)=N(X_g,L)-N^{ns}(X_g,L)$, it follows by \eqref{pr-in-1-1} and \eqref{pr-in-1-3} that for large enough $g>0$,
\begin{equation}\nonumber
\begin{aligned}
\Prob&\left(X_g\in\M_g;\ \left|1-\frac{N^{ns}(X_g,L)}{N(X_g,L)} \right| \geq \left( \frac{g}{L^2}\right)^k\right)
\leq \Prob\left(X_g \notin \mathcal{A}_g \right)\\
&+\Prob\left(X_g \in \sM_g; \  N^s(X_g,L)\geq \frac{1}{4} \left(\frac{g}{L^2} \right)^k \cdot\Li(e^L)\right)\\
&\prec_k \left(\left(1-\Prob\left(X_g \in \mathcal{A}_g \right)\right)+ \frac{g}{L^2}\right) .
\end{aligned}
\end{equation}
By letting $g\to \infty$, then the conclusion follows by \eqref{pr-in-1-2} and our assumption on $L=L(g)$. The proof is complete.
\end{proof}

\subsection{Significantly less than $\sqrt{g}$ and larger than $12\ln g$}
In this and next subsection we confirm the first part of \cite[Conjecture $1.2$]{LW21}. That is, we show that as $g\to \infty$, on most surfaces in $\sM_g$ most geodesics of length significantly less than $\sqrt{g}$ are simple and non-separating. The proof is split into two subcases: when $L(g)>12 \ln g$, we apply the spectral method in this paper; when $L(g)\leq 12 \ln g$, we apply our previous works \cite{NWX20, WX22-GAFA}. 

In this subsection we prove
\begin{theorem}\label{conj-simple}
Assume that $L=L(g)$ satisfies $$L(g)>12\ln g \text{ and }\lim \limits_{g\to \infty}\frac{L^2(g)}{g}=0.$$ Then for any $\delta=\delta(g)>0$ satisfying $\lim \limits_{g\to \infty}\delta(g)=0$ and $\lim \limits_{g\to \infty} \frac{1}{\delta}\left(\frac{L^2}{g}+\frac{1}{g^{1-\epsilon_0}} \right)=0$ for some $\epsilon_0 \in (0,1)$, we have
\[\lim \limits_{g\to \infty}\Prob\left(X_g\in\M_g;\ \left|1-\frac{N_{nsep}^s(X_g,L)}{N(X_g,L)} \right|<\delta\right)=1.\]
\end{theorem}

We first make the following lemma.
\begin{lemma}\label{exp-1}
Assume that $L=L(g)$ satisfies $$\lim \limits_{g\to \infty}L(g)=\infty \text{ and }\lim \limits_{g\to \infty}\frac{L^2(g)}{g}=0.$$ Then
\[\E \left[N_{nsep}^s(X_g,L) \right]= \frac{\Li(e^L)}{2}\cdot \left(1+O\left(\frac{L^2}{g}+\frac{L\ln L}{e^L}\right)\right)\]
where the implied constant is universal.
\end{lemma}
\begin{proof}
First by Mirzakhani's integration formula, i.e., Theorem \ref{thm Mirz int formula}, we have
\[\E \left[N_{nsep}^s(X_g,L) \right]=\frac{1}{V_g}\frac{1}{2}\int_0^L V_{g-1,2}(x,x)xdx.\]
Since $L=L(g)=o(\sqrt{g})$, by Theorem \ref{thm Vgn(x) small x} we have that for all $x\in [0,L]$,
\[V_{g-1,2}(x,x)=V_{g-1,2}\cdot \left(\frac{\sinh\left(\frac{x}{2}\right)}{\frac{x}{2}} \right)^2\cdot \left(1+O\left(\frac{x^2}{g} \right)\right)\] where the implied constant is universal. By Theorem \ref{thm Vgn/Vgn+1} we also know that
\[\frac{V_g}{V_{g-1,2}}=1+O\left(\frac{1}{g}\right).\] Thus, 
\begin{equation}
\nonumber
\begin{aligned}
&\E \left[N_{nsep}^s(X_g,L) \right]=\int_0^L \left(\sinh\left(\frac{x}{2}\right)\right)^2 \cdot \frac{2}{x}\cdot\left(1+O\left(\frac{x^2}{g} \right)\right)\\
&\cdot \left(1+O\left(\frac{1}{g}\right)\right)dx=\frac{\Li(e^L)}{2}\cdot \left(1+O\left(\frac{L^2}{g}+\frac{L\ln L}{e^L}\right)\right)
\end{aligned}
\end{equation}
where we apply $\lim \limits_{g\to \infty}L=\lim \limits_{g\to \infty}L(g)=\infty$ in the last equality. 

The proof is complete.
\end{proof}

Now we are ready to prove Theorem \ref{conj-simple}.
\begin{proof}[Proof of Theorem \ref{conj-simple}]
Recall that $N(X_g,L)=\frac{\pi_{X_g}(e^L)}{2}$. Since $L>12 \ln g$, it follows by Theorem \ref{thm pgt E} that for any $\epsilon \in (0,1)$, 
\[\E \left[N(X_g,L)\right]=\frac{\E \left[\pi_{X_g}(e^L)\right]}{2}=\frac{\Li(e^L)}{2}\cdot \left( 1+O_\epsilon\left(\frac{1}{g^{1-\epsilon}}\right)\right)\]
which together with Lemma \ref{exp-1} imply that
\[\E \left[\left|N(X_g,L)-N_{nsep}^s(X_g,L) \right| \right]=\Li(e^L)\cdot O_\epsilon\left(\frac{L^2}{g}+\frac{1}{g^{1-\epsilon}}\right).\]
Now we choose $\epsilon=\epsilon_0$. Then by Markov's inequality we have
\begin{equation}\label{pr-in-0-1}
\begin{aligned}
&\Prob\left(X_g \in \sM_g; \ \left|N(X_g,L)-N_{nsep}^s(X_g,L) \right|\geq \frac{\delta \cdot\Li(e^L)}{4} \right)\\
&=O_{\epsilon_0}\left(\frac{1}{\delta}\left(\frac{L^2}{g}+\frac{1}{g^{1-\epsilon_0}} \right)\right)\to 0 \text{ as } g\to \infty.
\end{aligned}
\end{equation}
Let $\mathcal{A}_g$ be in Theorem \ref{thm pgt 3/4 rdm}. Then we have 
\begin{equation}\label{pr-in-0-2}
\lim \limits_{g\to \infty}\Prob\left(X_g \in \mathcal{A}_g \right)=1
\end{equation}
and for any $X_g\in \mathcal{A}_g$ and large enough $g>0$,
\begin{equation}\label{pr-in-0-3}
\begin{aligned}
N(X_g,L)\geq \frac{1}{2}\cdot \Li(e^L) \cdot \left(1+O\left(\frac{g}{ e^{\frac{L}{4}}}\right) \right)\geq \frac{\Li(e^L)}{4}
\end{aligned}
\end{equation}
where we apply $L>12 \ln g$ in the last inequality. Thus, it follows by \eqref{pr-in-0-1} and \eqref{pr-in-0-3} that for large enough $g>0$,
\begin{equation}\nonumber
\begin{aligned}
&\Prob\left(X_g\in\M_g;\ \left|1-\frac{N_{nsep}^s(X_g,L)}{N(X_g,L)} \right|\geq \delta\right)
\leq \Prob\left(X_g \notin \mathcal{A}_g \right)\\
&+\Prob\left(X_g \in \sM_g; \ \left|N(X_g,L)-N_{nsep}^s(X_g,L) \right|\geq \frac{\delta \cdot\Li(e^L)}{4} \right)\\
&=\left(1-\Prob\left(X_g \in \mathcal{A}_g \right)\right)+O_{\epsilon_0}\left(\frac{1}{\delta}\left(\frac{L^2}{g}+\frac{1}{g^{1-\epsilon_0}} \right)\right).
\end{aligned}
\end{equation}
By letting $g\to \infty$, then the conclusion follows by \eqref{pr-in-0-2} and our assumption on $\delta$. The proof is complete.
\end{proof}

\subsection{No more than $12\ln g$} As mentioned in the introduction, it follows by \cite[Theorem 4.4]{Mirz13} and \cite[Theorem 4]{NWX20} that if $L=L(g)\leq (1-\eps_0)\ln g$ for some constant $\eps_0>0$, then
\[\lim \limits_{g\to \infty}\Prob\left(X_g\in\M_g;\ N(X_g,L)=N_{nsep}^s(X_g,L)\right)=1.\]
In this subsection we prove
\begin{theorem}\label{conj-simple-2}
Assume $L=L(g)$ satisfies that for some $\eps_0\in(0,1)$,
$$ L(g)\geq (1-\eps_0)\ln g\ \ \text{and}\ \ \limsup_{g\to\infty} \frac{L(g)}{\ln g}<\infty.$$
Then for any fixed $\eps>0$ we have
\[\lim \limits_{g\to \infty}\Prob\left(X_g\in\M_g;\ \left|1-\frac{N_{nsep}^s(X_g,L)}{N(X_g,L)} \right|<\frac{1}{g^{1-\eps}}\right)=1.\]
\end{theorem}


Recall that
\begin{equation} \nonumber
N(X_g,L)=N_{nsep}^s(X_g,L)+\sum_{k=1}^{\left[\frac{g}{2}\right]}N_{sep,k}^s(X_g,L)+N^{ns}(X_g,L).
\end{equation}
We will deal with the three terms on the $\mathrm{RHS}$ above separately. \\

\noindent \textbf{Expected value on simple and separating closed geodesics.} Similar as Lemma \ref{exp-3} we first show that
\begin{lemma}\label{exp-3-small}
For any $L=L(g)$ with $\lim \limits_{g\to \infty}L(g)=\infty$, we have
\[\sum_{i=1}^{\left[\frac{g}{2}\right]}\E \left[N_{sep,i}^s(X_g,L) \right]\prec \frac{\Li(e^L)}{g}.\]
\end{lemma}
\begin{proof}
It follows by Theorem \ref{thm Mirz int formula} and \ref{thm Vgn(x) small x} that
\begin{equation}\nonumber
\begin{aligned}
&\sum_{i=1}^{\left[\frac{g}{2}\right]}\E \left[N_{sep,i}^s(X_g,L) \right]\leq\frac{1}{V_g}\sum_{i=1}^{\left[\frac{g}{2}\right]}\int_0^L V_{g-i,1}(x)V_{i,1}(x)xdx\\
&\leq \left(\sum_{i=1}^{\left[\frac{g}{2}\right]}\frac{V_{g-i,1}V_{i,1}}{V_g}\right)\cdot \int_0^L \left( \sinh\left(\frac{x}{2}\right)\right)^2 \cdot \frac{4}{x} dx\\
&\asymp  \left(\sum_{i=1}^{\left[\frac{g}{2}\right]}\frac{V_{g-i,1}V_{i,1}}{V_g}\right)\cdot \Li(e^L)
\end{aligned}
\end{equation}
where we apply $\lim \limits_{g\to \infty}L(g)=\infty$ in the last equation.
By Theorem \ref{thm sum V*V} we also know that
\[\sum_{i=1}^{\left[\frac{g}{2}\right]}\frac{V_{g-i,1}V_{i,1}}{V_g}\asymp \frac{1}{g}.\] Thus, we get
\begin{equation}
\nonumber
\begin{aligned}
\sum_{i=1}^{\left[\frac{g}{2}\right]}\E \left[N_{sep,i}^s(X_g,L) \right]\prec \frac{\Li(e^L)}{g} 
\end{aligned}
\end{equation}
as desired.
\end{proof}

\noindent\textbf{Expected value on non-simple closed geodesics.} To control the number  $N^{ns}(X_g,L)$ of non-simple closed geodesics, we use our previous work \cite[Section 7]{WX22-GAFA}. Since the proofs here are almost identical to the ones in \cite[Section 7]{WX22-GAFA}, we only outline them. One may see \cite[Section 7]{WX22-GAFA} for more details. More precisely, we show
\begin{theorem}\label{b-ns}
	Assume that $L=L(g)\leq 12\ln g$ and $\lim\limits_{g\to\infty} L(g)=\infty$. Then for any $\eps>0$, as $g\to \infty$ we have
	\[\E \left[N^{ns}(X_g,L)\right] \prec_\eps  \frac{e^{(1+\eps)L}}{g}.\]
\end{theorem}

We use $X=X_g\in \sM_g$ for simplicity. Consider $\Gamma$ to be either a non-simple closed geodesic $\gamma\subset X$ or a pair of distinct closed geodesics $(\gamma_1,\gamma_2)$ in $X$ satisfying $\gamma_1\cap\gamma_2\neq \emptyset$. For the later case, denote the length of $\Gamma$ to be the total length $\ell_\Gamma(X):=\ell_{\gamma_1}(X)+\ell_{\gamma_2}(X)$.

As in \cite[Section 7]{WX22-GAFA}, for any $T=T(g)>0$ and $\Gamma$ as above of length $\ell_\Gamma(X)\leq T$, let $X(\Gamma)\subset X$ be the connected subsurface of geodesic boundary (we warn here that two distinct simple closed geodesics on the boundary of $X(\Gamma)$ may correspond to a single simple closed geodesic in $X$) such that
\ben
\item $\Gamma\subset X(\Gamma)$ is filling;
\item $\ell(\partial X(\Gamma))\leq 2\ell_\Gamma(X)\leq 2T$;
\item $\area(X(\Gamma))\leq 4 \ell_\Gamma(X) \leq 4T$.
\een

\begin{def*} For $X\in \sM_g$ and $0<T=T(g)\leq\frac{\area(X)}{8}$, we define
$$
\rm{Sub_T(X)}:=\left\{ \parbox[1]{8.5cm}{$Y\subset X$ is a connected subsurface of geodesic boundary such that $\ell(\partial Y)\leq 2T$ and $\area(Y)\leq 4T$
}\right\}
$$ 
where we allow two distinct simple closed geodesics on the boundary of $Y$ to be a single simple closed geodesic in $X$.
\end{def*}
As in \cite[Section 7]{WX22-GAFA}, let $\mathcal{P}^{ns}(X)$ be the set of oriented primitive non-simple closed geodesics of $X$. And let $\mathcal{P}^{ns,2}(X)$ be the set of pairs of oriented primitive closed geodesics $(\gamma_1,\gamma_2)$ with $\gamma_1\cap\gamma_2\neq\emptyset$ on $X$. Consider the composition $\mathcal{F}$ of the following two maps
\[\Gamma \mapsto X(\Gamma) \mapsto \partial X(\Gamma),\]
then we have that for  $\gamma \in \mathcal{P}^{ns}(X)$ of length $\ell_\gamma(X)\leq T$,
\be \label{up-mul}
\#\left\{\gamma' \in \mathcal{P}^{ns}(X); \ 
\begin{aligned}
&\ell_{\gamma'}(X)\leq T \\
&\textit{and} \ \mathcal{F}(\gamma')=\mathcal{F}(\gamma)
\end{aligned} \right\}\leq 2N_1^{\text{fill}}\left(X(\gamma), T\right).
\ene
And for $\Gamma \in \mathcal{P}^{ns,2}(X)$ of length $\ell_\Gamma(X)\leq T$,
\be \label{up-mul,pair}
\#\left\{\Gamma' \in \mathcal{P}^{ns,2}(X); \ 
\begin{aligned}
&\ell_{\Gamma'}(X)\leq T \\
&\textit{and} \ \mathcal{F}(\Gamma')=\mathcal{F}(\Gamma)
\end{aligned} \right\}\leq 4N_2^{\text{fill}}\left(X(\Gamma),T\right).
\ene
Here $N_1^{\text{fill}}(X(\gamma), T)$ and $N_2^{\text{fill}}(X(\Gamma),T)$ are the number of filling (unoriented) closed geodesics in $X(\gamma)$ and 2-tuples in $X(\Gamma)$, as defined in Definition \ref{def fill k-tuple}.

Denote
$$N^{ns,2}(X,T):=\frac{1}{4}\#\left\{\Gamma\in\mathcal{P}^{ns,2}(X);\ \ell_\Gamma(X)\leq T\right\}$$ 
to be the number of pairs of intersected unoriented closed geodesics  $(\gamma_1,\gamma_2)$ of $X$ of total length $\ell_{\gamma_1}(X)+\ell_{\gamma_2}(X)\leq T$. Similar to \cite[Proposition 30]{WX22-GAFA}, we first show that
\begin{proposition}\label{V-upp}
For any $\eps_1>0$ and $M\in\Z_{\geq 1}$, there exists a constant $c(\eps_1,M)>0$ only depending on $\eps_1$ and $M$ such that as $g\to \infty$,
\be \label{eq-v-upp}
\begin{aligned}
\max\left\{\begin{aligned} N^{ns}(X,T), \\ N^{ns,2}(X,T)
\end{aligned}\right\}
\prec T^3 e^T\left(\sum_{\substack{Y\in \rm{Sub_T(X)};\\ |\chi(Y)|\geq M+1}} \textbf{1}_{[0,2T]}(\ell(\partial Y))\right) \\
+c(\eps_1,M) \left(\sum_{\substack{Y\in \rm{Sub_T(X)};\\ 1\leq |\chi(Y)|\leq M}} T e^{T-\frac{1-\eps_1}{2}\ell(\partial Y)} \mathrm{\textbf{1}_{[0,2T]}}(\ell(\partial Y))\right). 
\end{aligned}
\ene
\end{proposition}
\bp
By the collar Lemma, either a non-simple closed geodesic or a pair of intersecting closed geodesics has length $>\arcsinh 1$. So it suffices to consider for large $T$. Rewrite
\be \label{V-e-0-1}
\begin{aligned}
&&N^{ns}(X,T)=\frac{1}{2}\cdot \sum_{\gamma \in \mathcal{P}^{ns}(X)} \textbf{1}_{[0,T]}(\ell_{\gamma}(X))  \\
&&= \frac{1}{2}\times  \sum_{\substack{Y\in \rm{Sub_T(X)};\\ |\chi(Y)|\geq M+1}} \ \sum_{\gamma\subset Y \ \textit{is filling}}  \textbf{1}_{[0,T]}(\ell_{\gamma}(X)) \\
&&+\frac{1}{2}\times\sum_{\substack{Y\in \rm{Sub_T(X)};\\ 1\leq|\chi(Y)|\leq M}}\ \sum_{\gamma\subset Y \ \textit{is filling}} \textbf{1}_{[0,T]}(\ell_{\gamma}(X)). 
\end{aligned}
\ene
Now we consider the first term in the \rm{RHS} of \eqref{V-e-0-1}. Since $\gamma \subset Y$ is filling, 
\be \label{in-to-b}
\frac{\ell(\partial Y)}{2}\leq \ell_\gamma(Y)=\ell_{\gamma}(X).
\ene 
For $Y\in \rm{Sub_T(X)}$ we know that $\area(Y)\leq 4 T$. So by \cite[Lemma 10]{WX22-GAFA} we have 
\be\label{c-g-b}
N_1^{\text{fill}}\left(Y, T\right)\prec Te^T.
\ene
By \eqref{in-to-b} and \eqref{c-g-b} we have    
\be \label{V-e-0-2}
\begin{aligned}
&&\sum_{\substack{Y\in \rm{Sub_T(X)};\\ |\chi(Y)|\geq M+1}}  \sum_{\gamma\subset Y \ \textit{is filling}} \textbf{1}_{[0,T]}(\ell_{\gamma}(X))\\
&& \leq \sum_{\substack{Y\in \rm{Sub_T(X)};\\ |\chi(Y)|\geq M+1}} N_1^{\text{fill}}\left(Y, T\right) \textbf{1}_{[0,2T]}(\ell(\partial Y)) \\
&&\prec Te^T\left(\sum_{\substack{Y\in \rm{Sub_T(X)};\\ |\chi(Y)|\geq M+1}}\textbf{1}_{[0,2T]}(\ell(\partial Y))\right).
\end{aligned}
\ene

\noindent Now we consider the second term in the \rm{RHS} of \eqref{V-e-0-1}. Let $Y\in \rm{Sub_T(X)}$ with $m=|\chi(Y)|\in[1,M]$. For any $\eps_1>0$, let $c'(\eps_1,M)=\max\limits_{1\leq m \leq M} c(\eps_1,m)>0$ where $c(\eps_1,m)>0$ is the constant in Theorem \ref{sec-count}. Then by Theorem \ref{sec-count}, we have that for large enough $g$,
\bear\label{f-u-s-1}
\sum_{\gamma\subset Y \ \textit{is filling}}  \textbf{1}_{[0,T]}(\ell_{\gamma}(X)) \prec c'(\eps_1,M)  e^{T-\frac{1-\eps_1}{2}\ell(\partial Y)} \textbf{1}_{[0,2T]}(\ell(\partial Y))
\eear
which implies that
\be \label{V-e-0-3}
\begin{aligned}
&\sum_{\substack{Y\in \rm{Sub_T(X)};\\ 1\leq|\chi(Y)|\leq M}}\ \sum_{\gamma\subset Y \ \textit{is filling}}  \textbf{1}_{[0,T]}(\ell_{\gamma}(X))\\
&\prec c'(\eps_1,M) \left(\sum_{\substack{Y\in \rm{Sub_T(X)};\\ 1\leq |\chi(Y)|\leq M}} e^{T-\frac{1-\eps_1}{2}\ell(\partial Y)} \textbf{1}_{[0,2T]}(\ell(\partial Y))\right).
\end{aligned}
\ene

Then the conclusion for $N^{ns}(X,T)$ follows by \eqref{V-e-0-1}, \eqref{V-e-0-2} and \eqref{V-e-0-3}. \\

For $N^{ns,2}(X,T)$, the proof is identical as above through replacing \eqref{c-g-b} by 
$$N_2^{\text{fill}}(Y,T)\prec T^3 e^T$$
obtained from Lemma \ref{thm count k-tuple} in Appendix \ref{appendix}, and replacing \eqref{f-u-s-1} by
\begin{equation*}
\sum_{\Gamma\subset Y \ \textit{is filling}}  \textbf{1}_{[0,T]}(\ell_{\Gamma}(X)) \prec c'(\eps_1,M) T e^{T-\frac{1-\eps_1}{2}\ell(\partial Y)} \textbf{1}_{[0,2T]}(\ell(\partial Y))
\end{equation*}
obtained from Theorem \ref{thm count fill k-tuple}.
\ep

Now we follow \cite{WX22-GAFA} to estimate the expected values $\E \left[N^{ns}(X,T)\right]$ and $\E \left[N^{ns,2}(X,T)\right]$ through taking an integral of \eqref{eq-v-upp} in Proposition \ref{V-upp} over $\sM_g$. The proofs are identical to the ones in \cite[Section 7]{WX22-GAFA}. So we only point out the necessary modifications. First we recall the Assumption ($\star$) in \cite[Section 7.2]{WX22-GAFA}:
\begin{assum*} \label[$\star$]{star} let $Y_0\in \rm{Sub_T(X)}$ satisfying
	\begin{itemize}
           \item $Y_0$ is homeomorphic to $S_{g_0,k}$ for some $g_0\geq 0$ and $k>0$ with $m=|\chi(Y_0)|= 2g_0-2+k\geq 1$;
		\item the boundary $\partial Y_0$ is a simple closed multi-geodesics in $X$ consisting of $k$ simple closed geodesics which has $n_0$ pairs of simple closed geodesics for some $n_0\geq 0$ such that each pair corresponds to a single simple closed geodesic in $X$;
           \item the interior of its complement $X\setminus Y_0$ consists of $q$ components $S_{g_1,n_1},\cdots,$ $S_{g_q,n_q}$ for some $q\geq 1$ where $\sum_{i=1}^q n_i=k-2n_0$.
	\end{itemize}
\end{assum*}
First we study a single orbit $\Mod_g \cdot Y_0 \subset \rm{Sub_T(X)}$.
\begin{proposition}\label{upp-orbit}
Let $Y_0\in \rm{Sub_T(X)}$ satisfying Assumption ($\star$). Then we have that as $g\to \infty$,
\beqar
\int_{\sM_g}\sum_{Y\in \Mod_g \cdot Y_0}\textbf{1}_{[0,2T]}(\ell(\partial Y))dX \prec  e^{4T} V_{g_0,k} \frac{ V_{g_1,n_1}\cdots V_{g_q,n_q}}{n_0!n_1!\cdots n_q!}.
\eeqar
\end{proposition}
\bp
The proof is identical to the proof of \cite[Proposition 31]{WX22-GAFA} through replacing $e^{-\frac{\ell(\partial Y)}{4}}$ by $1$.
\ep

As in \cite{NWX20, WX22-GAFA}, for $r\geq1$ we define 
$$
W_{r}:=
\begin{cases}
V_{\frac{r}{2}+1}&\text{if $r$ is even},\\[5pt]
V_{\frac{r+1}{2},1}&\text{if $r$ is odd}.
\end{cases}
$$
\begin{proposition}\label{upp-orbit-gk}
Let $g_0\geq 0$ and $k\geq 1$ be fixed with $m=2g_0-2+k\geq 1$. Then we have that as $g\to \infty$,
\beqar 
\int_{\sM_g}\sum\limits_{\substack{Y\in \rm{Sub_T(X)}; \\ Y\cong S_{g_0,k}}}\textbf{1}_{[0,2T]}(\ell(\partial Y)) dX \prec  e^{4T}W_{2g_0+k-2} W_{2g-2g_0-k}.  \nonumber
\eeqar
\end{proposition}
\bp
The proof is identical to the one of \cite[Proposition 32]{WX22-GAFA} through replacing $e^{-\frac{\ell(\partial Y)}{4}}$ by $1$ and applying Proposition \ref{upp-orbit}.
\ep

For any $Y\in \rm{Sub_T(X)}$ it is known that $|\chi(Y)|\leq \frac{4T}{2\pi}$. Now we bound the integral of $\sum \limits_{Y\in \rm{Sub_T(X)}} \textbf{1}_{[0,2T]}(\ell(\partial Y))$
over $\sM_g$ when $m\leq|\chi(Y)|\leq [\frac{4T}{2\pi}]$ where $m\geq 1$ is a fixed integer. That is, we allow $g_0$ and $k$ in Assumption ($\star$) to vary such that $m\leq 2g_0+k-2\leq [\frac{4T}{2\pi}]$.
\begin{proposition}\label{upp-orbit-m}
Let $m\geq 1$ be any fixed integer. Then we have that there exists a constant $c(m)>0$ only depending on $m$ such that as $g\to \infty$,
\beqar 
\int_{\sM_g}\sum\limits_{\substack{Y\in \rm{Sub_T(X)}; \\ m\leq|\chi(Y)|\leq [\frac{4T}{2\pi}]}}\textbf{1}_{[0,2T]}(\ell(\partial Y))  dX  \prec  Te^{4T} c(m) \frac{V_g}{g^m}.  \nonumber
\eeqar
\end{proposition}
\bp
The proof is identical to the proof of \cite[Proposition 33]{WX22-GAFA} through replacing $e^{-\frac{\ell(\partial Y)}{4}}$ by $1$ and applying Proposition \ref{upp-orbit-gk}.
\ep

For the expected value of the \rm{RHS} in \eqref{eq-v-upp}, Proposition \ref{upp-orbit-m} implies that if $T\leq C\ln g$ for some constant $C>0$ and $m$ large enough, then we have
\[\frac{T^3 e^T}{V_g}\int_{\sM_g}\sum\limits_{\substack{Y\in \rm{Sub_T(X)}; \\ m\leq|\chi(Y)|\leq [\frac{4T}{2\pi}]}} \textbf{1}_{[0,2T]}(\ell(\partial Y))  dX\prec_{C,m} \frac{e^T}{g}.\]
Next we aim to show the expected value of the second term of the \rm{RHS} of \eqref{eq-v-upp} in Proposition \ref{V-upp} also has the similar property for bounded $m$.

In the sequel, we always assume that $Y_0 \cong S_{g_0,k}\in \rm{Sub_T(X)}$ satisfies Assumption ($\star$) with an additional assumption that
\[1\leq 2g_0+k-2\leq M\]
for a constant $M\in\Z_{\geq 1}$. Then $1\leq k\leq M+2$ and $0\leq g_0\leq \frac{1}{2}(M+1)$.
\begin{proposition}\label{16-small-1}
Let $g_0\geq 0$ and $k\geq 1$ be fixed with $m=2g_0-2+k$. For any $\eps_1>0$, then we have that as $g\to \infty$,
\[\int_{\sM_g}\sum\limits_{\substack{Y\in \rm{Sub_T(X)}; \\ Y\cong S_{g_0,k}}} e^{T-\frac{1-\eps_1}{2}\ell(\partial Y)} \textbf{1}_{[0,2T]}(\ell(\partial Y)) dX  \prec_m T^{4m+2}e^{(1+\eps_1)T}\frac{V_g}{g^{m}}. \]
\end{proposition}
\bp
The proof is identical to the one of \cite[Proposition 34]{WX22-GAFA} in which the power $66$ comes from the bound:
$$6g_0-6+4k=3m+k\leq 66$$ 
which is above \cite[Equation (46)]{WX22-GAFA}; in this current case, we have
\[6g_0-6+4k=3m+k\leq 4m+2\]
as desired.
\ep

Since we consider $|\chi(Y)|=m=2g_0+k-2\in [1,M]$, both $g_0$ and $k$ are bounded. As a direct consequence of Proposition \ref{16-small-1} we have
\begin{proposition}\label{16-small-2}
For any $\eps_1>0$ and fixed $M\in\Z_{\geq 1}$, then we have that as $g\to \infty$,
\[\int_{\sM_g}\sum_{\substack{Y\in \rm{Sub_T(X)};\\ 1\leq |\chi(Y)|\leq M}} T e^{T-\frac{1-\eps_1}{2}\ell(\partial Y)} \textbf{1}_{[0,2T]}(\ell(\partial Y) dX  \prec_M T^{4M+3}e^{(1+\eps_1)T}\frac{V_g}{g}. \]
\end{proposition}

Now we are ready to prove the following effective bound for the expected value $\E \left[N^{ns}(X_g,L)\right]$ and $\E \left[N^{ns,2}(X_g,L)\right]$.
\begin{theorem}\label{b-ns,ns2}
Assume that $L=L(g)\leq C\ln g$ for some constant $C>0$ and $\lim\limits_{g\to\infty} L(g)=\infty$. Then for any $\eps>0$, as $g\to \infty$ we have
\[\E \left[N^{ns}(X_g,L)\right] \prec_{C,\eps} \frac{e^{(1+\eps)L}}{g}\]
and
\[\E \left[N^{ns,2}(X_g,L)\right] \prec_{C,\eps} \frac{e^{(1+\eps)L}}{g}.\]
\end{theorem}

\bp
The conclusion clearly follows by Proposition \ref{V-upp}, Proposition \ref{upp-orbit-m} for any fixed $m\geq 100(1+C)$ and Proposition \ref{16-small-2} for $M=m-1$ and $\eps_1=\frac{\eps}{2}$ where we assume $g$ is large enough such that $L^{3+400(1+C)}\prec e^{\frac{\eps}{2} L}$.
\ep

\bp[Proof of Theorem \ref{b-ns}]
The conclusion clearly follows from Theorem \ref{b-ns,ns2} by letting $C=12$.
\ep

\noindent \textbf{Expected value on simple and non-separating closed geodesics.}
We have shown in Lemma \ref{exp-1} that 
$$\E[N_{nsep}^{s}(X_g,L)]=\frac{1}{2}\Li(e^L) \cdot\left(1+O\left(\frac{L^2}{g}+\frac{L\ln L}{e^L}\right)\right)$$
when $\limg L(g)=\infty$ and $\limg\frac{L(g)^2}{g}=0$. Now we aim to show that
\begin{theorem}\label{thm prob nsep}
	Assume $L=L(g)$ satisfies 
	$$\limg L(g)=\infty\ \ \text{and}\ \ \limsup_{g\to\infty} \frac{L(g)}{\ln g}<\infty.$$
	Then for any fixed $\eps>0$ we have
	$$\limg\Prob\left(X_g\in\M_g;\ N_{nsep}^{s}(X_g,L)>\frac{1-\eps}{2}\Li(e^L)\right)=1.$$
\end{theorem}

Similar to the methods in \cite{MP19, NWX20}, we compute the variance of $N_{nsep}^{s}(X_g,L)$ by using the Chebyshev inequality. For any fixed $\eps>0$ we have 
\begin{equation}\label{equ Chebyshev}
\begin{aligned}
	&\Prob\left(X_g\in\M_g;\ \left|1-\frac{N_{nsep}^{s}(X_g,L)}{\E[N_{nsep}^{s}(X_g,L)]}\right| \geq \frac{1}{2}\eps\right) \\
	&\leq \frac{\text{\rm{Var}}[N_{nsep}^{s}(X_g,L)]}{\frac{\eps^2}{4} \E[N_{nsep}^{s}(X_g,L)]^2} \\
	&= \frac{4}{\eps^2} \frac{\E[N_{nsep}^{s}(X_g,L)^2] - \E[N_{nsep}^{s}(X_g,L)]^2}{\E[N_{nsep}^{s}(X_g,L)]^2}
\end{aligned}
\end{equation}
where $\text{\rm{Var}}[\cdot]=\E[\cdot^2]-\E[\cdot]^2$ is the variance of a random variable. 

We decompose $N_{nsep}^{s}(X_g,L)^2$ into four different parts as follows.
\begin{definition}
	Let $\sP'(X_g)$ to be the set of all unoriented primitive closed geodesics on $X_g$. For $1\leq k\leq [\frac{g-1}{2}]$, define
	$$Y_1(X_g,L):=\#\left\{(\alpha,\beta)\in\sP'(X_g)^2;\quad 
	\begin{aligned}
		&\alpha,\beta\ \text{are simple and non-separating,} \\
		&\ell_\alpha(X_g)\leq L,\ \ell_\beta(X_g)\leq L, \\
		&\alpha\cap\beta=\emptyset,\ \alpha\neq\beta, \\
		&X_g\setminus(\alpha\cup\beta) \cong S_{g-2,4}
	\end{aligned} \right\},$$
	$$Y_2^k(X_g,L):=\#\left\{(\alpha,\beta)\in\sP'(X_g)^2;\quad 
	\begin{aligned}
		&\alpha,\beta\ \text{are simple and non-separating,} \\
		&\ell_\alpha(X_g)\leq L,\ \ell_\beta(X_g)\leq L, \\
		&\alpha\cap\beta=\emptyset,\ \alpha\neq\beta, \\
		&X_g\setminus(\alpha\cup\beta) \cong S_{k,2}\cup S_{g-k-1,2}
	\end{aligned} \right\},$$
	$$Z(X_g,L):=\#\left\{(\alpha,\beta)\in\sP'(X_g)^2;\quad 
	\begin{aligned}
		&\alpha,\beta\ \text{are simple and non-separating,} \\
		&\ell_\alpha(X_g)\leq L,\ \ell_\beta(X_g)\leq L, \\
		&\alpha\cap\beta\neq\emptyset,\ \alpha\neq\beta 
	\end{aligned} \right\}.$$
\end{definition}

Then
\begin{equation}\label{equ nsep N^2}
	N_{nsep}^{s}(X_g,L)^2 = N_{nsep}^{s}(X_g,L) + Y_1(X_g,L) + \sum_{k=1}^{[\frac{g-1}{2}]}Y_2^k(X_g,L) + Z(X_g,L).
\end{equation}
Next we estimate the expectation of each term of the $\mathrm{RHS}$ of \eqref{equ nsep N^2} above.

\begin{lemma}\label{thm E[Y_1]}
	Assume $L=L(g)$ satisfies $\lim\limits_{g\to\infty} L(g)=\infty$ and $\lim\limits_{g\to\infty}\frac{L(g)^2}{g}=0$. Then
	$$\left|\E[Y_1(X_g,L)]-\E[N_{nsep}^s(X_g,L)]^2\right| \prec \frac{1}{g} e^{2L}.$$
\end{lemma}
\begin{proof}
	By Mirzakhani's integration formula, i.e., Theorem \ref{thm Mirz int formula}, we have
	$$\E[Y_1(X_g,L)]=\frac{1}{V_g} \frac{1}{4} \int_0^L\int_0^L V_{g-2,4}(x,x,y,y)xydxdy.$$
	Since $L=L(g)=o(\sqrt{g})$, by Theorem \ref{thm Vgn(x) small x} we have that for all $x,y\in [0,L]$,
	$$\frac{V_{g-2,4}(x,x,y,y)}{V_{g-2,4}} =  \left(\frac{\sinh\left(\frac{x}{2}\right)}{\frac{x}{2}} \right)^2 \left(\frac{\sinh\left(\frac{y}{2}\right)}{\frac{y}{2}} \right)^2 \left(1+O\left(\frac{x^2+y^2}{g} \right)\right).$$
	By Theorem \ref{thm Vgn/Vgn+1},
	$$\frac{V_{g-2,4}}{V_{g}}=\frac{V_{g-2,4}}{V_{g-1,2}}\frac{V_{g-1,2}}{V_{g}}=1+O\left(\frac{1}{g}\right).$$
	So
	\begin{eqnarray*}
		\E[Y_1(X_g,L)]&=& \frac{1}{4} \int_0^L\int_0^L \left(\sinh\left(\frac{x}{2}\right)\sinh\left(\frac{y}{2}\right)\right)^2 \frac{4}{xy} \\
		&&\left(1+O\left(\frac{1+x^2+y^2}{g} \right)\right)dxdy \\
		&=& \left(\int_0^L \left(\sinh\left(\frac{x}{2}\right)\right)^2 \frac{1}{x} dx \right)^2 \left(1+O\left(\frac{L^2}{g} \right)\right).
	\end{eqnarray*}
	Similarly, 
	\begin{eqnarray*}
		\E[N_{nsep}^s(X_g,L)]&=&\frac{1}{V_g}\frac{1}{2}\int_0^L V_{g-1,2}(x,x)xdx \\
		&=& \int_0^L \left(\sinh\left(\frac{x}{2}\right)\right)^2 \frac{1}{x} dx \cdot \left(1+O\left(\frac{L^2}{g} \right)\right).
	\end{eqnarray*}
	Thus, we obtain
	\begin{eqnarray*}
		&&\left|\E[Y_1(X_g,L)]-\E[N_{nsep}^s(X_g,L)]^2\right| \\
		&=&\left(\int_0^L \left(\sinh\left(\frac{x}{2}\right)\right)^2 \frac{1}{x} dx \right)^2 \cdot O\left(\frac{L^2}{g}\right) \\
		&\prec& \frac{1}{g} e^{2L}
	\end{eqnarray*}
as desired.
\end{proof}

\begin{lemma}\label{thm E[Y_2^k]}
	Assume $L=L(g)$ satisfies $\lim\limits_{g\to\infty} L(g)=\infty$. Then 
	$$\sum_{k=1}^{[\frac{g-1}{2}]} \E[Y_2^k(X_g,L)] \prec \frac{1}{g^2} \frac{e^{2L}}{L^2}.$$
\end{lemma}
\begin{proof}
	By Mirzakhani's integration formula, i.e., Theorem \ref{thm Mirz int formula}, we have
	$$\E[Y_2^k(X_g,L)]\leq\frac{1}{V_g} \int_0^L\int_0^L V_{k,2}(x,y)V_{g-k-1,2}(x,y)xydxdy.$$
	Then by Theorem \ref{thm Vgn(x) small x}, 
	\begin{eqnarray*}
		\E[Y_2^k(X_g,L)]
		&\leq&\frac{V_{k,2}V_{g-k-1,2}}{V_g} \int_0^L\int_0^L \left(\frac{\sinh\left(\frac{x}{2}\right)}{\frac{x}{2}} \frac{\sinh\left(\frac{y}{2}\right)}{\frac{y}{2}}\right)^2 xydxdy \\
		&\prec& \frac{V_{k,2}V_{g-k-1,2}}{V_g} \cdot \frac{e^{2L}}{L^2}.
	\end{eqnarray*}
	Recall that Theorem \ref{thm sum V*V} says that 
	$$\sum_{k=1}^{[\frac{g-1}{2}]} \frac{V_{k,2}V_{g-k-1,2}}{V_g} \asymp \frac{1}{g^2}.$$
	Therefore, the proof is finished by taking a summation over $k$.
\end{proof}

\begin{lemma}\label{thm E[Z]}
	Assume $L=L(g)$ satisfies $\lim\limits_{g\to\infty} L(g)=\infty$ and $L(g)\leq C\ln g$ for some constant $C>0$. Then for any $\eps>0$, 
	$$\E[Z(X_g,L)] \prec_{C,\eps} \frac{1}{g} e^{(2+\eps)L}.$$
\end{lemma}
\begin{proof}
The conclusion clearly follows by Theorem \ref{b-ns,ns2}. Here we recall that $N^{ns,2}(X,T)$ is the number of pairs of intersected unoriented closed geodesics pair $(\gamma_1,\gamma_2)$ of $X$ of total length $\ell_{\gamma_1}(X)+\ell_{\gamma_2}(X)\leq T$, and hence $Z(X_g,L)\leq N^{ns,2}(X_g,2L)$.
\end{proof}

Now we prove Theorem \ref{thm prob nsep}.
\begin{proof}[Proof of Theorem \ref{thm prob nsep}]
	Now assume 
	$$\limg L(g)=\infty\ \ \text{and}\ \ \limsup_{g\to\infty} \frac{L(g)}{\ln g}\leq C-1<\infty$$
	for some constant $C>1$. So $L\leq C\ln g$ for large enough $g$.
	
	Then it follows by Lemma \ref{exp-1}, \ref{thm E[Y_1]}, \ref{thm E[Y_2^k]}, \ref{thm E[Z]} and Equation \eqref{equ nsep N^2} that for any $\eps_1>0$,
	\begin{eqnarray*}
		&&\frac{\E[N_{nsep}^{s}(X_g,L)^2] - \E[N_{nsep}^{s}(X_g,L)]^2}{\E[N_{nsep}^{s}(X_g,L)]^2}\\
		&\prec_{\eps_1}& \frac{\Li(e^L) +\frac{1}{g}e^{2L} +\frac{1}{g^2}\frac{e^{2L}}{L^2} +\frac{1}{g}e^{(2+\eps_1)L}}{\Li(e^L)^2} \\
		&\prec& L e^{-L} +\frac{1}{g} L^2 e^{\eps_1 L} \\
		&\prec_C& L e^{-L} +(\ln g)^2 g^{\eps_1 C -1}.
	\end{eqnarray*}
	
	Now take $\eps_1=\frac{1}{2C}$. For any fixed $\eps>0$, by \eqref{equ Chebyshev} as $g\to\infty$ we have
	\begin{eqnarray*}
		&&\Prob\left(X_g\in\M_g;\ N_{nsep}^{s}(X_g,L) \leq (1-\tfrac{1}{2}\eps)\E[N_{nsep}^{s}(X_g,L)] \right) \\
		&\leq&\Prob\left(X_g\in\M_g;\ \left|1-\frac{N_{nsep}^{s}(X_g,L)}{\E[N_{nsep}^{s}(X_g,L)]}\right| \geq \frac{1}{2}\eps\right) \\
		&\leq& \frac{4}{\eps^2} \frac{\E[N_{nsep}^{s}(X_g,L)^2] - \E[N_{nsep}^{s}(X_g,L)]^2}{\E[N_{nsep}^{s}(X_g,L)]^2} \\
		&\prec_C& \frac{1}{\eps^2} \left(L e^{-L} +(\ln g)^2 g^{-\frac{1}{2}}\right).
	\end{eqnarray*}
	By Lemma \ref{exp-1}, we know that for $g$ large enough,
 $$\E[N_{nsep}^{s}(X_g,L)]\geq \frac{1-\eps/2}{2}\Li(e^L).$$ Thus, we have that for large enough $g$,
 	\begin{eqnarray*}
		&&\Prob\left(X_g\in\M_g;\ N_{nsep}^{s}(X_g,L)> \frac{(1-\eps/2)^2}{2}\Li(e^L)\right) \\
		&\geq&1-\Prob\left(X_g\in\M_g;\ N_{nsep}^{s}(X_g,L) \leq (1-\tfrac{1}{2}\eps)\E[N_{nsep}^{s}(X_g,L)] \right)\\
		&=&1-O_C\left( \frac{1}{\eps^2} \left(L e^{-L} +(\ln g)^2 g^{-\frac{1}{2}}\right)\right).
	\end{eqnarray*}
Then the conclusion follows by letting $g\to \infty$ because $\frac{(1-\eps/2)^2}{2}\Li(e^L)>\frac{1-\eps}{2}\Li(e^L).$
\end{proof}

Now we are ready to finish the proof of Theorem \ref{conj-simple-2}.
\begin{proof}[Proof of Theorem \ref{conj-simple-2}]
Assume that 
$$\limg L(g)=\infty\ \ \text{and}\ \ \limsup_{g\to\infty}\frac{L(g)}{\ln g}\leq C-1<\infty$$
for some constant $C>1$. So $L\leq C\ln g$ for large enough $g$.

By \eqref{equ N=s+ns}, Lemma \ref{exp-3-small} and Theorem \ref{b-ns,ns2}, for any $\eps_1>0$, when $g$ is large enough we have
\begin{eqnarray*}
&&\E\left[\left|N(X_g,L)-N_{nsep}^s(X_g,L)\right|\right] \\
&=& \E\left[\sum_{k=1}^{[g/2]} N_{sep,k}^s(X_g,L) +N^{ns}(X_g,L)\right] \\
&\prec_{C,\eps_1}& \frac{1}{g}\Li(e^L) + \frac{1}{g}e^{(1+\eps_1)L} \\
&\prec& \frac{1}{g}e^{(1+\eps_1)L}.
\end{eqnarray*}
Then by Markov's inequality we have
$$\Prob\left(X_g\in\M_g;\ \left|N(X_g,L)-N_{nsep}^s(X_g,L)\right|\geq \frac{e^{(1+2\eps_1)L}}{g} \right) \prec_{C, \eps_1} e^{-\eps_1 L}$$
which converges to $0$ as $g\to \infty$. 

On the other hand, by Theorem \ref{thm prob nsep} we have 
$$\limg\Prob\left(X_g\in\M_g;\ N_{nsep}^{s}(X_g,L)>\frac{1}{4}\Li(e^L)\right)=1.$$
Thus, 
$$\limg\Prob\left(X_g\in\M_g;\ \left|1-\frac{N_{nsep}^s(X_g,L)}{N(X_g,L)}\right|< \frac{4 e^{(1+2\eps_1)L}}{g\Li(e^L)} \right) =1.$$

Then the proof is finished by taking suitable $\eps_1=\eps_1(C, \epsilon)>0$.
\end{proof}

\bp[Proof of Theorem \ref{mt-geodesic}]
It clearly follows from Theorem \ref{conj-nsimple}, \ref{conj-simple} and \ref{conj-simple-2}.
\ep


\appendix
\begin{appendices}
	
\section{Counting filling $k$-tuples}\label{appendix}
In this Appendix, we give a brief sketch of the proof of Theorem \ref{thm count fill k-tuple} here. One may see \cite[Section 8]{WX22-GAFA} for more details.

First, as a simple application of Corollary \ref{thm count ge^L upp} we have
\begin{lemma}\label{thm count k-tuple}
	For fixed $k\in\Z_{\geq1}$, there exists a constant $c_k>0$ only depending on $k$ such that for any hyperbolic surface $X\in\T_{g,n}(L_1,\cdots,L_n)$ and $L>0$, we have 
	$$\#\left\{(\gamma_1,\cdots,\gamma_k)\in\sP(X)^k;\ \sum_{i=1}^k \ell(\gamma_i)\leq L\right\} \leq c_k\cdot(2g-2+n)^k (1+L)^{k-1}e^L$$
	where $\sP(X)$ is the set of all oriented non-peripheral primitive closed geodesics in $X$.
\end{lemma}
\begin{proof}
	By doubling $X$ along its boundaries, we get a closed hyperbolic surface $\overline{X}\in\T_{2g+n-1}$. Each non-peripheral primitive closed geodesic in $X$ has two copies in $\overline{X}$. So by Corollary \ref{thm count ge^L upp}, for any $T>0$ we have 
	\begin{eqnarray*}
		\#\{\gamma\in\sP(X);\ \ell(\gamma)\leq T\}
		&\leq& \frac{1}{2}\#\{\gamma\in\sP(\overline{X});\ \ell(\gamma)\leq T\} \\
		&\leq& (2g-2+n)e^{T+7}.
	\end{eqnarray*}
	
	\noindent Now we prove by induction on $k$. First the lemma holds for $k=1$ by the above inequality. Now assume that the lemma holds for $k-1$. Then for $k\geq 2$, we get
	\begin{eqnarray*}
		&&\#\left\{(\gamma_1,\cdots,\gamma_k)\in\sP(X)^k;\ \sum_{i=1}^k \ell(\gamma_i)\leq L\right\} \\
		&\leq& \sum_{T=1}^{[L]+1} \#\left\{\gamma_1\in\sP(X);\ T-1<\ell(\gamma_1)\leq T\right\} \\
		&&\times \#\left\{(\gamma_2,\cdots,\gamma_k)\in\sP(X)^{k-1};\ \sum_{i=1}^{k-1} \ell(\gamma_i)\leq L-T+1\right\} \\
		&\prec_k& \sum_{T=1}^{[L]+1} (2g-2+n)e^T \cdot (2g-2+n)^{k-1} (L-T+2)^{k-2} e^{L-T+1} \\
		&\prec& \sum_{T=1}^{[L]+1} (2g-2+n)^k (1+L)^{k-2} e^L \\
		&\leq& (2g-2+n)^k (1+L)^{k-1} e^L
	\end{eqnarray*} as desired.
\end{proof}

Similar to \cite[Theorem 38]{WX22-GAFA}, we study counting filling $k$-tuples instead of filling geodesics. We prove

\begin{theorem}\label{thm fill k-tuple decrease}
	There exists a universal constant $L_0>0$ such that for any $0<\eps<1$, $m=2g-2+n\geq 1$, $\Delta\geq 0$ and $\sum_{i=1}^n x_i \geq \Delta +L_0\cdot\frac{m\cdot n}{\eps}$, the following holds: for any hyperbolic surface $X\in \T_{g,n}(x_1,\cdots,x_n)$, one can always find a new hyperbolic surface $Y\in \T_{g,n}(y_1,\cdots,y_n)$ satisfying
	\begin{enumerate}
		\item $y_i\leq x_i$ for all $1\leq i\leq n$;
		\item $\sum_{i=1}^n x_i - \sum_{i=1}^n y_i = \Delta$;
		\item for any filling $k$-tuple $\Gamma=(\gamma_1,\cdots,\gamma_k)$ in $S_{g,n}$, we have
		$$\sum_{i=1}^k\ell_{\gamma_i}(X) - \sum_{i=1}^k\ell_{\gamma_i}(Y) \geq \frac{1}{2}(1-\eps)\Delta.$$ 
	\end{enumerate}
\end{theorem}

\begin{proof}
	The proof is identical to the one of \cite[Theorem 38]{WX22-GAFA} through replacing filling curves by filling $k$-tuples. We give a sketch about the proof here, and one may refer to \cite[Section 8]{WX22-GAFA} for more details. 
	
	Consider a longest boundary component of $X$, denoted by $\alpha$. Let $w$ be the width of the maximal embedded half-collar of $\alpha$ in $X$, that is, 
	$$w=\sup \left\{d>0;\ 
	\begin{aligned}
		&\{p\in X;\ \dist(p,\alpha)<d\} \\ 
		&\text{is homeomorphic to a cylinder and} \\
		&\text{does not intersect with} \ (\partial X\setminus\alpha)
	\end{aligned} \right\}.$$
	Then the closed curve $\{p\in X;\ \dist(p,\alpha)=w\}$ either touches another component of $\partial X$ or has non-empty self-intersection. For both cases, it induces a pair of pants, denoted by $P$, in which $\alpha$ is a boundary curve (see \cite[Page 382 Construction]{WX22-GAFA} for details). 
	
	For any $k\in\Z_{\geq1}$ and filling $k$-tuple $\Gamma=(\gamma_1,\cdots,\gamma_k)$ of $X$, at least one of $\gamma_i$'s must intersect with the pair of pants $P$ (actually this is the only place where we apply the filling property for $k$-tuples). Then as in \cite[Page 384]{WX22-GAFA} we reduce the length of $\partial P\cap\partial X$ and keep the part $X\setminus P$ unchanged, the length of those $\gamma_i$'s which do not intersect with $P$ are unchanged. For those $\gamma_i$'s which intersect with $P$, by \cite[Remark 41]{WX22-GAFA}, length of $\gamma_i$ reduce at least the multiplication of $\frac{1-\eps}{2}$ and the reduction of $\partial P\cap\partial X$ if $\ell(\alpha)$ is large enough.
	
	So if the initial total boundary length $\sum x_i$ is large enough, we can always find a long boundary component and apply the above process to reduce the total boundary length. The process is workable unless all boundary components are bounded by some constant related to $\eps$ and $m$. As argued above, similar as \cite[Proposition 40]{WX22-GAFA}, for each time the total length of any filling $k$-tuple must reduce at least the multiplication of $\frac{1-\eps}{2}$ and the reduction of the total boundary length. Then by using the same argument as in the proof of \cite[Theorem 38]{WX22-GAFA}, one can get a desired new hyperbolic surface $Y$ satisfying all the three required properties.
\end{proof}

Now applying Theorem \ref{thm fill k-tuple decrease}, we finish the proof of Theorem \ref{thm count fill k-tuple}.
\begin{proof}[Proof of Theorem \ref{thm count fill k-tuple}] 
	
	If $n=0$, $\ell(\partial X)=0$. Then the conclusion holds directly by Lemma \ref{thm count k-tuple}. Now assume $n\geq 1$ and $X\in \T_{g,n}(x_1,\cdots,x_n)$.
	
	\underline{Case 1}: the total boundary length $\sum_{i=1}^n x_i > L_0\frac{mn}{\eps}$. By Theorem \ref{thm fill k-tuple decrease} we may take $\Delta = \sum x_i - L_0\frac{mn}{\eps}$ and then get a hyperbolic surface $Y\in\T_{g,n}(y_1,\cdots,y_n)$ such that $\sum y_i = L_0\frac{mn}{\eps}$ and for any filling $k$-tuple $\Gamma=(\gamma_1,\cdots,\gamma_k)$ in $S_{g,n}$ we have
	$$\sum_{i=1}^k\ell_{\gamma_i}(X) - \sum_{i=1}^k\ell_{\gamma_i}(Y) \geq \frac{1-\eps}{2}\cdot \left(\sum_{i=1}^n x_i - L_0\frac{mn}{\eps}\right).$$
	Thus 
	\begin{equation*}
		N_k^{\text{fill}}\left(X,L\right) \leq N_k^{\text{fill}}\left(Y,L-\frac{1-\eps}{2}\cdot \left(\sum_{i=1}^n x_i - L_0\frac{mn}{\eps}\right)\right)
	\end{equation*}
	which together with Lemma \ref{thm count k-tuple} and the fact that $n \leq m+2$ implies that
	\begin{eqnarray*}
		N_k^{\text{fill}}\left(X,L\right) 
		&\prec_k& m^k \left(1+L-\tfrac{1-\eps}{2} (\sum x_i - L_0\tfrac{mn}{\eps})\right)^{k-1} e^{L-\frac{1-\eps}{2}(\sum x_i - L_0\frac{mn}{\eps})} \\
		&\prec_{k,\eps,m}&  (1+L)^{k-1} e^{L-\frac{1-\eps}{2}\sum x_i}.
	\end{eqnarray*}
	
	\underline{Case 2}: the total boundary length $\sum_{i=1}^n x_i \leq L_0\frac{mn}{\eps}$. Then by Lemma \ref{thm count k-tuple} directly we have 
	\begin{eqnarray*}
		N_k^{\text{fill}}\left(X,L\right) 
		&\prec_k& m^k (1+L)^{k-1} e^L \\
		&\prec_{k,\eps,m}& (1+L)^{k-1} e^{L-\frac{1-\eps}{2}\sum x_i}.
	\end{eqnarray*}

The proof is complete.
\end{proof}

\end{appendices}

\bibliographystyle{amsalpha}
\bibliography{ref}

\end{document}